\documentclass[a4paper,10pt]{amsart}
\usepackage[utf8]{inputenc}
 \usepackage{setspace}
\usepackage[cmtip,arrow]{xy}
\usepackage{lscape}
\usepackage{latexsym}
\usepackage[activeacute,english]{babel}
\usepackage{graphicx}
\usepackage{amsmath}
\usepackage{amsfonts}

\usepackage{pst-node}

\usepackage{amsthm}
\usepackage{amssymb}
\usepackage{amscd}
\usepackage{mathrsfs}
\usepackage{pb-diagram,pb-xy}
\xyoption{all}
\usepackage{graphics}
\usepackage{stmaryrd}
\usepackage{yfonts,bbm}
\usepackage{tikz-cd}

\usepackage{hyperref}
\hypersetup{colorlinks=true,linkcolor=blue,citecolor=red,linktocpage=true}
\usepackage{url}

\newtheorem{theorem}{Theorem}[section]
\newtheorem{lemma}[theorem]{Lemma}
\newtheorem{proposition}[theorem]{Proposition}
\newtheorem{corollary}[theorem]{Corollary}
\newtheorem{definition}[theorem]{Definition}
\newtheorem{example}[theorem]{Example}

\newtheorem{Remark}[theorem]{Remark}

\newcommand{\C}{\mathcal{C}}

\newcommand{\Mod}{\mathrm{Mod}}

\newcommand{\Tr}{\mathrm{Tr}}

\begin{document}
\title{Homological theory of $K$-idempotent ideals in dualizing varieties}

\author[Rodriguez, Sandoval, Santiago]{Luis Gabriel Rodr\'iguez-Vald\'es$^{1}$, \\ Martha Lizbeth Shaid Sandoval-Miranda$^{2}$,  \\ Valente Santiago-Vargas$^{3}$}

\thanks{The authors are grateful to the project PAPIIT-Universidad Nacional Aut\'onoma de M\'exico IN100520}
\dedicatory{}
\maketitle

\begin{abstract}

In this work , we develop the theory of $k$-idempotent ideals in the setting of dualizing varieties. Several results given previously in \cite{APG}  by M. Auslander, M. I. Platzeck, and G. Todorov are extended to this context. Given an ideal $\mathcal{I}$ (which is the trace of a projective module), we construct a canonical recollement which is the analog to a well-known recollement in categories of modules over artin algebras. Moreover, we study the homological properties of the categories involved in such a recollement. Consequently,  we find conditions on the ideal $\mathcal{I}$ to obtain quasi-hereditary algebras in such a recollement. Applications to bounded derived categories are also given.
\end{abstract}


\section{Introduction}\label{sec:1}
Dualizing $R$-varieties have appeared  in the theory of  locally bounded $k$-categories over a field $k$, as well as in the study of  categories of graded modules over artin algebras. What is more, they are  in connection with  covering theory. One of the advantages of the notion of dualizing $R$-variety defined in  \cite{AusVarietyI} is that it provides a common setting for the category $\mathrm{proj}(\Lambda)$ of finitely generated projective $\Lambda$-modules, $\mathrm{mod}(\Lambda)$ and $\mathrm{mod}(\mathrm{mod}(\Lambda))$, which  play an important role in the study of an artin algebra $\Lambda$.\\


Recently, in \cite{Parra}, C. Parra, M. Saorin, and S. Virili, have studied modules over a small preadditive category $\mathcal{C},$ and have shown various interesting results on torsion theories, idempotent ideals, as well as recollements, in that context. Particularly, they have generalized a result given by J. P. Jans, which states that there is a bijective correspondence between idempotent ideals in the preadditive category $\mathcal{C}$ and TTF's  in the category of modules $\mathrm{Mod}(\mathcal{C})$.\\

In \cite{APG}, M. Auslander, M. I. Platzeck and G. Todorov studied homological ideals in the case of the category $\mathrm{mod}(\Lambda)$, where $\Lambda$ is an artin algebra. They proved several fundamental results related to homological ideals, and they connected such a notion with the context of quasi-hereditary algebras.\\


Following all of these ideas, it seems natural to extend this study to the setting of rings with several objects (see  \cite{Mitchelring}). This extension is best expressed in the language of dualizing $R-$varieties. Our main porpouse in this paper is to show generalizations of results given in \cite{APG}, in the context of dualizing varieties.
 In the following, we describe our results in more detail.\\

After the introduction (section \ref{sec:1}), in section  \ref{sec:2} we consider a preadditive category $\mathcal{C}$, and we recall basic definitions and results about $\mathrm{Mod}(\mathcal{C})$ and dualizing varieties.\\

In section \ref{sec:3}, we  consider the notion of ideal in a preadditive category $\mathcal{C}$.
We begin our work by  generalizing the classical adjunction for artin algebras given by \[\mathrm{Hom}_{\Lambda/I}(Y,\mathrm{Tr}_{\Lambda/I}(X))\simeq \mathrm{Hom}_{\Lambda}(Y,X)\]  to the case of rings with several objects. We study certain derived functors, and with the help of these derived functors, we give homological characterizations of when the functor $\overline{\mathrm{Tr}}_{\frac{\mathcal{C}}{\mathcal{I}}}:\mathrm{Mod}(\mathcal{C})\longrightarrow \mathrm{Mod}(\mathcal{C}/\mathcal{I})$ preserves injective coresolutions of length $k$ (see  \ref{tresfuntores} and \ref{Extiniciores}).\\

In section \ref{sec:4}, we study conditions on the ideal $\mathcal{I}$ under which we can restrict our previous results to the subcategory $\mathrm{mod}(\mathcal{C})$ of  finitely presented $\mathcal{C}$-modules. We introduce {\it the property $(A)$} in the ideal $\mathcal{I}$ (see definition \ref{propertyA}), and we prove that if an ideal satisfies property $A$ then we can restrict our attention to the case of finitely presented modules (see \ref{restresfuntores}). In particular, we prove that if $\mathcal{C}$ is a dualizing variety and $\mathcal{I}$ is an ideal satisfying  property $A$, then $\mathcal{C}/\mathcal{I}$ is also  dualizing  (see \ref{cocienteduali}).\\

In section \ref{sec:5}, we introduce the notion of $k$-idempotent ideal in preadditive categories (see  \ref{kidemcat}). We describe the idempotent ideals in terms of the vanishing of certain derived functors (see  \ref{caractidem} and \ref{caractidemfin}). Moreover, if the ideal $\mathcal{I}$ satisfies property $A$ and $\mathcal{C}$ is a dualizing $R$-variety, by using Auslander-Reiten duality, we give characterizations of when $\mathcal{I}$ is $k$-idempotent in terms of the functors $\mathrm{Ext}_{\mathrm{mod}(\mathcal{C})}^{i}(-,-)$ and $\mathrm{Tor}_{i}^{\mathcal{C}}(-,-)$ (see \ref{caractidemfin2}).\\

In section \ref{sec:6}, we prove the dual basis lemma for the category $\mathrm{Mod}(\mathcal{C})$ (see \ref{dualbasislema}), and given a projective module $P$, we introduce the trace ideal $\mathcal{I}:=\mathrm{Tr}_{P}\mathcal{C}$ (see \ref{deftraceideal}). We also show that if 
$\mathcal{C}$ is a dualizing $R$-variety, then  $\mathrm{Tr}_{P}\mathcal{C}$ satisfies property $A$ (see \ref{trasatproA}). We study projective resolutions of $k$-idempotent ideals
and introduce specials subcategories,  $\mathbb{P}_{k}$ and $\mathbb{I}_{k}$ (see \ref{defiPk}). Here we prove that  $\mathcal{I}=\mathrm{Tr}_{P}\mathcal{C}$ is $k+1$-idempotent  if and only if $\mathcal{I}(C',-)\in \mathbb{P}_{k}$ for all $C'\in \mathcal{C}$ (see \ref{otracaidemo}), which is a generalization of the result \cite[Theorem 2.1]{APG}.\\

In section \ref{sec:7},  we consider $P=\mathrm{Hom}_{\mathcal{C}}(C,-)\in \mathrm{mod}(\mathcal{C})$ and $R_{P}:=\mathrm{End}_{\mathrm{mod}(\mathcal{C})}(P)^{op}$ and we study the functor $\mathrm{Hom}_{\mathrm{mod}(\mathcal{C})}(P,-)$. We obtain a generalization of a well known recollement (see for example \cite[example 3.4]{Psaro3}) to the setting of dualizing $R$-varieties (see \ref{recolleproyefin}). Finally in this section, we prove that $\mathrm{mod}(R_{P})$ is equivalent to certain subcategories of $\mathrm{mod}(\mathcal{C})$ (see \ref{lemainteres}).\\

In section \ref{sec:8},  we study some homological properties of the functor \[\mathrm{Hom}_{\mathrm{mod}(\mathcal{C})}(P,-):\mathrm{mod}(\mathcal{C})\rightarrow \mathrm{mod}(R_{P}),\]
and how the homological properties of $\mathrm{mod}(\mathcal{C})$ and $\mathrm{mod}(R_{P})$ are related. We explore the relationship between injective coresolutions in $\mathrm{mod}(\mathcal{C})$ and $\mathrm{mod}(R_{P})$. We give necessary and sufficient conditions for $\mathbb{I}_{1}$ to be equal to $\mathbb{I}_{\infty}$. Finally, we study the property of an ideal to be projective.  This notion is important because the condition that $\mathcal{I}(C',-)$ being projective is part of the definition of heredity ideal given in \cite{Martin1}. We show that under certain conditions we are able to produce quasi-hereditary algebras (see \ref{quasiher2}). In addition, we have an application to derived categories (see \ref{fullcatdev}), which is a generalization of a well-known result for the category $\mathrm{Mod}(R),$ where $R$ is an associative ring.\\

Finally, in section \ref{sec:9}, some examples of $k$-homological ideals are explored.

\section{Preliminaries}\label{sec:2}
In this section we recall the basic notions of rings with several objects and the notion of dualizing $R$-variety introduced by Auslander and Reiten in \cite{AusVarietyI}.

\subsection{Categorical Foundations}
Recall that a category $\C$ together with an abelian group structure in each of the sets of morphisms $\C(C_{1},C_{2})$ is said to be a  \textbf{preadditive category}  whenever all the composition mappings
$\C(C',C'')\times \C(C,C')\longrightarrow \C(C,C'')$
in $ \C $ are bilinear mappings of abelian groups. A covariant functor $ F:\C_{1}\longrightarrow \C_{2} $ between  preadditive categories $ \C_{1} $ and $ \C_{2} $ is said to be \textbf{additive} if for each pair of objects $ C $ and $ C' $ in $ \C_{1}$, the mapping $ F:\C_{1}(C,C')\longrightarrow \C_{2}(F(C),F(C')) $ is a morphism of abelian groups. If $ \C $ is a  preadditive category, we always consider its opposite  category $ \C^{op}$ as a preadditive category by letting $\C^{op} (C',C)= \C(C,C') $. We follow the usual convention of identifying each contravariant  functor $F$ from a category $ \C $ to $ \mathcal{D} $ with a covariant functor $F$ from  $ \C^{op} $ to $ \mathcal{D}$.\\

An arbitrary category $\mathcal{C}$ is $\textbf{small}$ if the class of objects of $\mathcal{C}$ is a set. An $\textbf{additive}$ $\textbf{category}$ is a preadditive category $\mathcal{C}$ such that every finite family of objects in $\mathcal{C}$ has a coproduct. Given a small preadditive category $\mathcal{C}$ and $\mathcal{D}$ an arbitrary preadditive category, we denote by $(\mathcal{C},\mathcal{D})$ the category of all the covariant additive functors.

\subsection{The category $\mathrm{Mod}(\mathcal{C})$}
Throughout this section $\mathcal{C}$ will denote an arbitrary small preadditive category, and $\mathrm{Mod}(\mathcal{C})$ will denote the \textit{category of  additive covariant functors} from $\mathcal{C}$ to  the category of abelian groups $ \mathbf{Ab}$. This is  also referred as the category of $\mathcal{C}$-modules.  As usual, $\mathrm{Mod}(\mathcal{C}^{op})$ will be identified with the category of additive contravariant functors from $\mathcal{C}$ to  $ \mathbf{Ab}$.\\

Now, we recall some properties of the category $ \Mod(\C) $, and for more details, the reader may refer to   \cite{AusM1}. \\

As it is well-know, the category $\Mod(\C) $ is an abelian category with enough injectives and projectives. For each $C$ in $\C $, the $\C$-module $(C,-)$ given by $(C,-)(X)=\C(C,X)$ for each $X$ in $\C$, has the property that for each $\C$-module $M$, the mapping $\left( (C,-),M\right)\longrightarrow M(C)$ given by $f\mapsto f_{C}(1_{C})$ for each $\C$-morphism $f:(C,-)\longrightarrow M$ is an isomorphism of abelian groups (Yoneda's Lemma). \\

\begin{enumerate}
\item The functor $\mathbb{Y}:\C\longrightarrow \Mod(\C) $ given by $\mathbb{Y}(C)=(C,-) $ is fully faithful.

\item  For each family $\lbrace  C_{i}\rbrace _{i\in I}$ of objects in $ \C $, the $ \C $-module $ \underset{i\in I}\amalg \mathbb{Y}(C_{i})$ is projective.

\item  For each $M\in \mathrm{Mod}(\mathcal{C})$, there exists an epimorphism $ \underset{i\in I}\amalg \mathbb{Y}(C_{i})\rightarrow M$ for some family $ \lbrace C_{i}\rbrace_{i\in I}$  in $ \C $.  We say that $M$ is $\textbf{finitely generated}$ if such family is finite.

\item  A $\textbf{finitely generated projective}$ $\mathcal{C}$-module is a direct summand of $ \underset{i\in I}\amalg \mathbb{Y}(C_{i})$ for some finite family of objects $ \lbrace C_{i}\rbrace_{i\in I} $ in $\mathcal{C}$.

\end{enumerate}

\subsection{Change of Categories}
\label{subsec:2}
The results in this subsection are taken directly from \cite{AusM1}. Let $ \mathcal{C} $ be a small predditive category. There is a unique (up to isomorphism)  functor $ \otimes_{\mathcal{C}}:\mathrm{Mod}(\mathcal{C}^{op})\times \mathrm{Mod}(\mathcal{C})\longrightarrow \mathbf{Ab} $ called the \textbf{tensor product}. The abelian group $\otimes_{\mathcal{C}}(A,B) $  is denoted by $A\otimes_{\mathcal{C}} B $ for all $\mathcal{C}^{op}$-modules $A$ and all $\mathcal{C}$-modules $B$.
\begin{proposition}\label{AProposition0} The tensor product has the following properties:
\begin{enumerate}
\item
\begin{enumerate}
\item For each $ \mathcal{C}$-module $B$, the functor $\otimes_{\mathcal{C}}B:\mathrm{Mod}(\mathcal{C}^{op})\longrightarrow \mathbf{Ab}$ given by $(\otimes_{\mathcal{C}}B)(A)=A\otimes_{\mathcal{C}}B$ for all $\mathcal{C}^{op}$-modules $A$ is right exact.
\item For each $ \mathcal{C}^{op}$-module $A$, the functor $ A\otimes_{\mathcal{C}}:\mathrm{Mod}(\mathcal{C})\longrightarrow \mathbf{Ab}$ given by $(A\otimes_{\mathcal{C}})(B)=A\otimes_{\mathcal{C}}B $ for all $ \mathcal{C}$-modules $B$ is right exact.
\end{enumerate}

\item For each $ \mathcal{C}^{op}$-module $A$ and each $ \mathcal{C}$-module $B$, the functors $A\otimes_{\mathcal{C}}$ and $\otimes_{\mathcal{C}}B $ preserve arbitrary coproducts.

\item For each object $C$ in $ \mathcal{C} $ we have $ A\otimes_{\mathcal{C}}(C,-) =A(C)$ and $(-,C)\otimes_{\mathcal{C}}B=B(C)$ for all $\mathcal{C}^{op}$-modules $A$ and all $\mathcal{C}$-modules $B$.
\end{enumerate}
\end{proposition}
Suppose now that $ \mathcal{C}' $ is a preadditive subcategory of the small category $ \mathcal{C} $.
 We use the tensor product of $ \mathcal{C}'$-modules to describe the left adjoint $ \mathcal{C}\otimes_{\mathcal{C}'} $ of the restriction functor $ \mathrm{res}_{\mathcal{C}'}:\mathrm{Mod}(\mathcal{C})\longrightarrow \mathrm{Mod}(\mathcal{C}')$ given by $M\mapsto M|_{\mathcal{C}'}$.\\
Define the functor $ \mathcal{C}\otimes_{\mathcal{C}'}:\mathrm{Mod}\left( \mathcal{C}'\right) \longrightarrow \mathrm{Mod}\left( \mathcal{C}\right)  $ by $ (\mathcal{C}\otimes_{\mathcal{C}'}M)\left( C\right) =\mathcal{C}(-,C) \mid_{\mathcal{C}'}\otimes_{\mathcal{C}'} M $ for all $ M\in \mathrm{Mod}\left(\mathcal{C}'\right)  $ and $ C\in\mathcal{C}.$ 

\begin{proposition}
$\textnormal{\cite[Proposition 3.1]{AusM1}}$ \label{Rproposition1}
Let $ \mathcal{C}' $ be a subcategory of the small category $\mathcal{C}$. Then the functor $ \mathcal{C}\otimes_{\mathcal{C}'}:\mathrm{Mod}\left( \mathcal{C}'\right) \longrightarrow \mathrm{Mod}\left( \mathcal{C}\right)  $ satisfies:

\begin{enumerate}
\item $\mathcal{C}\otimes_{\mathcal{C}'}$ is right exact and preserves  coproducts;

\item The composition $ \mathrm{Mod}\left( \mathcal{C}'\right) \overset{\mathcal{C}\otimes_{\mathcal{C}'}}\longrightarrow \mathrm{Mod}\left( \mathcal{C}\right)\overset{\mathrm{res}_{\mathcal{C}'}}\longrightarrow \mathrm{Mod}\left( \mathcal{C}'\right)    $ is the identity; 

\item For each object $ C'\in \mathcal{C}' $, we have $ \mathcal{C}\otimes_{\mathcal{C}'}\mathcal{C}'\left(C',-\right) =\mathcal{C}\left(C', -\right); $

\item The restriction mapping $\mathcal{C}\left( \mathcal{C}\otimes_{\mathcal{C}'}M,N\right)\longrightarrow \mathcal{C}'\left( M,N\mid_{\mathcal{C}'}\right)$ is an isomorphism for each $\mathcal{C'}$-module $M$ and each $ \mathcal{C} $-module $ N $;
\item $ \mathcal{C}\otimes_{\mathcal{C}'} $ is a fully faithful functor.
\end{enumerate}
\end{proposition}
Having described the left adjoint $ \mathcal{C}\otimes_{\mathcal{C}'} $ of the restriction functor $ \mathrm{res}_{\mathcal{C}'}:\mathrm{Mod}\left(\mathcal{C}\right) \longrightarrow \mathrm{Mod}\left( \mathcal{C}'\right),$ we now describe its right adjoint. Define the functor $ \mathcal{C}'\left( \mathcal{C},-\right):\mathrm{Mod}\left( \mathcal{C}'\right) \longrightarrow \mathrm{Mod}\left( \mathcal{C}\right)$ by $\mathcal{C}'\left( \mathcal{C},M\right)\left( X\right) =\mathcal{C}'\left( \mathcal{C}(X,-) \mid_{\mathcal{C}'},M\right)    $ for all $ \mathcal{C}' $-modules $ M $ and all objects $ X $ in $ \mathcal{C}.$ We have  the following proposition.

\begin{proposition} 
$\textnormal{\cite[Proposition 3.4]{AusM1}}$ 
\label{Rproposition2}
Let $ \mathcal{C}' $ be a subcategory of the small category $\mathcal{C}$. Then the functor $ \mathcal{C}'\left( \mathcal{C},-\right) :\mathrm{Mod}\left( \mathcal{C}'\right) \longrightarrow \mathrm{Mod}\left( \mathcal{C}\right)$ has the following properties:
\begin{enumerate}

\item $ \mathcal{C}'\left( \mathcal{C},-\right)  $ is left exact and preserves inverse limits;

\item The composition $\mathrm{Mod}\left( \mathcal{C}'\right) \overset{\mathcal{C}'\left( \mathcal{C},-\right) }\longrightarrow \mathrm{Mod}\left( \mathcal{C}\right)\overset{\mathrm{res}_{\mathcal{C}'}}\longrightarrow \mathrm{Mod}\left( \mathcal{C}'\right)    $ is the identity;

\item The restriction mapping $\mathcal{C}\left( N,\mathcal{C}'\left( \mathcal{C},M\right) \right) \longrightarrow \mathcal{C}'\left( N\mid_{\mathcal{C}'},M\right)$ is an isomorphism for each $\mathcal{C}' $-module $M $ and $\mathcal{C}$-module $N$;

\item $\mathcal{C}'\left( \mathcal{C},-\right)  $ is a fully faithful functor.
\end{enumerate}
\end{proposition}

\subsection{Dualizing varieties and Krull-Schmidt Categories}
\label{subsec:3}
Let $\mathcal{C}$ be an additive category.  It is said that $\mathcal{C}$ is a category in which $\textbf{idempotents split}$ if given $e:C\longrightarrow C$ an idempotent endomorphism of an object $C\in \mathcal{C}$, then $e$ has a kernel in $\mathcal{C}$. Let us denote by $\mathbf{proj}(\mathcal{C})$ the full subcategory of $\mathrm{Mod}(\mathcal{C})$ consisting of
all finitely generated projective $\mathcal{C}$-modules. It is well known that $\mathrm{proj}(\mathcal{C})$ is a small additive category in which idempotents split, the functor $\mathbb{Y}:\mathcal{C}\rightarrow \mathrm{proj}(\mathcal{C})$ given by $\mathbb{Y}(C)=\mathcal{C}(C,-)$, is fully faithful, and induces by restriction $\mathrm{res}:\mathrm{Mod}(\mathrm{proj}(\mathcal{C})^{op})\rightarrow \mathrm{Mod}(\mathcal{C})$ an equivalence of categories. 
Also, recall the following notion given by Auslander in \cite{AusM1},  a $\textbf{variety}$  is a small additive category in which idempotents split.\\
Given a ring $R$, we denote by $\mathrm{Mod}(R)$ the category of left $R$-modules and by $\mathrm{mod}(R)$ the full subcategory of $\mathrm{Mod}(R)$ consisting of the finitely generated left $R$-modules. Now, we recall some notions from \cite{AusVarietyI}.

\begin{definition}\label{Rvarietydef}
Let $R$ be a commutative artin ring. An $\mathbf{R}$-$\textbf{category}$ $\mathcal{C}$, is an additive category such that $\mathcal{C}(C_{1},C_{2})$ is an $R$-module, and the composition is $R$-bilinear.  An $\mathbf{R}$-$\textbf{variety}$ $\mathcal{C}$ is a variety which is an $R$-category. An $R$-variety $\mathcal{C}$ is $\mathbf{Hom}$-\textbf{finite}, if for each pair of objects $C_{1},C_{2}$ in $\mathcal{C},$ the $R$-module $\mathcal{C}(C_{1},C_{2})$ is finitely generated. We denote by $(\mathcal{C},\mathrm{mod}(R))$, the full subcategory of $(\mathcal{C},\mathrm{
\mathrm{Mod}}(R))$ consisting of the $\mathcal{C}$-modules such that for
every $C$ in $\mathcal{C}$ the $R$-module $M(C)$ is finitely generated. 
\end{definition}

Suppose $\mathcal{C}$ is a Hom-finite $R$-variety. If $M:\mathcal{C}\longrightarrow \mathbf{Ab}$ is a $\mathcal{C}$-module, for each $C\in \mathcal{C}$ the abelian group $M(C)$ has a structure of $\mathrm{End}_{\mathcal{C}}(C)$-module, and hence as an $R$-module since $\mathrm{End}_{\mathcal{C}}(C)$ is an $R$-algebra. Then, $\mathrm{\mathrm{Mod}}(\mathcal{C})$ is an $R$-variety, which we identify with the category of
covariant functors $(\mathcal{C},\mathrm{Mod}(R))$. Moreover, $(\mathcal{C},\mathrm{mod}(R))$ is abelian and the inclusion $(\mathcal{C},\mathrm{mod}(R))\rightarrow (\mathcal{C},\mathrm{\mathrm{Mod}}(R))$ is exact.

\begin{definition}\label{defifinipre}
Let $\mathcal{C}$ be a Hom-finite $R$-variety. We denote by $\mathrm{mod}(\mathcal{C})$ the full subcategory of $\mathrm{Mod}(\mathcal{C})$ whose objects are the  $\textbf{finitely presented functors}$.
That is, $M\in \mathrm{mod}(\mathcal{C})$ if and only if  there exists an exact sequence in $\mathrm{Mod}(\mathcal{C})$
$$\xymatrix{\mathrm{Hom}_{\mathcal{C}}(C_{0},-)\ar[r] & \mathrm{Hom}_{\mathcal{C}}(C_{1},-)\ar[r] & M\ar[r] & 0.}$$
\end{definition}
Consider the functor $\mathbb{D}_{\mathcal{C}}:(\mathcal{C},\mathrm{mod}(R))\rightarrow (\mathcal{C}^{op},\mathrm{mod}(R))$, which is defined as
follows: for any object $C$ in $\mathcal{C}$, $\mathbb{D}_{\mathcal{C}}(M)(C)=\mathrm{Hom}
_{R}(M(C),E) $ where $E$ is
the $\textbf{injective envelope}$ of $R/\mathrm{rad}(R)\in \mathrm{mod}(R)$. The functor $\mathbb{D}_{\mathcal{C}}$ defines a duality between $(
\mathcal{C},\mathrm{mod}(R))$ and $(\mathcal{C}^{op},\mathrm{mod}(R))$. We have the following definition due to Auslander and Reiten (see \cite{AusVarietyI}).

\begin{definition}\label{dualizinvar}
An $\mathrm{Hom}$-finite $R$-variety $\mathcal{C}$ is \textbf{dualizing}, if
the functor $\mathbb{D}_{\mathcal{C}}:(\mathcal{C},\mathrm{mod}(R))\rightarrow (\mathcal{C}^{op},\mathrm{mod}(R))
$ induces a duality between the categories $\mathrm{mod}(\mathcal{C})$ and $
\mathrm{mod}(\mathcal{C}^{op}).$
\end{definition}

It is clear from the definition that for dualizing varieties $\mathcal{C}$
the category $\mathrm{mod}(\mathcal{C})$ has enough injectives and projectives. To finish, we recall the following definition:

\begin{definition}
An additive category $\mathcal{C}$ is \textbf{Krull-Schmidt}, if every
object in $\mathcal{C}$ decomposes in a finite sum of objects whose
endomorphism ring are local.
\end{definition}

We recall the following result.

\begin{theorem}\label{moddualizin}
Let $\mathcal{C}$ a dualizing $R$-variety. Then $\mathrm{mod}(
\mathcal{C})$ is a dualizing $R$-variety.
\end{theorem}

\section{An adjunction and some derived functors}\label{sec:3}

In the article \cite{APG}, Auslander-Platzeck-Todorov studied homological ideals in the case of $\mathrm{mod}(\Lambda)$ where $\Lambda$ is an artin algebra. Given a two sided ideal $I$ of $\Lambda$ they consider $\Lambda/I$ and they studied the trace  $\mathrm{Tr}_{\Lambda/I}(M)$ of a $\Lambda$-module $M$ defined as $\mathrm{Tr}_{\Lambda/I}(M)=\!\!\!\!\!\!
\!\!\!\!\!\!\displaystyle\sum _{f\in\mathrm{Hom}_{\Lambda}(\Lambda/I, M)}
\!\!\!\!\!\!\!\!\!\!\!\!\mathrm{Im}(f).$ In order to define the analogous of  $\mathrm{Tr}_{\Lambda/I}(M)$ in the category $\mathrm{Mod}(\mathcal{C})$ we introduced the following notions. \\
In this section $\mathcal{C}$ will be a small preadditive category. Let  $\mathcal{M}=\{M_{i}\}_{i\in I}$ be a family of $\mathcal{C}$-modules and set $M:=\bigoplus_{i\in I}M_{i}$. For $F\in \mathrm{Mod}(\mathcal{C})$ we  define $\Lambda_{F}:=\mathrm{Hom}_{\mathrm{Mod}(\mathcal{C})}(M,F),$ and for $\lambda\in \Lambda_{F}$ we set $u_{\lambda}:M\longrightarrow M^{(\Lambda_{F})}$ as the $\lambda$-th inclusion of $M$ into $M^{(\Lambda_{F})}:=\bigoplus_{\lambda\in \Lambda_{F}}M$. For $\lambda\in \Lambda_{F}$ we have the morphism $\lambda:M\longrightarrow F$, then by the universal property of the coproduct, there exists a unique morphism $\Theta_{F}:M^{(\Lambda_{F})}\longrightarrow F$ such that $\lambda=\Theta_{F}\circ u_{\lambda}$ 
for every $\lambda\in \Lambda_{F}$.

\begin{definition}\label{definitraza}	
The $\textbf{trace}$ of $F$ respect to the family $\mathcal{M}=\{M_{i}\}_{i\in I}$, denoted by $\mathrm{Tr}_{\mathcal{M}}(F)$, is the image of  $\Theta_{F}$. That is, we have the following factorization $\xymatrix{M^{(\Lambda_{F})}\ar[r]^{\Delta_{F}} &  \mathrm{Tr}_{\mathcal{M}}(F)\ar[r]^{\Psi_{F}} & F}$ where $\Theta_{F}=\Psi_{F}\Delta_{F}$ with $\Delta_{F}$ an epimorphism and $\Psi_{F}$ a monomorphism.
\end{definition}

\begin{proposition}\label{tracefuntor}
For each family $\mathcal{M}=\{M_{i}\}_{i\in I}$, we have a functor
$\mathrm{Tr}_{\mathcal{M}}:\mathrm{Mod}(\mathcal{C})\rightarrow \mathrm{Mod}(\mathcal{C})$.
\end{proposition}
\begin{proof}
Straightforward.
\end{proof}

Now, we recall the following definitions that will be essential throughout this work.

\begin{definition}
Let $ \mathcal{C}$ be a preadditive category. An $\textbf{ideal}$ $ \mathcal{I} $  of  $ \mathcal{C} $ is an additive subfunctor of $\mathrm{Hom}_\mathcal{C}(-,-) $. That is, $ \mathcal{I} $ is a class of the morphisms in $\mathcal{C}$  such that:
\begin{enumerate}
\item [(a)]   $\mathcal{I}(A,B)=\mathrm{Hom}_{\mathcal{C}}(A,B)\cap\mathcal{I} $ is an abelian subgroup of $\mathrm{Hom}_\mathcal{C}(A,B) $ for each $ A,B\in\mathcal{C}$; 

\item [(b)] If $ f\in\mathcal{I}(A,B),\;g\in\mathrm{Hom}_\mathcal{C}(C,A) $ and $ h\in\mathrm{Hom}_\mathcal{C}(B,D) $,  then $ hfg\in\mathcal{I}(C,D)$.

\item [(c)] Let $ \mathcal{I} $ and $ \mathcal{J}$ be ideals in  $ \mathcal{C} $. The  $\textbf{product of ideals}$ $\mathcal{I}\mathcal{J}$ is defined as follows: for each $ A,B\in\mathcal{C}$ we set
$$\mathcal{IJ}(A,B):=\left \{\sum_{i=1}^{n}f_{i}g_{i} \mathrel{\bigg|} g_{i}\in \mathcal{C}(A,C_{i}), f_{i}\in \mathcal{C}(C_{i},B)\,\text{for some}\,\, \, C_{i}\in\mathcal{C}\right \} .$$
We say that an ideal $ \mathcal{I} $ of $ \mathcal{C} $ is $\textbf{idempotent}$ if $ \mathcal{I}^2=\mathcal{I}.$
\item [(d)]
Let  $\mathcal{I}$ be an ideal of $\mathcal{C}$, we set $\mathrm{Ann}(\mathcal{I}):=\{F\in \mathrm{Mod}(\mathcal{C})\mid F(f)=0\,\,\forall f\in \mathcal{I}(A,B)\,\,\forall A,B\in \mathcal{C}\}.$
\end{enumerate}
\end{definition}
Now, we recall the construction of the quotient category.
Let $\mathcal{I}$ be an ideal in a preadditive category $\mathcal{C}$. The $\textbf{quotient category}$ $\mathcal{C}/\mathcal{I}$ is defined as follows: it has the same objects as $\mathcal{C}$ and $\mathrm{Hom}_{\mathcal{C}/\mathcal{I}}(A,B):=\frac{\mathrm{Hom}_{\mathcal{C}}(A,B)}{\mathcal{I}(A,B)}$ for each $A,B\in \mathcal{C}/\mathcal{I}$.\\
For $\overline{f}=f+\mathcal{I}(A,B)\in \mathrm{Hom}_{\mathcal{C}/\mathcal{I}}(A,B)$ and $\overline{g}=g+\mathcal{I}(B,C)\in \mathrm{Hom}_{\mathcal{C}/\mathcal{I}}(B,C)$ we set
$$\overline{g}\circ \overline{f}:=gf+\mathcal{I}(A,C)\in \mathrm{Hom}_{\mathcal{C}/\mathcal{I}}(A,C).$$ 
Let  $\mathcal{I}$ be an ideal of $\mathcal{C}$, we have the canonical functor $\pi:\mathcal{C}\longrightarrow \mathcal{C}/\mathcal{I}$ defined as:  $\pi(A)=A$ $\forall A\in \mathcal{C}$ and $\pi(f):=\overline{f}=f+\mathcal{I}(A,B)\in \mathrm{Hom}_{\mathcal{C}/\mathcal{I}}(A,B)$  $\forall f\in \mathrm{Hom}_{\mathcal{C}}(A,B)$.

\begin{definition}
Let $\mathcal{I}$ be an ideal in a preadditive category $\mathcal{C}$ and consider the functor $\pi:\mathcal{C}\longrightarrow \mathcal{C}/\mathcal{I}$. We  have the functor $\pi_{\ast}:\mathrm{Mod}(\mathcal{C}/\mathcal{I})\longrightarrow \mathrm{Mod}(\mathcal{C})$  defined as follows: $\pi_{\ast}(F):=F\circ \pi$ for $F\in \mathrm{Mod}(\mathcal{C}/\mathcal{I})$ and $\pi_{\ast}(\eta)=\eta$ for $\eta:F\longrightarrow G$ in $\mathrm{Mod}(\mathcal{C}/\mathcal{I})$.
\end{definition}
Now, we construct the analogous of the functor $\mathrm{Tr}_{\Lambda/I}(-).$
\begin{definition}
Let $\mathcal{I}$ be an ideal in $\mathcal{C}$, for $C\in \mathcal{C}$ we set $M_{C}:=\frac{\mathrm{Hom}_{\mathcal{C}}(C,-)}{\mathcal{I}(C,-)}\in \mathrm{Mod}(\mathcal{C})$. We consider the familiy $\mathcal{M}=\{M_{C}\}_{C\in \mathcal{C}}$ and we define $\mathrm{Tr}_{\frac{\mathcal{C}}{\mathcal{I}}}:=\mathrm{Tr}_{\mathcal{M}}:\mathrm{Mod}(\mathcal{C})\longrightarrow \mathrm{Mod}(\mathcal{C}).$
\end{definition}

\begin{Remark}\label{isomoan} 
\begin{enumerate}
\item [(a)]
For every $F\in \mathrm{Mod}(\mathcal{C})$ we have that $\mathrm{Tr}_{\frac{\mathcal{C}}{\mathcal{I}}}(F)\in \mathrm{Ann}(\mathcal{I})$.

\item [(b)] It is well known that there exists a functor $\Omega:\mathrm{Ann}(\mathcal{I})\longrightarrow \mathrm{Mod}(\mathcal{C}/\mathcal{I})$ satisfying that $\pi_{\ast}\circ \Omega=1_{\mathrm{Ann}(\mathcal{I})}$ and  $\Omega\circ \pi_{\ast}=1_{\mathrm{Mod}(\mathcal{C}/\mathcal{I})}$ (see for example \cite[Lemma 2.1]{Parra}).
\end{enumerate}
\end{Remark}

We have the following properties and the proof is left to the reader.

\begin{lemma}\label{descripcionMC}
Let $\mathcal{I}$ be an ideal in $\mathcal{C}$.
\begin{enumerate}
\item [(a)]  Let $F\in \mathrm{Mod}(\mathcal{C}/\mathcal{I})$. Then $\mathrm{Tr}_{\frac{\mathcal{C}}{\mathcal{I}}}(F\circ \pi)=F\circ \pi$.

\item [(b)]  Let $F\in\mathrm{Mod}(\mathcal{C})$. Then $\mathrm{Tr}_{\frac{\mathcal{C}}{\mathcal{I}}}(F)=F$ if and only if $F\in \mathrm{Ann}(\mathcal{I})$.

\item [(c)]  Let $M_{C}:=\frac{\mathrm{Hom}_{\mathcal{C}}(C,-)}{\mathcal{I}(C,-)}\in \mathrm{Mod}(\mathcal{C})$, then 
$\Omega(M_{C})=\overline{M_{C}}=\mathrm{Hom}_{\mathcal{C}/\mathcal{I}}(C,-)$ and  $M_{C}=\mathrm{Hom}_{\mathcal{C}/\mathcal{I}}(C,-)\circ \pi=\pi_{\ast}(\mathrm{Hom}_{\mathcal{C}/\mathcal{I}}(C,-))$.
\end{enumerate}
\end{lemma}
\begin{proof}
Straightforward.
\end{proof}

Now, let $\overline{\mathrm{Tr}}_{\frac{\mathcal{C}}{\mathcal{I}}}:=\Omega \circ \mathrm{Tr}_{\frac{\mathcal{C}}{\mathcal{I}}}:
\mathrm{Mod}(\mathcal{C})\longrightarrow \mathrm{Mod}(\mathcal{C}/\mathcal{I})$. Let us see that  $\pi_{\ast}$ is left adjoint to $\overline{\mathrm{Tr}}_{\frac{\mathcal{C}}{\mathcal{I}}}$. 

\begin{proposition}\label{firstadjoin}
The functor $\pi_{\ast}:\mathrm{Mod}(\mathcal{C}/\mathcal{I})\longrightarrow \mathrm{Mod}(\mathcal{C})$ is left adjoint to  $\overline{\mathrm{Tr}}_{\frac{\mathcal{C}}{\mathcal{I}}}:=\Omega \circ \mathrm{Tr}_{\frac{\mathcal{C}}{\mathcal{I}}}:
\mathrm{Mod}(\mathcal{C})\longrightarrow \mathrm{Mod}(\mathcal{C}/\mathcal{I})$. That is, there exists a natural isomorphism
$$\theta_{F,G}:\mathrm{Hom}_{\mathrm{Mod}(\mathcal{C})}(\pi_{\ast}(F),G)\longrightarrow \mathrm{Hom}_{\mathrm{Mod}(\mathcal{C}/\mathcal{I})}(F,\overline{\mathrm{Tr}}_{\frac{\mathcal{C}}{\mathcal{I}}}(G))$$
for $F\in \mathrm{Mod}(\mathcal{C}/\mathcal{I})$ and $G\in \mathrm{Mod}(\mathcal{C})$.
\end{proposition}
\begin{proof}
First, we construct the unit $\eta:1_{\mathrm{Mod}(\mathcal{C}/\mathcal{I})} \longrightarrow \overline{\mathrm{Tr}}_{\frac{\mathcal{C}}{\mathcal{I}}}\circ \pi_{\ast}$ of the adjunction.
Indeed, we have that $(\overline{\mathrm{Tr}}_{\frac{\mathcal{C}}{\mathcal{I}}}\circ \pi_{\ast})(F)=(\Omega\circ \mathrm{Tr}_{\frac{\mathcal{C}}{\mathcal{I}}})(F\circ \pi)=\Omega(F\circ \pi)=(\Omega\circ \pi_{\ast})(F)=F$ (see \ref{descripcionMC}(a) and \ref{isomoan}(b)). Then, for each $F\in \mathrm{Mod}(\mathcal{C}/\mathcal{I})$ we define
$\eta_{F}:=1_{F}:F \longrightarrow (\overline{\mathrm{Tr}}_{\frac{\mathcal{C}}{\mathcal{I}}}\circ \pi_{\ast})(F).$\\
Now, we define the counit $\epsilon:\pi_{\ast} \circ  \overline{\mathrm{Tr}}_{\frac{\mathcal{C}}{\mathcal{I}}}\longrightarrow 1_{\mathrm{Mod}(\mathcal{C})}$  of the adjunction. We note that for $G\in \mathrm{Mod}(\mathcal{C})$ we have $(\pi_{\ast} \circ  \overline{\mathrm{Tr}}_{\frac{\mathcal{C}}{\mathcal{I}}})(G)=(\pi_{\ast}\circ \Omega \circ \mathrm{Tr}_{\frac{\mathcal{C}}{\mathcal{I}}})(G)=
(\pi_{\ast}\circ \Omega)(\mathrm{Tr}_{\frac{\mathcal{C}}{\mathcal{I}}}(G))=\mathrm{Tr}_{\frac{\mathcal{C}}{\mathcal{I}}}(G)$ (see \ref{isomoan}(b)). Then, for $G\in\mathrm{Mod}(\mathcal{C})$ we define
$\epsilon_{G}:=\Psi_{G}$ where $\Psi_{G}:\mathrm{Tr}_{\frac{\mathcal{C}}{\mathcal{I}}}(G)\longrightarrow G$ is the canonical inclusion given in \ref{definitraza}. Now, it is straightforward to check the triangular identities.

\end{proof}

For the following results of this section we are going to use the functor $ \otimes_{\mathcal{C}}:\mathrm{Mod}(\mathcal{C}^{op})\times \mathrm{Mod}(\mathcal{C})\longrightarrow \mathbf{Ab}$ which we introduced in section \ref{subsec:2}. Let $\mathcal{I}$ be an ideal in a preadditive category $\mathcal{C}$. We recall the following functor (for more details see \cite{LeOS2}).

\begin{definition}\label{HomCI}
We define the functor $ \frac{\mathcal{C}}{\mathcal{I}}\otimes_{\mathcal{C}}:\mathrm{Mod}(\mathcal{C})\longrightarrow\mathrm{Mod}(\mathcal{C}/\mathcal{I}) $ as follows: for $M\in \mathrm{Mod}(\mathcal{C})$ we set 
$\left( \frac{\mathcal{C}}{\mathcal{I}}\otimes _{\mathcal{C}}M\right)(C):= \frac{\mathcal{C}(-,C)}{\mathcal{I}(-,C)}\otimes_{\mathcal{C}}M$  for all $C\in \mathcal{C}/\mathcal{I}$
and  $\left( \frac{\mathcal{C}}{\mathcal{I}}\otimes _{\mathcal{C}}M\right)(\overline{f})=\frac{\mathcal{C}}{\mathcal{I}}(-,f)\otimes_{\mathcal{C}}M$ for all $\overline{f}=f+\mathcal{I}(C,C')\in \mathrm{Hom}_{\mathcal{C}/\mathcal{I}}(C,C')$.
\end{definition}
We also recall the following functor which will be fundamental in this work.

\begin{definition}\label{TenCI}
We define the functor $\mathcal{C}(\frac{\mathcal{C}}{\mathcal{I}},-):\mathrm{Mod}\left(\mathcal{C}\right) \longrightarrow \mathrm{Mod}\left(\mathcal{C}/\mathcal{I}\right) $ as follows: for $ M\in \mathrm{Mod}(\mathcal{C})$ we set 
$\mathcal{C}(\frac{\mathcal{C}}{\mathcal{I}},M)(C)=\mathcal{C}\left(\frac{\mathcal{C}(C,-)}{\mathcal{I}(C,-)},M\right)$ for all $C\in \C/\mathcal{I}$ and 
$\mathcal{C}(\frac{\mathcal{C}}{\mathcal{I}},M)(\overline{f})=\mathcal{C}\left( \frac{\mathcal{C}}{\mathcal{I}}(f,-),M\right) 
$ for all $\overline{f}=f+\mathcal{I}(C,C')\in \mathrm{Hom}_{\mathcal{C}/\mathcal{I}}(C,C')$.
\end{definition}
It is well known that $\mathcal{C}(\frac{\mathcal{C}}{\mathcal{I}},-):\mathrm{Mod}\left(\mathcal{C}\right) \longrightarrow \mathrm{Mod}\left(\mathcal{C}/\mathcal{I}\right)$ is right adjoint to $\pi_{\ast}$ and $\frac{\mathcal{C}}{\mathcal{I}}\otimes_{\mathcal{C}}$ is left adjoint to $\pi_{\ast}$ (see for example, \cite[Proposition 3.9]{LeOS2}), and hence by \ref{firstadjoin}, we have that $\mathcal{C}(\frac{\mathcal{C}}{\mathcal{I}},-)\simeq \overline{\mathrm{Tr}}_{\frac{\mathcal{C}}{\mathcal{I}}}$ since adjoint functors are unique up to isomorphisms. Thus, we have the following result.

\begin{proposition}\label{tresfuntores}
Let $\mathcal{I}$ be an ideal in $\mathcal{C}$  and $\pi:C\longrightarrow \mathcal{C}/\mathcal{I}$ the canonical functor. Then we have the following diagram
$$\xymatrix{\mathrm{Mod}(\mathcal{C}/\mathcal{I})\ar[rr]|{\pi_{\ast}}  &  &\mathrm{Mod}(\mathcal{C})\ar@<-2ex>[ll]_{\pi^{\ast}}\ar@<2ex>[ll]^{\pi^{!}} }$$
where $(\pi^{\ast},\pi_{\ast})$ and $(\pi_{\ast},\pi^{!})$ are adjoint pairs, $\pi^{!}:=\mathcal{C}(\frac{\mathcal{C}}{\mathcal{I}},-)\simeq \overline{\mathrm{Tr}}_{\frac{\mathcal{C}}{\mathcal{I}}}$ and $\pi^{\ast}:=\frac{\mathcal{C}}{\mathcal{I}}\otimes_{\mathcal{C}}$.
\end{proposition}
We recall that for every small preadditive category $\mathcal{C}$ it is well known that $\mathrm{Mod}(\mathcal{C})$ is an abelian category with enough projectives and enough injectives. So, we can define derived functors in $\mathrm{Mod}(\mathcal{C})$.\\
Let $M\in \mathrm{Mod}(\mathcal{C})$, we denote by
$\mathrm{Ext}^{i}_{\mathrm{Mod}(\mathcal{C})}(M,-):\mathrm{Mod}(\mathcal{C})\longrightarrow \mathbf{Ab}$ the $i$-th  derived functor of $\mathrm{Hom}_{\mathrm{Mod}(\mathcal{C})}(M,-):\mathrm{Mod}(\mathcal{C})\longrightarrow \mathbf{Ab}$. Similarly we have $\mathrm{Ext}^{i}_{\mathrm{Mod}(\mathcal{C})}(-,M):\mathrm{Mod}(\mathcal{C})^{op}\longrightarrow \mathbf{Ab}$.\\
Now, we can construct canonical morphisms.

\begin{proposition}\label{morenexte}
Let $G\in \mathrm{Mod}(\mathcal{C})$  and $F\in \mathrm{Mod}(\mathcal{C}/\mathcal{I})$. Then, there exists a morphisms of abelian groups
$\varphi^{i}_{F,G}:\mathrm{Ext}^{i}_{\mathrm{Mod}(\mathcal{C}/\mathcal{I})}(F,\overline{\mathrm{Tr}}_{\frac{\mathcal{C}}{\mathcal{I}}}(G))\longrightarrow \mathrm{Ext}^{i}_{\mathrm{Mod}(\mathcal{C})}(\pi_{\ast}(F),G)$ for each $i\geq 0$.
\end{proposition}
\begin{proof}
Since $\pi_{\ast}$ is exact, it induces a morphism
$$\mathrm{Ext}^{i}_{\mathrm{Mod}(\mathcal{C}/\mathcal{I})}(F,\overline{\mathrm{Tr}}_{\frac{\mathcal{C}}{\mathcal{I}}}(G))\longrightarrow  \mathrm{Ext}^{i}_{\mathrm{Mod}(\mathcal{C})}(\pi_{\ast}(F),\pi_{\ast}(\overline{\mathrm{Tr}}_{\frac{\mathcal{C}}{\mathcal{I}}}(G))).$$
Moreover, the counit $\epsilon:\pi_{\ast}(\overline{\mathrm{Tr}}_{\frac{\mathcal{C}}{\mathcal{I}}}(G)\longrightarrow G$ induces a morphism
$$ \mathrm{Ext}^{i}_{\mathrm{Mod}(\mathcal{C})}(\pi_{\ast}(F),\pi_{\ast}(\overline{\mathrm{Tr}}_{\frac{\mathcal{C}}{\mathcal{I}}}(G)))\longrightarrow  \mathrm{Ext}^{i}_{\mathrm{Mod}(\mathcal{C})}(\pi_{\ast}(F),G).$$
Hence, we have the required morphism.
\end{proof}
Now, we give the following definition which is the analogous to the multiplication of an ideal and a module in the classical sense.

\begin{definition}\label{prodIfunc}
Let $G\in\mathrm{Mod}(\mathcal{C})$ and $\mathcal{I}$ an ideal in $\mathcal{C}$. We define $\mathcal{I}G$ as the subfunctor of $G$ defined as follows: for $X\in \mathcal{C}$ we set
$\displaystyle\mathcal{I}G(X):=\sum_{f\in \bigcup_{C\in \mathcal{C}}\mathcal{I}(C,X)}\mathrm{Im} (G(f)).$
\end{definition}

Let $N\in \mathrm{Mod}(\mathcal{C}^{op})$ and consider the functor
$N\otimes-:\mathrm{Mod}(\mathcal{C})\longrightarrow \mathbf{Ab}$. We denote by $\mathrm{Tor}_{i}^{\mathcal{C}}(N,-):\mathrm{Mod}(\mathcal{C})\longrightarrow \mathbf{Ab}$ the $i$-th left derived functor of $N\otimes -$.

\begin{proposition}\label{morentor}
Let $F\in \mathrm{Mod}((\mathcal{C}/\mathcal{I})^{op})$ and $G\in \mathrm{Mod}(\mathcal{C})$, and consider a projective resolution $(P^{\bullet},\epsilon_{G})$ of $G$.
Then for each $i\geq 0$, there exists a canonical morphism of abelian groups $\psi_{F,G}^{i}:\mathrm{Tor}^{\mathcal{C}}_{i}(F\circ \pi,G)\longrightarrow \mathrm{Tor}^{\mathcal{C}/\mathcal{I}}_{i}(F, G/\mathcal{I}G).$
\end{proposition}
\begin{proof}
Similar to \ref{morenexte}.
\end{proof}

Now, we give the following definition.

\begin{definition}
Consider the functors $\mathcal{C}(\frac{\mathcal{C}}{\mathcal{I}},-), \frac{\mathcal{C}}{\mathcal{I}}\otimes_{\mathcal{C}}-:\mathrm{Mod}(\mathcal{C})\longrightarrow \mathrm{Mod}(\mathcal{C}/\mathcal{I})$ given in the definitions \ref{HomCI} and \ref{TenCI}. We denote by $\mathbb{EXT}^{i}_{\mathcal{C}}(\mathcal{C}/\mathcal{I},-):\mathrm{Mod}(\mathcal{C})\longrightarrow \mathrm{Mod}(\mathcal{C}/\mathcal{I})$ the $i$-th right derived functor of $\mathcal{C}(\frac{\mathcal{C}}{\mathcal{I}},-)$ and $\mathbb{TOR}_{i}^{\mathcal{C}}(\mathcal{C}/\mathcal{I},-):\mathrm{Mod}(\mathcal{C})\longrightarrow \mathrm{Mod}(\mathcal{C}/\mathcal{I})$ the $i$-th left derived functor of $ \frac{\mathcal{C}}{\mathcal{I}}\otimes_{\mathcal{C}}$.
\end{definition}

\begin{Remark}\label{descEXT}
Consider the functor $\mathbb{EXT}^{i}_{\mathcal{C}}(\mathcal{C}/\mathcal{I},-):\mathrm{Mod}(\mathcal{C})\longrightarrow \mathrm{Mod}(\mathcal{C}/\mathcal{I})$.  For $C\in \mathcal{C}/\mathcal{I}$ we have that: $\mathbb{EXT}^{i}_{\mathcal{C}}(\mathcal{C}/\mathcal{I},M)(C)=\mathrm{Ext}^{i}_{\mathrm{Mod}(\mathcal{C})}\left(\frac{\mathrm{Hom}_{\mathcal{C}}(C,-)}{\mathcal{I}(C,-)},M\right).$
\end{Remark}

Now, we have the following proposition which will help us to characterize $k$-idempotent ideals in the forthcoming sections.

\begin{proposition}\label{Extiniciores}
Let $\mathcal{I}$ an ideal in $\mathcal{C}$, $\pi:C\longrightarrow \mathcal{C}/\mathcal{I}$ the canonical functor and consider the diagram given in \ref{tresfuntores}. Let $G\in \mathrm{Mod}(\mathcal{C})$, and $0\rightarrow G\rightarrow I_{0}\rightarrow I_{1}\rightarrow \dots\rightarrow$ an injective coresolution of $G$. For $1\leq k\leq \infty$, the following conditions are equivalent.
\begin{enumerate}
\item [(a)]  $0\rightarrow \overline{\mathrm{Tr}}_{\frac{\mathcal{C}}{\mathcal{I}}}(G)\rightarrow \overline{\mathrm{Tr}}_{\frac{\mathcal{C}}{\mathcal{I}}}(I_{0})\rightarrow \overline{\mathrm{Tr}}_{\frac{\mathcal{C}}{\mathcal{I}}}(I_{1})\rightarrow \dots\rightarrow \overline{\mathrm{Tr}}_{\frac{\mathcal{C}}{\mathcal{I}}}(I_{k})$ is the beginning of an injective coresolution of  $\overline{\mathrm{Tr}}_{\frac{\mathcal{C}}{\mathcal{I}}}(G)\in \mathrm{Mod}(\mathcal{C}/\mathcal{I})$.

\item [(b)] $\mathbb{EXT}^{i}_{\mathcal{C}}(\mathcal{C}/\mathcal{I},G)=0$ for all $1\leq i\leq k$.

\item [(c)] For $F\in \mathrm{Mod}(\mathcal{C}/\mathcal{I})$ the morphisms given in \ref{morenexte},
$$\varphi^{i}_{F,G}:\mathrm{Ext}^{i}_{\mathrm{Mod}(\mathcal{C}/\mathcal{I})}(F,\overline{\mathrm{Tr}}_{\frac{\mathcal{C}}{\mathcal{I}}}(G))\longrightarrow \mathrm{Ext}^{i}_{\mathrm{Mod}(\mathcal{C})}(\pi_{\ast}(F),G),$$
are isomorphisms for $1\leq i\leq k$.
\end{enumerate}
\end{proposition}
\begin{proof}
$(b)\Leftrightarrow (a)$. By definition of the derived functor, we have that $\mathbb{EXT}^{i}_{\mathcal{C}}(\mathcal{C}/\mathcal{I},G)$ is the $i$-th homology of the complex of $\mathcal{C}/\mathcal{I}$-modules
$$\mathcal{C}(\mathcal{C}/I,I_{0})\rightarrow \mathcal{C}(\mathcal{C}/I,I_{1})\rightarrow \dots\rightarrow \mathcal{C}(\mathcal{C}/I,I_{k})\rightarrow \cdots$$
But $\overline{\mathrm{Tr}}_{\frac{\mathcal{C}}{\mathcal{I}}}=\mathcal{C}(\mathcal{C}/I,-)$, then we have that $\mathbb{EXT}^{i}_{\mathcal{C}}(\mathcal{C}/\mathcal{I},G)=0$ for all $1\leq i\leq k$ if and only if the following complex is exact
$$0\rightarrow \overline{\mathrm{Tr}}_{\frac{\mathcal{C}}{\mathcal{I}}}(G)\rightarrow \overline{\mathrm{Tr}}_{\frac{\mathcal{C}}{\mathcal{I}}}(I_{0})\rightarrow \overline{\mathrm{Tr}}_{\frac{\mathcal{C}}{\mathcal{I}}}(I_{1})\rightarrow \dots\rightarrow \overline{\mathrm{Tr}}_{\frac{\mathcal{C}}{\mathcal{I}}}(I_{k})$$
where each $\overline{\mathrm{Tr}}_{\frac{\mathcal{C}}{\mathcal{I}}}(I_{j})$ is an injective $\mathcal{C}/\mathcal{I}$-module.\\
$(a)\Rightarrow (c)$. Suppose that  $0\rightarrow \overline{\mathrm{Tr}}_{\frac{\mathcal{C}}{\mathcal{I}}}(G)\rightarrow \overline{\mathrm{Tr}}_{\frac{\mathcal{C}}{\mathcal{I}}}(I_{0})\rightarrow \overline{\mathrm{Tr}}_{\frac{\mathcal{C}}{\mathcal{I}}}(I_{1})\rightarrow \dots\rightarrow \overline{\mathrm{Tr}}_{\frac{\mathcal{C}}{\mathcal{I}}}(I_{k})$ is the beginning of an injective coresolution of  $\overline{\mathrm{Tr}}_{\frac{\mathcal{C}}{\mathcal{I}}}(G)$. We can complete to an injective coresolution: $0\rightarrow \overline{\mathrm{Tr}}_{\frac{\mathcal{C}}{\mathcal{I}}}(G)\rightarrow \overline{\mathrm{Tr}}_{\frac{\mathcal{C}}{\mathcal{I}}}(I_{0})\rightarrow \dots\rightarrow \overline{\mathrm{Tr}}_{\frac{\mathcal{C}}{\mathcal{I}}}(I_{k})\rightarrow I_{k+1}'\rightarrow I_{k+2}'\rightarrow\cdots .$ By using the construction in \ref{morenexte}, we get isomorphisms for $1\leq i\leq k$
$$\varphi^{i}_{F,G}:\mathrm{Ext}^{i}_{\mathrm{Mod}(\mathcal{C}/I)}(F,\overline{\mathrm{Tr}}_{\frac{\mathcal{C}}{\mathcal{I}}}(G))\longrightarrow \mathrm{Ext}^{i}_{\mathrm{Mod}(\mathcal{C})}(\pi_{\ast}(F),G).$$
$(c)\Rightarrow (b)$. By \ref{descripcionMC}(c), we have that 
$M_{C}:=\frac{\mathrm{Hom}_{\mathcal{C}}(C,-)}{\mathcal{I}(C,-)}\in \mathrm{Mod}(\mathcal{C})$ satisfies that  $M_{C}=\mathrm{Hom}_{\mathcal{C}/\mathcal{I}}(C,-)\circ \pi=\pi_{\ast}\Big(\mathrm{Hom}_{\mathcal{C}/\mathcal{I}}(C,-)\Big)$. Let us fix $i$ such that $1\leq i\leq k$. We have that $\mathbb{EXT}^{i}_{\mathcal{C}}(\mathcal{C}/\mathcal{I},G)\in \mathrm{Mod}(\mathcal{C}/\mathcal{I})$ is  defined for $C\in \mathcal{C}/\mathcal{I}$ as follows. By \ref{descEXT}, we have that
\begin{align*}
\mathbb{EXT}^{i}_{\mathcal{C}}(\mathcal{C}/\mathcal{I},G)(C)& :=\mathrm{Ext}^{i}_{\mathrm{Mod}(\mathcal{C})}\left(\frac{\mathrm{Hom}_{\mathcal{C}}(C,-)}{\mathcal{I}(C,-)},G\right)\\
& =\mathrm{Ext}^{i}_{\mathrm{Mod}(\mathcal{C})}\left(\pi_{\ast}\Big(\mathrm{Hom}_{\mathcal{C}/\mathcal{I}}(C,-)\Big),G\right)\\
& \simeq \mathrm{Ext}^{i}_{\mathrm{Mod}(\mathcal{C}/I)}\left(\mathrm{Hom}_{\mathcal{C}/\mathcal{I}}(C,-),\overline{\mathrm{Tr}}_{\frac{\mathcal{C}}{\mathcal{I}}}
(G)\right) \quad [\text{hypothesis}]\\
& =0\quad \quad \quad \quad\quad \quad [ \mathrm{Hom}_{\mathcal{C}/\mathcal{I}}(C,-)\,\,\text{es projective in }\,\, \mathrm{Mod}(\mathcal{C}/\mathcal{I})]
\end{align*}
\end{proof}
Now, we have the following result that is analogous to the previous result.

\begin{proposition}\label{Toreequiv}
Let $\mathcal{I}$ be an ideal in $\mathcal{C}$,  $\pi:C\longrightarrow \mathcal{C}/\mathcal{I}$ the canonical functor and consider the diagram given in \ref{tresfuntores}.
Let $G\in \mathrm{Mod}(\mathcal{C})$, and $\dots \rightarrow P_{k}\rightarrow\dots\rightarrow P_{1}\rightarrow P_{0}\rightarrow G\rightarrow 0$ a projective resolution of $G$. For $1\leq  k\leq \infty $, the following conditions are equivalent.
\begin{enumerate}
\item [(a)] $P_{k}/\mathcal{I}P_{k}\rightarrow\dots\rightarrow P_{1}/\mathcal{I}P_{1}\rightarrow P_{0}/\mathcal{I}P_{0}\rightarrow G/\mathcal{I}G\rightarrow 0$ is the beginning of a projective resolution of  $G/\mathcal{I}G\in \mathrm{Mod}(\mathcal{C}/\mathcal{I})$.

\item [(b)] $\mathbb{TOR}_{i}^{\mathcal{C}}(\mathcal{C}/\mathcal{I},G)=0$ for $1\leq i\leq k$.

\item [(c)] For $F\in \mathrm{Mod}((\mathcal{C}/\mathcal{I})^{op})$ the morphisms given in \ref{morentor},  
$$\psi_{F,G}^{i}:\mathrm{Tor}^{\mathcal{C}}_{i}((\hat{\pi})_{\ast}F),G)\longrightarrow \mathrm{Tor}^{\mathcal{C}/\mathcal{I}}_{i}(F, G/\mathcal{I}G),$$
are isomorphisms for $1\leq i\leq k$.
\end{enumerate}
\end{proposition}

\section{Property $A$ and restriction of adjunctions}\label{sec:4}
In this section we will use some of the notions given in \ref{subsec:3}.  First we recall the following well-known result (see \cite{AusM1}).
\begin{proposition}\label{projfingen}
Let $\mathcal{C}$ be a variety and $\mathrm{proj}(\mathcal{C})$ the category of finitely generated projective $\mathcal{C}$-modules. Consider the Yoneda functor $\mathbb{Y}:\mathcal{C}\longrightarrow \mathrm{proj}(\mathcal{C})$ defined as $\mathbb{Y}(C):=\mathrm{Hom}_{\mathcal{C}}(C,-)$. Then $\mathbb{Y}$ is a contravariant functor which is full, faithful and dense.
\end{proposition}
Let $\mathcal{C}$ be a Hom-finite $R$-variety, we recall that $\mathrm{mod}(\mathcal{C})$ denotes the full subcategory of $\mathrm{Mod}(\mathcal{C})$ whose objects are the finitely presented functors (see definition \ref{defifinipre}). The aim of this section is to restrict the functors obtained in the last section.

\begin{proposition}\label{restri2fun}
Let $\mathcal{C}$ be a Hom-finite $R$-variety,  $\mathcal{I}$ an ideal in $\mathcal{C}$ and $\pi:C\longrightarrow \mathcal{C}/\mathcal{I}$ the canonical functor. 
Consider the  upper part of the diagram given in \ref{tresfuntores}.
\begin{enumerate}
\item [(a)] We can restrict $\pi^{\ast}$ to a functor $\pi^{\ast}:\mathrm{mod}(\mathcal{C})\longrightarrow \mathrm{mod}(\mathcal{C}/\mathcal{I})$. 

\item [(b)]  If for every $C\in \mathcal{C}$ there exists an epimorphism $\mathrm{Hom}_{\mathcal{C}}(C',-)\longrightarrow  \mathcal{I}(C,-)$, we can restrict the functor $\pi_{\ast}$ to a functor 
$\pi_{\ast}:\mathrm{mod}(\mathcal{C}/\mathcal{I})\longrightarrow \mathrm{mod}(\mathcal{C})$.

\item[(c)] If for every $C\in \mathcal{C}$ there exists an epimorphism $\mathrm{Hom}_{\mathcal{C}}(C',-)\longrightarrow  \mathcal{I}(C,-)$, we have the adjoint pair

$$\xymatrix{\mathrm{mod}(\mathcal{C}/\mathcal{I})\ar[rr]|{\pi_{\ast}}  &  &\mathrm{mod}(\mathcal{C}).\ar@<-2ex>[ll]_{\pi^{\ast}} }$$
\end{enumerate}
\end{proposition}
\begin{proof}
\begin{enumerate}
\item [(a)] Let us see that we have $\pi^{\ast}:\mathrm{mod}(\mathcal{C})\longrightarrow \mathrm{mod}(\mathcal{C}/\mathcal{I}).$  Indeed, we know that
$\pi^{\ast}:\mathrm{Mod}(\mathcal{C})\longrightarrow \mathrm{Mod}(\mathcal{C}/\mathcal{I})$ is right exact. Moreover, by the construction of $\pi^{\ast}$ it follows that $\pi^{\ast}(\mathrm{Hom}_{\mathcal{C}}(C,-))(C')=\frac{\mathcal{C}(-,C')}{\mathcal{I}(-,C')}\otimes \mathrm{Hom}_{\mathcal{C}}(C,-)=\frac{\mathcal{C}(C,C')}{\mathcal{I}(C,C')}=\mathrm{Hom}_{\mathcal{C}/\mathcal{I}}(C,C')$, and thus  $\pi^{\ast}(\mathrm{Hom}_{\mathcal{C}}(C,-))=\mathrm{Hom}_{\mathcal{C}/\mathcal{I}}(C,-)$ (see \ref{AProposition0}). From this we have the restriction $\pi^{\ast}:\mathrm{mod}(\mathcal{C})\longrightarrow \mathrm{mod}(\mathcal{C}/\mathcal{I})$. 

\item [(b)] Let us see that if $M\in \mathrm{mod}(\mathcal{C}/\mathcal{I})$, then $\pi_{\ast}(M)\in \mathrm{mod}(\mathcal{C})$. Indeed, let $M\in \mathrm{mod}(\mathcal{C}/\mathcal{I})$, then there exists an exact sequence $\mathrm{Hom}_{\mathcal{C}/\mathcal{I}}(X,-)\rightarrow \mathrm{Hom}_{\mathcal{C}/\mathcal{I}}(Y,-)\rightarrow  M\rightarrow  0$
with $X,Y\in \mathcal{C}/\mathcal{I}$. Applying $\pi_{\ast}$,  by \ref{descripcionMC} we have the following exact sequence 
$\frac{\mathrm{Hom}_{\mathcal{C}}(X,-)}{\mathcal{I}(X,-)}\rightarrow \frac{\mathrm{Hom}_{\mathcal{C}}(Y,-)}{\mathcal{I}(Y,-)}\rightarrow \pi_{\ast}(M)\rightarrow 0$.
We assert that  $\frac{\mathrm{Hom}_{\mathcal{C}}(X,-)}{\mathcal{I}(X,-)}$ is finitely presented for each $X\in \mathcal{C}$. To prove this we consider $\xymatrix{0\ar[r] & \mathcal{I}(X,-)\ar[r] & \mathrm{Hom}_{\mathcal{C}}(X,-)\ar[r] & \frac{\mathrm{Hom}_{\mathcal{C}}(X,-)}{\mathcal{I}(X,-)}\ar[r] & 0}$
By hypothesis we have that $\mathcal{I}(X,-)$ is finitely generated, then by theorem \cite[proposition 4.2(c)]{AusM1}, we have that $\frac{\mathrm{Hom}_{\mathcal{C}}(X,-)}{\mathcal{I}(X,-)}$ is finitely presented. Then by \cite[proposition 4.2(b)]{AusM1}, we conclude that $\pi_{\ast}(M)$ is finitely presented. 

\item [(c)] Follows from (a) and (b).
\end{enumerate}
\end{proof}

\begin{Remark}
Let $\mathcal{C}$ be a Hom-finite $R$-variety. We get following commutative diagram
$$\xymatrix{(\mathcal{C}/\mathcal{I},\mathrm{mod}(R))\ar[r]^{(\pi_{1})_{\ast}}\ar[d]^{\mathbb{D}_{\mathcal{C}/\mathcal{I}}} & (\mathcal{C},\mathrm{mod}(R))\ar[d]^{\mathbb{D}_{\mathcal{C}}}\\
((\mathcal{C}/\mathcal{I})^{op},\mathrm{mod}(R))\ar[r]^{(\pi_{2})_{\ast}} & (\mathcal{C}^{op},\mathrm{mod}(R))}$$
\end{Remark}

Consider the functor  $\Omega: \mathrm{Ann}(\mathcal{I})\longrightarrow \mathrm{Mod}(\mathcal{C}/\mathcal{I})$  defined in \ref{isomoan}, we know that $\Omega$ is an equivalence of categories with inverse $\pi_{\ast}:
\mathrm{Mod}(\mathcal{C}/\mathcal{I})\longrightarrow \mathrm{Ann}(\mathcal{I})$. 
Then we have the following proposition which tells us that we can restrict the functor $\pi_{\ast}$.

\begin{proposition}\label{annintermod}
Let $\mathcal{C}$ be a Hom-finite $R$-variety and $\mathcal{I}$ an ideal in $\mathcal{C}$ such that for each $C\in \mathcal{C}$ there exists an epimorphism $\mathrm{Hom}_{\mathcal{C}}(C',-)\longrightarrow  \mathcal{I}(C,-)$. Then, there exists an equivalence 
$\pi_{\ast}|_{\mathrm{mod}(\mathcal{C}/\mathcal{I})}:\mathrm{mod}(\mathcal{C}/\mathcal{I})\longrightarrow \mathrm{mod}(\mathcal{C})\cap \mathrm{Ann}(I).$
\end{proposition}
\begin{proof}
It is straightforward.
\end{proof}

Because of the last proposition we are now interested in ideals that satisfy the hypothesis of \ref {annintermod} . So we have the following definition.

\begin{definition}\label{propertyA}
Let $\mathcal{C}$ be a preadditive category. We say that an ideal $\mathcal{I}$ satisfies the $\textbf{property (A)}$ if  for every $C\in \mathcal{C}$ there exists epimorphisms $\mathrm{Hom}_{\mathcal{C}}(X,-)\longrightarrow \mathcal{I}(C,-)\longrightarrow 0$ and 
$\mathrm{Hom}_{\mathcal{C}}(-,Y)\longrightarrow \mathcal{I}(-,C)\longrightarrow 0.$
\end{definition}

The following result tells us that if the ideal $\mathcal{I}$  satisfies property $A$, then the category $\mathcal{C}/\mathcal{I}$ is dualizing provided $\mathcal{C}$ is also dualizing.

\begin{proposition}\label{cocienteduali}
Let $\mathcal{C}$ be a dualizing $R$-variety and $\mathcal{I}$ an ideal which satisfies property $A$. Then $\mathcal{C}/\mathcal{I}$ is a dualizing $R$-variety and the following diagram
$$\xymatrix{\mathrm{mod}(\mathcal{C}/\mathcal{I})\ar[r]^{(\pi_{1})_{\ast}}\ar[d]^{\mathbb{D}_{\mathcal{C}/\mathcal{I}}} & \mathrm{mod} (\mathcal{C})\ar[d]^{\mathbb{D}_{\mathcal{C}}}\\
\mathrm{mod}((\mathcal{C}/\mathcal{I})^{op})\ar[r]^{(\pi_{2})_{\ast}} &\mathrm{mod}(\mathcal{C}^{op})}$$
is commutative.
\end{proposition}
\begin{proof}
Let $\mathbb{D}_{\mathcal{C}}:\mathrm{mod}(\mathcal{C})\longrightarrow \mathrm{mod}(\mathcal{C}^{op})$ be the duality. It is enough to see that we have a functor $\mathbb{D}_{\mathcal{C}}:\mathrm{mod}(\mathcal{C})\cap \mathrm{Ann}(\mathcal{I})\longrightarrow \mathrm{mod}(\mathcal{C}^{op})\cap \mathrm{Ann}(\mathcal{I}^{op})$.\\
Indeed, let $M\in \mathrm{mod}(\mathcal{C})\cap \mathrm{Ann}(\mathcal{I})$ and consider $\mathbb{D}_{\mathcal{C}}(M)\in\mathrm{mod}(\mathcal{C}^{op})$. Let $f^{op}\in\mathcal{I}^{op}(B^{op},A^{op})$ then $f\in \mathcal{I}(A,B)$. Therefore, $\mathbb{D}(M)(f^{op})):=\mathrm{Hom}_{R}(M(f),I(R/r))=0$ since $M(f)=0$. Then $\mathcal{C}/\mathcal{I}$ is dualizing and $\mathbb{D}_{\mathcal{C}/\mathcal{I}}\simeq (\mathbb{D}_{\mathcal{C}})|_{\mathrm{mod}(\mathcal{C})\cap \mathrm{Ann}(\mathcal{I})}$. This way we have proved that the required diagram is commutative.
\end{proof}

The proof of the following result is similar to the one given in \cite[Proposition 2.7]{Yasuaki}. 

\begin{proposition}\label{restresfuntores}
Let $\mathcal{C}$ be a dualizing $R$-variety and $\mathcal{I}$ an ideal which satisfies property $A$.  Let $\pi_{1}:\mathcal{C}\longrightarrow \mathcal{C}/\mathcal{I}$ be the canonical functor, then we can restrict the diagram given in \ref{tresfuntores} to the finitely presented modules
$$\xymatrix{\mathrm{mod}(\mathcal{C}/\mathcal{I})\ar[rr]|{(\pi_{1}){\ast}}  &  &\mathrm{mod}(\mathcal{C})\ar@<-2ex>[ll]_{\pi_{1}^{\ast}}\ar@<2ex>[ll]^{\pi_{1}^{!}} }$$
\end{proposition}
\begin{proof}
The proof given in \cite[Proposition 2.7]{Yasuaki} works in this setting.
\end{proof}

Now, we give some examples where the property $A$ holds. We recall that 
the (Jacobson) $\textbf{radical}$ of an additive category $\mathcal{C}$ is the two-sided ideal $\mathrm{rad}_{\mathcal C}$ in $\mathcal{C}$ defined by the formula
$\mathrm{rad}_{\mathcal {C}}(X,Y)=\{h\in\mathcal{C}(X,Y)\mid 1_X-gh\text{ is invertible for any } g\in\mathcal{C}(Y,X)\}$ for all objects $X$ and $Y$ of $\mathcal{C}$.

\begin{proposition}
Let $\mathcal{C}$ be a dualizing $R$-variety and  $\mathcal{I}=\mathrm{rad}(\mathcal{C})(-,-)$  the radical ideal. Then $\mathcal{I}$ satisfies the property $A$.
\end{proposition}
\begin{proof}
See \cite[Prop. 2.10 (2)]{Iyama1} in p. 128.
\end{proof}

In order to give more examples of ideals satisfying the property $A$ we recall the following definition.

\begin{definition}\label{finitype}
Let $\mathcal{C}$ be small abelian $R$-category with the following properties.
\begin{enumerate}
\item [(a)] There is only a finite number of nonisomorphic simple objects in $\mathcal{C}$.
\item [(b)] Every object in $\mathcal{C}$ is of finite length.
\end{enumerate}
(It is well known that under this hypothesis $\mathcal{C}$ is a Krull-Schmidt category). It is said that $\mathcal{C}$ is of $\textbf{finite representation type}$ if $\mathcal{C}$  has only a finite number of non-isomorphic indecomposable objects (see \cite[p. 12]{AusM2}).
\end{definition}

Following the notation in  \cite[p. 3]{AusM2}, an object $M\in \mathrm{Mod}(\mathcal{C})$ is $\textbf{finite}$ if it is both noetherian an artinian. That is, if $M$ satisfies the ascending and descending chain condition on submodules.

\begin{proposition}
Let $\mathcal{C}$ be of finite representation type as in definition \ref{finitype}. Then  every ideal $\mathcal{I}(-,-)$ in $\mathcal{C}$ satisfies property $A$.
\end{proposition}
\begin{proof}
By \cite[3.6 (a) and (b)]{AusM2}, we have that $\mathrm{Mod}(\mathcal{C})$ and $\mathrm{Mod}(\mathcal{C}^{op})$ are locally finite. By \cite[3.1]{AusM2}, we have that each $\mathrm{Hom}_{\mathcal{C}}(C,-)$ and  $\mathrm{Hom}_{\mathcal{C}}(-,C)$ are finite for each $C\in \mathcal{C}$. Since the subcategory of finite modules is a Serre subcategory, we have that the submodules of $\mathrm{Hom}_{\mathcal{C}}(C,-)$  and $\mathrm{Hom}_{\mathcal{C}}(-,C)$ are finite. In particular, each $\mathcal{I}(C,-)$ and $\mathcal{I}(-,C)$ are finite. By \cite[Corollary 1.7]{AusM2}, we have that $\mathcal{I}(C,-)$ and $\mathcal{I}(-,C)$ are finitely generated. Then $\mathcal{I}$ satisfies property $A$.
\end{proof}

\begin{corollary}
If $\Lambda$ is an artin algebra of finite representation type, then every ideal in $\mathcal{C}=\mathrm{mod}(\Lambda)$ satisfies property $A$.
\end{corollary}

We recall the following notions. Let $\mathcal{A}$ be an arbitrary category and $\mathcal{B}$ a full subcategory in $\mathcal{A}$.  The full subcategory $\mathcal{B}$ is $\textbf{contravariantly finite}$ if for every $A\in \mathcal{A}$ there exists a morphism $f_{A}:B\longrightarrow A$ with $B\in \mathcal{B}$ such that if $f':B'\longrightarrow A$ is other morphism with $B'\in \mathcal{B}$, then there exist a morphism $g:B'\longrightarrow B$ such that $f'=f_{A}\circ g$. Dually, is defined the notion of $\textbf{covariantly finite}$. We say that $\mathcal{B}$ is $\textbf{functorially finite}$ if $\mathcal{B}$ is contravariantly finite and covariantly finite.\\

\begin{proposition}\label{euifunproA}
Let $\mathcal{C}$ be an additive category and $\mathcal{X}$ an additive full subcategory of $\mathcal{C}$. Let  $\mathcal{I}=\mathcal{I}_{\mathcal{X}}$  be the ideal of morphisms in $\mathcal{C}$ which factor through some object in $\mathcal{X}$. Then $\mathcal{I}$ satisfies property $A$ 
if and only if $\mathcal{X}$ is functorially finite in $\mathcal{C}$.
\end{proposition}
\begin{proof}
$(\Longleftarrow)$. Suppose that $\mathcal{X}$ is contravariantly finite. Then for each $C\in \mathcal{C}$, there exists a right $\mathcal{X}$-approximation $f_{C}:X\longrightarrow C$. Thus,  we have a morphism
$$\mathrm{Hom}_{\mathcal{C}}(-,f_{C}):\mathrm{Hom}_{\mathcal{C}}(-,X)\longrightarrow \mathrm{Hom}_{\mathcal{C}}(-,C).$$
We assert that $\mathrm{Im}(\mathrm{Hom}_{\mathcal{C}}(-,f_{C}))=\mathcal{I}(-,C)$. Indeed, let $C'\in \mathcal{C}$ and $\alpha\in \mathcal{I}(C',C)$. Since $\mathcal{I}=\mathcal{I}_{\mathcal{X}}$, there exists $X'\in \mathcal{X}$ and morphisms $\alpha':C'\longrightarrow X'$ and $\alpha'':X'\longrightarrow C$ such that $\alpha=\alpha''\alpha'$. Since $f_{C}$ is an $\mathcal{X}$-approximation, there exists $\beta:X'\longrightarrow X$ such that $\alpha''=f_{C}\beta$. Then
$\alpha=\alpha''\alpha'=f_{C}\beta\alpha'$. Then we have that $\alpha\in  \mathrm{Im}(\mathrm{Hom}_{\mathcal{C}}(-,f_{C})_{C'})$. Now, for $\gamma:C'\longrightarrow X$ we have that $(\mathrm{Hom}_{\mathcal{C}}(-,f_{C}))_{C'}(\gamma)=f_{C}\gamma\in \mathcal{I}(C',C)$ since $f_{C}\gamma$ factors through $X\in \mathcal{X}$ and $\mathcal{I}=\mathcal{I}_{\mathcal{X}}$, proving that $\mathrm{Im}(\mathrm{Hom}_{\mathcal{C}}(-,f_{C}))=\mathcal{I}(-,C)$. Thus,  there exists an epimorphism $\mathrm{Hom}_{\mathcal{C}}(-,f_{C}):\mathrm{Hom}_{\mathcal{C}}(-,X)\longrightarrow \mathcal{I}(-,C).$
Similarly, we can prove that if $\mathcal{X}$ is covariantly finite, then there exists and epimorphism $\mathrm{Hom}_{\mathcal{C}}(Y,-)\longrightarrow \mathcal{I}(C,-)\longrightarrow 0$ for each $C\in \mathcal{C}$. Therefore, we have that if $\mathcal{X}$ is functorially finite, then $\mathcal{I}$ satisfies property $A$.\\
The other implication is similar and the details are left to the reader.
\end{proof}

Now let us consider the $\textbf{transfinite radical}$ of $\mathcal{C}$ denoted by $\mathrm{rad}^{\ast}_{\mathcal{C}}(-,-)$ (see \cite{Stovi} for details).\\

\begin{proposition}
Let $\mathcal{C}$ be Hom-finite $R$-variety and suppose that  $\mathrm{rad}^{\ast}_{\mathcal{C}}(-,-)=0$. Let $\mathcal{I}$ be an idempotent ideal in $\mathcal{C}$ and let $\mathcal{X}=\{X\in \mathcal{C}\mid 1_{X}\in \mathcal{I}(X,X)\}.$
If $\mathcal{X}$ is functorially finite, then $\mathcal{I}$ satisfies property $A$.
\end{proposition}
\begin{proof}
Since $R$ is artinian and $\mathcal{C}$ is a Hom-finite $R$-variety, we have by  \cite{Stovi}, that $\mathcal{C}$ is Krull-Schmidt with local d.c.c on ideals (see \cite[Definition 5]{Stovi}). By \cite[Corollary 10]{Stovi}, we have that $\mathcal{I}=\mathcal{I}_{\mathcal{X}}$ where $\mathcal{X}=\{X\in \mathcal{C}\mid 1_{X}\in \mathcal{I}(X,X)\}.$ Now, if $\mathcal{X}$ is functorially finite, by \ref{euifunproA}, we have that $\mathcal{I}$ satisfies property $A$.
\end{proof}

\begin{example}
Let $\mathcal{C}=\mathrm{mod}(\Lambda)$ where  $\Lambda$ is a finite dimensional $K$-algebra over an algebraically closed field. If $\Lambda$ is a
standard selfinjective algebra of domestic representation type or $\Lambda$ is a special biserial algebra of domestic representation type, then  $\mathrm{rad}^{\ast}_{\mathcal{C}}(-,-)=0$ (see  \cite{Schroer}). We recall that  $\Lambda$ is of domestic representation type if there is a natural number $N$ such that for each dimension d, all but finitely many indecomposable modules of dimension d belong to at most $N $ one-parameter families.
\end{example}

\section{$k$-idempotent ideals}\label{sec:5}
In this section we will work in preadditive categories as well as in dualizing $R$-varieties. So, we will say explicitely in which context we are working on.\\
We introduce the definition of $k$-idempotent ideal in $\mathcal{C}$, which is the analogous to the one given by Aulander-Platzeck-Todorov in \cite{APG} for the case of artin algebras.
In order to do this we consider the morphisms given in \ref{morenexte}.
Then for $F,F'\in \mathrm{Mod}(\mathcal{C}/\mathcal{I})$ we have canonical morphisms
$$\varphi^{i}_{F,\pi_{\ast}(F')}:\mathrm{Ext}^{i}_{\mathrm{Mod}(\mathcal{C}/I)}(F,F')\longrightarrow \mathrm{Ext}^{i}_{\mathrm{Mod}(\mathcal{C})}(\pi_{\ast}(F),\pi_{\ast}(F')).$$

\begin{definition}\label{kidemcat}
Let $\mathcal{C}$ be a preadditive category  and $\mathcal{I}$ an ideal in $\mathcal{C}$.
\begin{enumerate}
\item [(a)] We say that $\mathcal{I}$ is $\textbf{k-idempotent}$ if 
$$\varphi^{i}_{F,\pi_{\ast}(F')}:\mathrm{Ext}^{i}_{\mathrm{Mod}(\mathcal{C}/I)}(F,F')\longrightarrow \mathrm{Ext}^{i}_{\mathrm{Mod}(\mathcal{C})}(\pi_{\ast}(F),\pi_{\ast}(F'))$$ is an isomorphism for all $F,F'\in \mathrm{Mod}(\mathcal{C}/\mathcal{I})$ and for all  $0\leq i\leq k$.

\item [(b)] We say that $\mathcal{I}$ is $\textbf{strongly idempotent}$ if 
$$\varphi^{i}_{F,\pi_{\ast}(F')}:\mathrm{Ext}^{i}_{\mathrm{Mod}(\mathcal{C}/I)}(F,F')\longrightarrow \mathrm{Ext}^{i}_{\mathrm{Mod}(\mathcal{C})}(\pi_{\ast}(F),\pi_{\ast}(F'))$$ is an isomorphism for all $F,F'\in \mathrm{Mod}(\mathcal{C}/\mathcal{I})$ and for all $0\leq i < \infty$.
\end{enumerate}
\end{definition}

We note that we have defined $k$-idempotent ideal in $\mathrm{Mod}(\mathcal{C})$, but this concept can also be defined in the category of finitely presented $\mathcal{C}$-modules $\mathrm{mod}(\mathcal{C})$. Next, we have a characterization of $k$-idempotent ideals in terms of the vanishing of certain derived functors.

\begin{proposition}\label{caractidem}
Let $\mathcal{C}$ be a preadditive category, $\mathcal{I}$ an ideal in $\mathcal{C}$ and $1\leq i\leq k$.  The following conditions are equivalent.
\begin{enumerate}
\item [(a)] $\mathcal{I}$ es $k$-idempotent
\item [(b)] $\varphi^{i}_{F,\pi_{\ast}(F')}:\mathrm{Ext}^{i}_{\mathrm{Mod}(\mathcal{C}/I)}(F,F')\longrightarrow \mathrm{Ext}^{i}_{\mathrm{Mod}(\mathcal{C})}(\pi_{\ast}(F),\pi_{\ast}(F'))$ is an isomorphism for all $F,F'\in \mathrm{Mod}(\mathcal{C}/\mathcal{I})$ and for all  $0\leq i\leq k$.

\item [(c)] $\mathbb{EXT}^{i}_{\mathcal{C}}(\mathcal{C}/\mathcal{I},F'\circ \pi)=0$ for $1\leq i\leq k$ and for $F'\in \mathrm{Mod}(\mathcal{C}/\mathcal{I})$.

\item [(d)] $\mathbb{EXT}^{i}_{\mathcal{C}}(\mathcal{C}/\mathcal{I},J\circ \pi)=0$ for $1\leq i\leq k$ and for each $J\in \mathrm{Mod}(\mathcal{C}/\mathcal{I})$ which is injective.

\end{enumerate}
\end{proposition}
\begin{proof}
The equivalences between (a), (b) and (c) are straightforward using \ref{Extiniciores}, and 
$(c)\Rightarrow (d)$ is trivial.\\
$(d)\Rightarrow (c)$. Let us see by induction on $i$ that  $\mathbb{EXT}^{i}_{\mathcal{C}}(\mathcal{C}/\mathcal{I},F'\circ \pi)=0$ for  all $F'\in \mathrm{Mod}(\mathcal{C}/\mathcal{I})$. Let us suppose that $i=1$ and $F'\in \mathrm{Mod}(\mathcal{C}/\mathcal{I})$. Consider the exact sequence
$$\xymatrix{0\ar[r] &F'\ar[r]^{\mu} & I\ar[r] & \frac{I}{F'}\ar[r] & 0}$$
where $I$ is an injective $\mathcal{C}$-module. Since $\overline{\mathrm{Tr}}_{\frac{\mathcal{C}}{\mathcal{I}}}\simeq \mathcal{C}(\frac{\mathcal{C}}{\mathcal{I}},-)$, and by the proof of the adjunction \ref{firstadjoin}, we have that $\eta:1_{\mathrm{Mod}(\mathcal{C}/\mathcal{I})}\longrightarrow \overline{\mathrm{Tr}}_{\frac{\mathcal{C}}{\mathcal{I}}}\circ \pi_{\ast}$ is an isomorphism.  Then we have the following commutative and exact diagram
$$\xymatrix{0\ar[r] & \mathcal{C}(\frac{\mathcal{C}}{\mathcal{I}},F'\circ \pi)\ar[r] &  \mathcal{C}(\frac{\mathcal{C}}{\mathcal{I}},I\circ \pi)\ar[r] & \mathcal{C}(\frac{\mathcal{C}}{\mathcal{I}},\frac{I}{F'}\circ \pi)\ar[r]^(.4){\delta} & \mathbb{EXT}^{1}_{\mathcal{C}}(\mathcal{C}/\mathcal{I},F'\circ \pi)\\
0\ar[r] &F'\ar[r]^{\mu}\ar[u] & I\ar[r]\ar[u] & \frac{I}{F'}\ar[r]\ar[u] & 0}$$
where the vertical morphisms are isomorphisms.
We conclude that $\delta=0$. Then we have the following exact sequence
$$\xymatrix{0 \ar[r] &  \mathbb{EXT}^{1}_{\mathcal{C}}(\mathcal{C}/\mathcal{I},F'\circ \pi)\ar[r]^(.35){} &   \mathbb{EXT}^{1}_{\mathcal{C}}(\mathcal{C}/\mathcal{I},I\circ \pi)\ar[r] &  \mathbb{EXT}^{1}_{\mathcal{C}}(\mathcal{C}/\mathcal{I},\frac{I}{F'}\circ \pi)\ar[r] & }$$ By hypothesis we have that $ \mathbb{EXT}^{1}_{\mathcal{C}}(\mathcal{C}/\mathcal{I},I\circ \pi)=0$, then we get that $\mathbb{EXT}^{1}_{\mathcal{C}}(\mathcal{C}/\mathcal{I},F'\circ \pi)=0$, proving the case $i=1$.\\
Now, let us suppose that $\mathbb{EXT}^{i-1}_{\mathcal{C}}(\mathcal{C}/\mathcal{I},N\circ \pi)=0$ for  all $N\in \mathrm{Mod}(\mathcal{C}/\mathcal{I})$.
Let $F'\in \mathrm{Mod}(\mathcal{C}/\mathcal{I})$. From the long exact homology sequence we have the exact sequence
$$\xymatrix{\mathbb{EXT}^{i-1}_{\mathcal{C}}(\mathcal{C}/\mathcal{I},\frac{I}{F'}\circ \pi)\ar[r]^(.35){} &   \mathbb{EXT}^{i}_{\mathcal{C}}(\mathcal{C}/\mathcal{I},F'\circ \pi)\ar[r] &  \mathbb{EXT}^{i}_{\mathcal{C}}(\mathcal{C}/\mathcal{I},I \circ \pi). }$$ 
Since $\frac{I}{F'}\in \mathrm{Mod}(\mathcal{C}/\mathcal{I})$, by induction we have that $\mathbb{EXT}^{i-1}_{\mathcal{C}}(\mathcal{C}/\mathcal{I},\frac{I}{F'}\circ \pi)=0$ and by hypothesis we have that 
$\mathbb{EXT}^{i}_{\mathcal{C}}(\mathcal{C}/\mathcal{I},I \circ \pi)=0$, then we conclude that $\mathbb{EXT}^{i}_{\mathcal{C}}(\mathcal{C}/\mathcal{I},F'\circ \pi)=0$, therefore, proving  the proposition.
\end{proof}

The following proposition tells us that we can restrict the result given in \ref{caractidem} to the category of finitely presented modules.

\begin{proposition}\label{caractidemfin}
Let $\mathcal{C}$ be a dualizing $R$-variety, $\mathcal{I}$ an ideal which satisfies property $A$ and $1\leq i\leq k$.  The following are equivalent.
\begin{enumerate}
\item [(a)] $\mathcal{I}$ es $k$-idempotent.
\item [(b)] $\varphi^{i}_{F,(\pi_{1})_{\ast}(F')}:\mathrm{Ext}^{i}_{\mathrm{mod}(\mathcal{C}/I)}(F,F')\longrightarrow \mathrm{Ext}^{i}_{\mathrm{mod}(\mathcal{C})}((\pi_{1})_{\ast}(F),(\pi_{1})_{\ast}(F'))$ is an isomorphism for all $F,F'\in \mathrm{mod}(\mathcal{C}/\mathcal{I})$ and for all  $0\leq i\leq k$.

\item [(c)] $\mathbb{EXT}^{i}_{\mathcal{C}}(\mathcal{C}/\mathcal{I},F'\circ \pi_{1})=0$ for $1\leq i\leq k$ and for $F'\in \mathrm{mod}(\mathcal{C}/\mathcal{I})$.

\item [(d)] $\mathbb{EXT}^{i}_{\mathcal{C}}(\mathcal{C}/\mathcal{I},J\circ \pi_{1})=0$ for $1\leq i\leq k$ and for each $J\in \mathrm{mod}(\mathcal{C}/\mathcal{I})$ which is injective.

\end{enumerate}
\end{proposition}
\begin{proof}
By \ref{restresfuntores}, we can restrict the diagram given in \ref{tresfuntores} to the category of finitely presented modules, then \ref{caractidem} holds for the case of finitely presented modules. 
\end{proof}

Now, we will work in the category $\mathrm{Mod}(\mathcal{C}^{op})$ and we will consider the corresponding canonical morphisms analogous to $\varphi^{i}_{F,\pi_{\ast}(F')}$, which we will denote by
 $$\delta^{i}_{F,(\pi_{2})_{\ast}(F')}:\mathrm{Ext}^{i}_{\mathrm{mod}((\mathcal{C}/\mathcal{I})^{op})}(F,F')\longrightarrow \mathrm{Ext}^{i}_{\mathrm{mod}(\mathcal{C}^{op})}\Big((\pi_{2})_{\ast}(F),(\pi_{2})_{\ast}(F')\Big)$$  for all $F,F'\in \mathrm{mod}((\mathcal{C}/\mathcal{I})^{op})$ and for all  $0\leq i\leq k$, where $\pi_{2}:\mathcal{C}^{op}\longrightarrow \mathcal{C}^{op}/\mathcal{I}^{op}$ is the projection. Therefore, we have that the result \ref{caractidem} holds for the category $\mathrm{Mod}(\mathcal{C}^{op})$.
 
\begin{proposition}
Let $\mathcal{C}$ be a dualizing $R$-variety  and $\mathcal{I}$ an ideal which satisfies property $A$. Then 
$\mathcal{I}$ is $k$-idempotent in $\mathcal{C}$ if and only if $\mathcal{I}^{op}$ is $k$-idempotent in $\mathcal{C}^{op}$.
\end{proposition}
\begin{proof}
Suppose that $\mathcal{I}$ is $k$-idempotent in $\mathcal{C}$.
Let us see that  $$\delta^{i}_{F,(\pi_{2})_{\ast}(F')}:\mathrm{Ext}^{i}_{\mathrm{mod}((\mathcal{C}/\mathcal{I})^{op})}(F,F')\longrightarrow \mathrm{Ext}^{i}_{\mathrm{mod}(\mathcal{C}^{op})}\Big((\pi_{2})_{\ast}(F),(\pi_{2})_{\ast}(F')\Big)$$  is an isomorphism for all $F,F'\in \mathrm{mod}((\mathcal{C}/\mathcal{I})^{op})$ and for all  $0\leq i\leq k$. By the proposition \ref{caractidemfin} it is enough to see that $\mathbb{EXT}^{i}_{\mathcal{C}^{op}}(\mathcal{C}^{op}/\mathcal{I}^{op},F'\circ \pi_{2})=0$ for $1\leq i\leq k$ and for $F'\in \mathrm{mod}(\mathcal{C}^{op}/\mathcal{I}^{op})$. Indeed, for $C\in \mathcal{C}^{op}/\mathcal{I}^{op}$ we have that
\begin{align*}
& \mathbb{EXT}^{i}_{\mathcal{C}^{op}}(\mathcal{C}^{op}/\mathcal{I}^{op},F'\circ \pi_{2})(C)=\\
& =\mathrm{Ext}^{i}_{\mathrm{mod}(\mathcal{C}^{op})}\left(\frac{\mathrm{Hom}_{\mathcal{C}}(-,C)}{\mathcal{I}(-,C)},(\pi_{2})_{\ast}(F')\right) \quad \quad \quad \quad \quad \quad \quad \quad \quad \quad \quad \quad \, \, \,\, [\text{see}\,\, \ref{descEXT}]\\
& =\mathrm{Ext}^{i}_{\mathrm{mod}(\mathcal{C}^{op})}\left((\pi_{2})_{\ast}\Big(\mathrm{Hom}_{\mathcal{C}/\mathcal{I}}(-,C)\Big),(\pi_{2})_{\ast}(F')\right) \quad \quad \quad \quad \quad \quad \quad \quad \quad \, [\text{see}\,\, \ref{descripcionMC}]\\
& \simeq \mathrm{Ext}^{i}_{\mathrm{mod}(\mathcal{C})}
\left(\mathbb{D}_{\mathcal{C}}^{-1}\Big((\pi_{2})_{\ast}(F')	\Big), \mathbb{D}_{\mathcal{C}}^{-1}\Big((\pi_{2})_{\ast}\Big(\mathrm{Hom}_{\mathcal{C}/\mathcal{I}}(-,C)\Big)\Big)\right) \quad [\mathbb{D}_{\mathcal{C}}\,\,\text{is a duality}] \\
& \simeq \mathrm{Ext}^{i}_{\mathrm{mod}(\mathcal{C})}
\left((\pi_{1})_{\ast}(\mathbb{D}_{\mathcal{C}/\mathcal{I}}^{-1}(F')), (\pi_{1})_{\ast}\Big(\mathbb{D}_{\mathcal{C}/\mathcal{I}}^{-1}\Big(\mathrm{Hom}_{\mathcal{C}/\mathcal{I}}(C,-)\Big)\Big)\right)\quad [\text{diagram in}\,\,\, \ref{cocienteduali}]\\
& \simeq \mathrm{Ext}^{i}_{\mathrm{mod}(\mathcal{C}/\mathcal{I})}
\left(\mathbb{D}_{\mathcal{C}/\mathcal{I}}^{-1}(F'), \mathbb{D}_{\mathcal{C}/\mathcal{I}}^{-1}\Big(\mathrm{Hom}_{\mathcal{C}/\mathcal{I}}(-,C)\Big)\right)\quad \quad  \quad \quad \, [\mathcal{I}\,\,\text{is}\,\,\text{k-f.p-idempotent}] \\
& =0\quad \quad \quad \quad \quad \quad \quad \quad \quad\quad \quad \quad \quad [ \mathbb{D}_{\mathcal{C}/\mathcal{I}}^{-1}(\mathrm{Hom}_{\mathcal{C}/\mathcal{I}}(-,C))\,\,\text{is injective in }\,\, \mathrm{mod}(\mathcal{C}/\mathcal{I}))]
\end{align*}
In the third equality we are using  $\mathrm{Ext}^{i}_{\mathrm{mod}(\mathcal{C})}(X,Y)\simeq \mathrm{Ext}^{i}_{\mathrm{mod}(\mathcal{C}^{op})}\Big(\mathbb{D}_{\mathcal{C}}(Y),\mathbb{D}_{\mathcal{C}}(X)\Big)$ for $X,Y\in \mathrm{mod}(\mathcal{C})$. Hence,  $\mathbb{EXT}^{i}_{\mathcal{C}{op}}(\mathcal{C}^{op}/\mathcal{I}^{op},F'\circ \pi_{2})=0$ for $1\leq i\leq k$ and for $F'\in \mathrm{mod}(\mathcal{C}^{op}/\mathcal{I}^{op})$, proving by \ref{caractidem} that $I^{op}$ is $k$-idempotent. The other implication is similar.\\
\end{proof}
Now, consider the morphism given in  \ref{morentor}. For $G=F'\circ \pi_{1}$ with $F'\in \mathrm{Mod}(\mathcal{C}/\mathcal{I})$ we have that $G/\mathcal{I}G\simeq F'$. Then for $F\in \mathrm{Mod}((\mathcal{C}/\mathcal{I})^{op})$  and  $F'\in \mathrm{Mod}(\mathcal{C}/\mathcal{I})$ we have the morphism $\psi_{F,(\pi_{1})_{\ast}(F')}^{i}:\mathrm{Tor}^{\mathcal{C}}_{i}(F\circ \pi_{2},F'\circ \pi_{1})\longrightarrow \mathrm{Tor}^{\mathcal{C}/\mathcal{I}}_{i}(F, F').$ 
The proof of the following two propositions are similar to \ref{caractidem} and \ref{caractidemfin}.

\begin{proposition}\label{psiisoeq} 
Let $\mathcal{C}$ be a preadditive category, $\mathcal{I}$ an ideal in $\mathcal{C}$ and $1\leq i\leq k$.  The following  conditions are equivalent.
\begin{enumerate}
\item [(a)] $\psi_{F,(\pi_{1})_{\ast}(F')}^{i}:\mathrm{Tor}^{\mathcal{C}}_{i}(F\circ \pi_{2},F'\circ \pi_{1})\longrightarrow \mathrm{Tor}^{\mathcal{C}/\mathcal{I}}_{i}(F, F')$ is an isomorphism for all $1\leq i\leq k$,  $F\in \mathrm{Mod}((\mathcal{C}/\mathcal{I})^{op})$ and  $F'\in \mathrm{Mod}(\mathcal{C}/\mathcal{I})$.

\item [(b)] $\mathbb{TOR}_{i}^{\mathcal{C}}(\mathcal{C}/\mathcal{I},F'\circ\pi_{1})=0$ for $1\leq i\leq k$  and for all $F'\in \mathrm{Mod}(\mathcal{C}/\mathcal{I})$.

\item [(c)] $\mathbb{TOR}_{i}^{\mathcal{C}}(\mathcal{C}/\mathcal{I},P\circ\pi_{1})=0$ for $1\leq i\leq k$ and for all $P\in \mathrm{Mod}(\mathcal{C}/\mathcal{I})$ that is projective.
\end{enumerate}
\end{proposition}
This result can be restricted  to the category of finitely presented modules.
\begin{proposition}\label{psiisoeqfin} 
Let $\mathcal{C}$ be a dualizing $R$-variety, $\mathcal{I}$ an ideal which satisfies property $A$  and $1\leq i\leq k$.  The following are equivalent.
\begin{enumerate}
\item [(a)] $\psi_{F,(\pi_{1})_{\ast}(F')}^{i}:\mathrm{Tor}^{\mathcal{C}}_{i}(F\circ \pi_{2},F'\circ \pi_{1})\longrightarrow \mathrm{Tor}^{\mathcal{C}/\mathcal{I}}_{i}(F, F')$ is an isomorphism for all $1\leq i\leq k$, $F\in \mathrm{mod}((\mathcal{C}/\mathcal{I})^{op})$ and  $F'\in \mathrm{mod}(\mathcal{C}/\mathcal{I})$.

\item [(b)] $\mathbb{TOR}_{i}^{\mathcal{C}}(\mathcal{C}/\mathcal{I},F'\circ\pi_{1})=0$ for $1\leq i\leq k$  and for all $F'\in \mathrm{mod}(\mathcal{C}/\mathcal{I})$.

\item [(c)] $\mathbb{TOR}_{i}^{\mathcal{C}}(\mathcal{C}/\mathcal{I},\mathrm{Hom}_{\mathcal{C}/\mathcal{I}}(C,-)\circ\pi_{1})=0$ for $1\leq i\leq k$ and for all $\mathrm{Hom}_{\mathcal{C}/\mathcal{I}}(C,-)\in \mathrm{mod}(\mathcal{C}/\mathcal{I})$.
\end{enumerate}
\end{proposition}

Now, in order to relate the functors $\mathbb{EXT}^{i}_{\mathcal{C}}(\mathcal{C}/\mathcal{I},-)$ and $\mathbb{TOR}_{i}^{\mathcal{C}}(\mathcal{C}/\mathcal{I},-)$ we need the Auslander-Reiten duality. We have the following result due to Auslander and Reiten.

\begin{proposition}\label{ARDuality}
Let $\mathcal{C}$ be a dualizing $R$-variety and $M\in \mathrm{mod}(\mathcal{C})$. Then we have an isomorphism of contravariant functors from $\mathrm{mod}(\mathcal{C}^{op})$ to $\mathrm{mod}(R)$
$$\mathbb{D}_{\mathrm{mod}(\mathcal{C}^{op})}(\mathrm{Tor}_{i}^{\mathcal{C}}(-, M))\simeq \mathrm{Ext}^{i}_{\mathrm{mod}(\mathcal{C}^{op})}(-,\mathbb{D}_{\mathcal{C}}(M)).$$
\end{proposition}
\begin{proof}
See \cite[Proposition 7.3]{AusVarietyI} in p. 341.
\end{proof}

Now, we have the following result that characterizes $k$-idempotent ideals in terms of the morphisms $\psi_{F,(\pi_{1})_{\ast}(F')}^{i}$.
\begin{proposition}\label{isopsiidem}
Let $\mathcal{C}$ be a dualizing $R$-variety and $\mathcal{I}$ an ideal which satisfies property $A$. Then $\mathcal{I}$ is $k$-idempotent if and only if $\psi_{F,(\pi_{1})_{\ast}(F')}^{i}$ is an isomorphism for every $F\in \mathrm{mod}((\mathcal{C}/\mathcal{I})^{op})$  and for every  $F'\in \mathrm{mod}(\mathcal{C}/\mathcal{I})$.
\end{proposition}
\begin{proof}
$(\Longrightarrow)$. Let $F'\in \mathrm{mod}(\mathcal{C}/\mathcal{I})$.
Let us see that $\mathbb{TOR}_{i}^{\mathcal{C}}(\mathcal{C}/\mathcal{I},F'\circ\pi_{1})=0$ for $1\leq i\leq k$.  Indeed, by using \ref{ARDuality} and the fact that $\mathcal{I}$ is $k$-idempotent,  for $C\in\mathcal{C}/\mathcal{I}$ we have 
\begin{align*}
& \mathbb{TOR}_{i}^{\mathcal{C}}(\mathcal{C}/\mathcal{I},F'\circ \pi_{1} )(C)\simeq \mathrm{Hom}_{R}\Big(\mathrm{Ext}_{\mathrm{Mod}(\mathcal{C}/\mathcal{I})}^{i}\Big(F',\mathbb{D}_{\mathcal{C}/\mathcal{I}}^{-1}\Big(\mathrm{Hom}_{\mathcal{C}/\mathcal{I}}(-,C)\Big)\Big),E\Big)= 0,
\end{align*}
where the last equality is because $\mathrm{Ext}_{\mathrm{Mod}(\mathcal{C}/\mathcal{I})}^{i}\Big(F',\mathbb{D}_{\mathcal{C}/\mathcal{I}}^{-1}\Big(\mathrm{Hom}_{\mathcal{C}/\mathcal{I}}(-,C)\Big)\Big)=0$ since the functor $\mathbb{D}_{\mathcal{C}/\mathcal{I}}^{-1}\Big(\mathrm{Hom}_{\mathcal{C}/\mathcal{I}}(-,C)\Big)$ is injective in $\mathrm{mod}(\mathcal{C}/\mathcal{I})$.
Therefore, by \ref{psiisoeqfin}  we have that  $\psi_{F,(\pi_{1})_{\ast}(F')}^{i}$ is isomorphism.
The other implication is similar.
\end{proof}

We finish this section with the following result, which is analogous to proposition 1.3 in \cite{APG}. 

\begin{corollary}\label{caractidemfin2}
Let $\mathcal{C}$ be a dualizing $R$-variety, $\mathcal{I}$ an ideal which satisfies property $A$. For $1\leq i\leq k$,  the following are equivalent.
\begin{enumerate}
\item [(a)] $\mathcal{I}$ es $k$-idempotent.
\item [(b)] $\varphi^{i}_{F,(\pi_{1})_{\ast}(F')}:\mathrm{Ext}^{i}_{\mathrm{mod}(\mathcal{C}/I)}(F,F')\longrightarrow \mathrm{Ext}^{i}_{\mathrm{mod}(\mathcal{C})}((\pi_{1})_{\ast}(F),(\pi_{1})_{\ast}(F'))$ is an isomorphism for all $F,F'\in \mathrm{mod}(\mathcal{C}/\mathcal{I})$ and for all  $0\leq i\leq k$.

\item [(c)] $\mathbb{EXT}^{i}_{\mathcal{C}}(\mathcal{C}/\mathcal{I},F'\circ \pi_{1})=0$ for $1\leq i\leq k$ and for $F'\in \mathrm{mod}(\mathcal{C}/\mathcal{I})$.

\item [(d)] $\mathbb{EXT}^{i}_{\mathcal{C}}(\mathcal{C}/\mathcal{I},J\circ \pi_{1})=0$ for $1\leq i\leq k$ and for each $J\in \mathrm{mod}(\mathcal{C}/\mathcal{I})$ which is injective.

\item [(e)] $\psi_{F,(\pi_{1})_{\ast}(F')}^{i}:\mathrm{Tor}^{\mathcal{C}}_{i}(F\circ \pi_{2},F'\circ \pi_{1})\longrightarrow \mathrm{Tor}^{\mathcal{C}/\mathcal{I}}_{i}(F, F')$ is an isomorphism for all $1\leq i\leq k$ and $F\in \mathrm{mod}((\mathcal{C}/\mathcal{I})^{op})$ and  $F'\in \mathrm{mod}(\mathcal{C}/\mathcal{I})$.

\item [(f)] $\mathbb{TOR}_{i}^{\mathcal{C}}(\mathcal{C}/\mathcal{I},F'\circ\pi_{1})=0$ for $1\leq i\leq k$  and for all $F'\in \mathrm{mod}(\mathcal{C}/\mathcal{I})$.

\item [(g)] $\mathbb{TOR}_{i}^{\mathcal{C}}(\mathcal{C}/\mathcal{I},\mathrm{Hom}_{\mathcal{C}/\mathcal{I}}(C,-)\circ\pi_{1})=0$ for $1\leq i\leq k$ and for all $\,$ $\mathrm{Hom}_{\mathcal{C}/\mathcal{I}}(C,-)\in \mathrm{mod}(\mathcal{C}/\mathcal{I})$.

\end{enumerate}
\end{corollary}
\begin{proof}
It follows from \ref{caractidemfin},  \ref{isopsiidem} and \ref{psiisoeqfin}.
\end{proof}

\section{Projective resolutions of $k$-idempotent ideals}\label{sec:6}

In this section we will work in preadditive categories as well as in dualizing $R$-varieties. So, we will say explicitely in which context we are working on.\\
In the previous section we characterized $k$-idempotent ideals in terms of
the projective resolutions of all $\mathcal{C}/\mathcal{I}$-modules. We will show here that knowing the projective resolutions of $\mathcal{I}(C,-)$ for all $C\in \mathcal{C}$  is enough to determine for which $k$ the ideal $\mathcal{I}$ is $k$-idempotent.\\

Next, we present a generalization of the so-called dual basis lemma  in classical ring theory (see \cite{Lamlec}).

\begin{proposition}(Dual basis Lemma)\label{dualbasislema}
Let $\mathcal{C}$ be a preadditive category. An object $P\in \mathrm{Mod}(\mathcal{C})$ is projective if and only if there exists a family of morphisms $\{\beta_{j}:P\longrightarrow \mathrm{Hom}_{\mathcal{C}}(C_{j},-)\}_{j\in J}$ and a family  $\{x_{j}\}_{j\in J}$ with $x_{j}\in P(C_{j})$ such that for all $X\in \mathcal{C}$ and for every $a\in P(X)$ there exists a finite subset $J_{X,a}\subseteq J$  such that
$$a=\sum_{j\in J_{X,a}}P([\beta_{j}]_{X}(a))(x_{j}).$$
\end{proposition}
\begin{proof}
$(\Longrightarrow)$
Since $\{\mathrm{Hom}_{\mathcal{C}}(C,-)\}_{C\in \mathcal{C}}$ is a generating set of projective modules, there exists an epimorphism $f:\bigoplus_{j\in J}\mathrm{Hom}_{\mathcal{C}}(C_{j},-)\longrightarrow P.$
We get the morphisms $\eta_{j}:=f u_{j}:\mathrm{Hom}_{\mathcal{C}}(C_{j},-)\longrightarrow P$, where $u_{j}:\mathrm{Hom}_{\mathcal{C}}(C_{j},-)$ $\longrightarrow \bigoplus_{j\in J}\mathrm{Hom}_{\mathcal{C}}(C_{j},-)$ is the $j$-th inclusion. By Yoneda lemma, $f u_{j}$ corresponds to one element $x_{j}:=[\eta_{j}]_{C_{j}}(1_{C_{j}})\in P(C_{j})$; furthermore,  $\eta_{j}:\mathrm{Hom}_{\mathcal{C}}(C_{j},-)\longrightarrow P$ is such that $[\eta_{j}]_{X}(\alpha)=P(\alpha)(x_{j})$ for all $X\in \mathcal{C}$.\\ 
Now, let $\gamma=(\gamma_{j})_{j\in J}\in \bigoplus_{j\in J}\mathrm{Hom}_{\mathcal{C}}(C_{j},X)$. Then, there exists a finite subset $J_{\gamma}$ of $J$ such that $\gamma_{j}=0$ if $j\notin J_{\gamma}$. We know that $f_{X}:\bigoplus_{j\in J}\mathrm{Hom}_{\mathcal{C}}(C_{j},X)\longrightarrow P(X)$ is defined for $\gamma=(\gamma_{j})_{j\in J}\in \bigoplus_{j\in J}\mathrm{Hom}_{\mathcal{C}}(C_{j},X)$ as follows: $f_{X}((\gamma_{j})_{j\in J})=\sum_{j\in J_{\gamma}}[\eta_{j}]_{X}(\gamma_{j})=\sum_{j\in J_{\gamma}}P(\gamma_{j})(x_{j}).$
Now, since $P$ is projective we have that $f$ is a split epimorphism, and then there exists $g:P\longrightarrow \bigoplus_{j\in J}\mathrm{Hom}_{\mathcal{C}}(C_{j},-)$ such that $fg=1_{P}$. Let us consider the projection
$\pi_{j}: \bigoplus_{j\in J}\mathrm{Hom}_{\mathcal{C}}(C_{j},-)\longrightarrow \mathrm{Hom}_{\mathcal{C}}(C_{j},-)$, then we have $\beta_{j}:=\pi_{j} g:P\longrightarrow \mathrm{Hom}_{\mathcal{C}}(C_{j},-).$ Then for $X\in \mathcal{C}$ we have that 
$g_{X}:P(X)\longrightarrow \bigoplus_{j\in  J}\mathrm{Hom}_{\mathcal{C}}(C_{j},X)$ is defined  as follows: 
$$g_{X}(a)=([\beta_{j}]_{X}(a))_{j\in J} \,\, \forall a\in P(X)$$ where $[\beta_{j}]_{X}(a):C_{j}\longrightarrow X$. Since $g_{X}(a)\in \bigoplus_{j\in J}\mathrm{Hom}_{\mathcal{C}}(C_{j},X)$, there exists a finite subset $J_{X,a}\subseteq J$ such that $[\beta_{j}]_{X}(a)=0$ if $j\notin J_{X,a}$. Then
$$a=f_{X}(g_{X}(a))=f_{X}\Big(([\beta_{j}]_{X}(a))_{j\in J}\Big)=\sum_{j\in J_{X,a}}P([\beta_{j}]_{X}(a))(x_{j}).$$
The other implication is similar and is left to the reader.
\end{proof}

We recall that  given a family of objects  $\mathcal{F}=\{F_{i}\}_{i\in I}$ and $M\in \mathrm{Mod}(\mathcal{C})$, in \ref{definitraza} we defined the $\textbf{trace}$ in $M$ of the family $\mathcal{F}$ which is denoted by $\mathrm{Tr}_{\mathcal{F}}(M)$. We have the following description of the trace.

\begin{Remark}\label{descritraxza}
Let $\mathcal{F}=\lbrace F_{i}\rbrace_{i\in I} $ be a family in $ \mathrm{Mod}(\mathcal{C})$.  For each $N\in \mathrm{Mod}(\mathcal{C})$ and  $X\in \mathcal{C}$ we have that $\mathrm{Tr}_{\mathcal{F}}(N)(X)=\!\!\!\!\!\!
\!\!\!\!\!\!\displaystyle\sum _{\{f\in\mathrm{Hom}(F, N)\;\mid\; F\in \mathcal{F}\}}
\!\!\!\!\!\!\!\!\!\!\!\!\!\!\!\!\!\!\mathrm{Im}(f_{X}).$ In the case $\mathcal{F}=\{F\}$ is just one object, we will write $\mathrm{Tr}_{F}N$.
\end{Remark}

We recall the following definition (see section 1.3 in \cite{Parra}).

\begin{definition}\label{deftraceideal}
Let $\mathcal{C}$ be a preadditive category and $\mathcal{F}=\lbrace F_{i}\rbrace_{i\in I} $ a family of objects in $ \mathrm{Mod}(\mathcal{C})$. For each $C\in \mathcal{C}$ consider the $\mathcal{C}$-submodule $\Tr_{\mathcal{F}}(\mathrm{Hom}_{\mathcal{C}}(C,-))$ of $\mathrm{Hom}_{\mathcal{C}}(C,-)$. We define the subfunctor  $\mathrm{Tr}_{\mathcal{F}}\mathcal{C}$ of
$\mathrm{Hom}_{\mathcal{C}}(-,-):\mathcal{C}^{op}\times \mathcal{C}\longrightarrow \mathbf{Ab}$ as follows:
$$(\mathrm{Tr}_{\mathcal{F}}\mathcal{C})(C,C'):=\mathrm{Tr}_{\mathcal{F}}(\mathrm{Hom}_{\mathcal{C}}(C,-))(C')$$ for all 
$C,C'\in \mathcal{C}$. This ideal will be called $\textbf{trace ideal}$. In the case that $\mathcal{F}=\{P\}$ with $P$ a projective $\mathcal{C}$-module, we will write $\mathrm{Tr}_{P}\mathcal{C}$.
\end{definition}

It is easy to see that $\mathrm{Tr}_{\mathcal{F}}\mathcal{C}(-,-)$ is a subfunctor of $\mathrm{Hom}_{\mathcal{C}}(-,-)$ and thus an ideal.\\

\begin{proposition}\label{TrPidempo}
Let $\mathcal{C}$ be a preadditive category and let $P$ be a projective $\mathcal{C}$-module. Then 
$\mathrm{Tr}_{P}\mathcal{C}$ defines an idempotent ideal of $\mathcal{C}$.
\end{proposition}
\begin{proof}
By using \ref{dualbasislema}, we can adapt the classical proof given \cite[Proposition 2.40]{Lamlec}; or see Lemma 2.5 in \cite{Parra}.
\end{proof}
Since $\mathrm{Tr}_{P}\mathcal{C}$ is an ideal,  we can define  $(\mathrm{Tr}_{P}\mathcal{C}\cdot F)(X):=
\sum_{f\in \bigcup_{C\in \mathcal{C}}(\mathrm{Tr}_{P}\mathcal{C})(C,X)}\mathrm{Im} (F(f))$ (see 
\ref{prodIfunc}). Now, we have the following result which is a generalization of the basic result in modules over a ring $R$.

\begin{proposition}\label{TRismultipli} 
Let $\mathcal{C}$ be a preadditive category, $F\in \mathrm{Mod}(\mathcal{C})$ and $\mathrm{Tr}_{P}\mathcal{C}$ the trace ideal. Then $\mathrm{Tr}_{P}\mathcal{C}\cdot F=\mathrm{Tr}_{P}(F).$
\end{proposition}
\begin{proof}
See Proposition 2.7 in \cite{Parra}.
\end{proof}

Let $M\in \mathrm{Mod}(\mathcal{C})$. We recall that $\mathrm{add}(M)$ is the full subcategory of $\mathrm{Mod}(\mathcal{C})$ whose objects are direct summands of finite coproducts of the module $M$. That is, $X\in \mathrm{add}(M)$ if and only if there exists a module $Y$ such that $X\oplus Y\simeq M^{n}$ for some $n\in \mathbb{N}$. The following proposition tells us when two finitely generated projective $\mathcal{C}$-modules produce the same ideal.

\begin{proposition}
Let $P$ and $Q$ finitely generated projective $\mathcal{C}$-modules. Then $\mathrm{Tr}_{P}\mathcal{C}=\mathrm{Tr}_{Q}\mathcal{C}$ if and only if $\mathrm{add}(P)=\mathrm{add}(Q)$.
\end{proposition}
\begin{proof}
Suppose that  $\mathrm{Tr}_{P}\mathcal{C}=\mathrm{Tr}_{Q}\mathcal{C}$. We have that $\mathrm{Tr}_{P}(P)=P$. Since $\mathrm{Tr}_{P}\mathcal{C}=\mathrm{Tr}_{Q}\mathcal{C}$ by \ref{TRismultipli}, we have that $P=\mathrm{Tr}_{P}P=\mathrm{Tr}_{P}\mathcal{C}\cdot P=\mathrm{Tr}_{Q}\mathcal{C}\cdot P=\mathrm{Tr}_{Q}P$. Then there exists an epimorphism $\eta:Q^{(I)}\longrightarrow P$. Since $P$ is finitely generated, we have that there exists a finite subset $J\subseteq I$  and an epimorphism $\eta':Q^{J}\longrightarrow P$ (see \cite[2.1(b)]{AusM1}). Since $P$ is projective, we have that $P$ is a direct summand of $Q^{J}$. Then $P\in \mathrm{add}(Q)$. Similarly, $Q\in \mathrm{add}(P)$, and thus $\mathrm{add}(P)=\mathrm{add}(Q)$. The other implication is similar.
\end{proof}

\begin{lemma}\label{addnoadd}
Let $\mathcal{C}$ be an additive category and $P=\mathrm{Hom}_{\mathcal{C}}(C,-)\in \mathrm{Mod}(\mathcal{C})$. Let us consider $\mathcal{B}:=\mathrm{add}(C)\subseteq \mathcal{C}$  and $\mathcal{F}:=\{\mathrm{Hom}_{\mathcal{C}}(C',-)\}_{C'\in \mathcal{B}}$, then for each $\mathrm{Hom}_{\mathcal{C}}(X,-)\in \mathrm{Mod}(\mathcal{C})$ we have that 
$\mathrm{Tr}_{\mathcal{F}}\Big(\mathrm{Hom}_{\mathcal{C}}(X,-)\Big)=\mathrm{Tr}_{P}\Big(\mathrm{Hom}_{\mathcal{C}}(X,-)\Big).$
\end{lemma}
\begin{proof}
Straightforward.
\end{proof}

\begin{corollary}\label{Traza=Iadd}
Let $\mathcal{C}$ be a Hom-finite $R$-variety and $P=\mathrm{Hom}_{\mathcal{C}}(C,-)\in \mathrm{Mod}(\mathcal{C})$. Then $\mathrm{Tr}_{P}\Big(\mathrm{Hom}_{\mathcal{C}}(X,-)\Big)=\mathcal{I}_{\mathrm{add}(C)}(X,-)$, where $\mathcal{I}_{\mathrm{add}(C)}$ is the ideal of the morphisms in $\mathcal{C}$ which factor through some object of $\mathrm{add}(C)$. 
\end{corollary}
\begin{proof}
It follows from \cite[Lemma 2.3]{MOSS} and \ref{addnoadd}.
\end{proof}
Now, we recall the following well known result.
\begin{proposition}\label{addfun}
Let $\mathcal{C}$ be an $R$-category which is Hom-finite, where $R$ is a commutative ring. Then for every $C\in \mathcal{C}$ we have that $\mathrm{add}(C)$ is functorially finite.
\end{proposition}
\begin{proof}
The proof given in \cite[Theorem 4.2]{AusSmalo} can be adapted to this setting.
\end{proof}
In the following results we focus in projective modules of the form $\mathrm{Hom}_{\mathcal{C}}(C,-)$ since in $R$-varieties all the finitely generated projectives are of such form (see \ref{projfingen}).
\begin{proposition}\label{trasatproA}
Let $\mathcal{C}$ be a dualizing $R$-variety, $P=\mathrm{Hom}_{\mathcal{C}}(C,-)$ a finitely generated projective $\mathcal{C}$-module and $\mathcal{I}=\mathrm{Tr}_{P}\mathcal{C}$.  Then $\mathcal{I}$ satisfies property $A$.
\end{proposition}
\begin{proof}
By \ref{Traza=Iadd}, we have that $\mathrm{Tr}_{P}\mathcal{C}=\mathcal{I}_{\mathrm{add}(C)}.$ By \ref{addfun}, we have that  $\mathrm{add}(C)$ is funtorially finite. By \ref{euifunproA}, we have that $\mathcal{I}_{\mathrm{add}(C)}$ satisfies the property $A$.
\end{proof}

\begin{corollary}\label{almostrecolle}
Let $\mathcal{C}$ be a dualizing $R$-variety, $P=\mathrm{Hom}_{\mathcal{C}}(C,-)\in \mathrm{mod}(\mathcal{C})$ and $\mathcal{I}=\mathrm{Tr}_{P}\mathcal{C}$.  Then we can restrict the diagram given in \ref{tresfuntores} to the finitely presented modules
$$\xymatrix{\mathrm{mod}(\mathcal{C}/\mathcal{I})\ar[rr]|{(\pi_{1}){\ast}}  &  &\mathrm{mod}(\mathcal{C})\ar@<-2ex>[ll]_{\pi_{1}^{\ast}}\ar@<2ex>[ll]^{\pi_{1}^{!}} }$$
\end{corollary}
\begin{proof}
It follows by \ref{trasatproA} and \ref{restresfuntores}.
\end{proof}
For a preadditive category $\mathcal{C}$, we recall the construction of the functor 
$(-)^{\ast}:\mathrm{Mod}(\mathcal{C})\longrightarrow \mathrm{Mod}(\mathcal{C}^{op})$
which is a generalization of the functor $\mathrm{Mod}(A)\longrightarrow \mathrm{Mod}(A^{op})$ given by $M\mapsto \mathrm{Hom}_{A}(M,A)$ for all the $A$-modules $M$, where $A$ is a ring.\\
Indeed, for each $\mathcal{C}$-module $M$ we define $M^{\ast}:\mathcal{C}\longrightarrow \mathbf{Ab}$ given by $M^{\ast}(C)=\mathrm{Hom}_{\mathrm{Mod}(\mathcal{C})}(M,\mathrm{Hom}_{\mathcal{C}}(C,-))$. Clearly $M^{\ast}$ is a $\mathcal{C}^{op}$-module. In this way we obtain a contravariant functor $(-)^{\ast}:\mathrm{Mod}(\mathcal{C})\longrightarrow \mathrm{Mod}(\mathcal{C}^{op})$ given by $M\mapsto M^{\ast}$.\\
If $M=\mathrm{Hom}_{\mathcal{C}}(C,-)$, it can be seen that $M^{\ast}\simeq \mathrm{Hom}_{\mathcal{C}}(-,C)$; we refer the reader to section 6 in \cite{AusVarietyI} for more details.

\begin{corollary}\label{PPast}
Let $\mathcal{C}$ be a Hom-finite $R$-variety, $P=\mathrm{Hom}_{\mathcal{C}}(C,-)\in \mathrm{Mod}(\mathcal{C})$ and consider the ideal $\mathcal{I}=\mathrm{Tr}_{P}\mathcal{C}$. Then we have that $\mathcal{I}^{op}=\mathrm{Tr}_{P^{\ast}}\mathcal{C}^{op}.$
\end{corollary}
\begin{proof}
It follows from  \ref{Traza=Iadd} and the fact that $P^{\ast}=\mathrm{Hom}_{\mathcal{C}}(-,C)$.
\end{proof}
For the next, we recall that if $\mathcal{C}$ is a dualizing $R$-variety by  \cite[Proposition 3.4]{AusVarietyI} we have that $\mathrm{mod}(\mathcal{C})$ has projective covers.

\begin{Remark}\label{porocoverad}
Let $\mathcal{C}$ be a dualizing $R$-variety, $P$  a projective $\mathcal{C}$-module and  $F\in \mathrm{mod}(\mathcal{C})$. Let $P_{0}(F)$ be the projective cover of $F$, then $\mathrm{Tr}_{P}(F)=F$ if and only if $P_{0}(F)\in \mathrm{add}(P)$.
\end{Remark}

We introduce the following definition that will be used along  the paper. 

\begin{definition}\label{defiPk}
Let $\mathcal{C}$ be a dualizing $R$-variety and $P\in \mathrm{mod}(\mathcal{C})$ be a projective module. For each $0\leq k\leq \infty$ we define $\mathbb{P}_{k}$ to be the full subcategory of $\mathrm{mod}(\mathcal{C})$ consisting of the $\mathcal{C}$-modules $X$ having a projective resolution
$$\xymatrix{\cdots P_{n}\ar[r] & P_{n-1}\ar[r] &\cdots\ar[r] & P_{1}\ar[r] & P_{0}\ar[r] & X\ar[r] & 0}$$
with $P_{i}\in \mathrm{add}(P)$ for $0\leq i\leq k$.
\end{definition}

We have the following easy lemma.

\begin{lemma}\label{percero}
Let $\mathcal{C}$ be a Hom-finite $R$-variety,   $P=\mathrm{Hom}_{\mathcal{C}}(C,-)\in \mathrm{Mod}(\mathcal{C})$ and $\mathcal{I}=\mathrm{Tr}_{P}\mathcal{C}$. Consider $\pi_{1}:\mathcal{C}\longrightarrow \mathcal{C}/\mathcal{I}$. Then 
\begin{enumerate}
\item [(a)] $\mathrm{Hom}_{\mathrm{mod}(\mathcal{C})}(Q,(\pi_{1})_{\ast}(Y))=0$ for all $Q\in \mathrm{add}(P)$ and for all $Y\in \mathrm{mod}(\mathcal{C}/\mathcal{I})$.

\item [(b)] $X\in \mathbb{P}_{0}$ if and only if $\mathrm{Hom}_{\mathrm{mod}(\mathcal{C})}(X,(\pi_{1})_{\ast}(Y))=0$ for all $Y\in \mathrm{mod}(\mathcal{C}/\mathcal{I})$.
\end{enumerate}
\end{lemma}

In the following proposition we give a characterization of the modules in $\mathbb{P}_{k}$ that will be used in the rest of the paper.

\begin{proposition}\label{Pkcarect}
Let $\mathcal{C}$ be a dualizing $R$-variety, $P=\mathrm{Hom}_{\mathcal{C}}(C,-)\in \mathrm{mod}(\mathcal{C})$ and $\mathcal{I}=\mathrm{Tr}_{P}\mathcal{C}$. For $1\leq k\leq \infty$ and $X\in \mathrm{mod}(\mathcal{C})$, the following are equivalent.
\begin{enumerate}
\item [(a)] $X\in \mathbb{P}_{k}$.

\item [(b)] $\mathrm{Ext}^{i}_{\mathrm{mod}(\mathcal{C})}(X,(\pi_{1})_{\ast}(Y))=0$ for all $Y\in\mathrm{mod}(\mathcal{C}/\mathcal{I})$ and $i=0,\dots, k$.

\item [(c)] $\mathrm{Ext}^{i}_{\mathrm{mod}(\mathcal{C})}(X,(\pi_{1})_{\ast}(J))=0$ for all $J\in\mathrm{mod}(\mathcal{C}/\mathcal{I})$ injective and $i=0,\dots, k$.
\end{enumerate}
\end{proposition}
\begin{proof}
$(a)\Rightarrow (b)$.  When applying $\mathrm{Hom}(-,(\pi_{1})_{\ast}Y)$ to a projective resolution of $X$, by \ref{percero}, we have a complex with the first $k$ terms zero.\\
$(b)\Rightarrow (a)$. We will proceed by induction on $k$. Suppose that $k=1$. That is, suppose that $\mathrm{Ext}^{i}_{\mathrm{mod}(\mathcal{C})}(X,(\pi_{1})_{\ast}(Y))=0$ for all $Y\in\mathrm{mod}(\mathcal{C}/\mathcal{I})$ and $i=0,1$. By  \ref{percero}, we have that the projective cover $P_{0}(X)$ of $X$ belongs to $\mathrm{add}(P)$. 
Now, let us consider the exact sequence $\xymatrix{0\ar[r] & L_{0}\ar[r] & P_{0}(X)\ar[r] & X\ar[r] & 0.}$ Applying $ \mathrm{Hom}_{\mathrm{mod}(\mathcal{C})}(-,(\pi_{1})_{\ast}(Y))$ to the last sequence we have the exact sequence

\vspace{0.3cm}

\begin{tikzpicture}[baseline= (a).base]

\node[scale=.85] (a) at (0,0){
\begin{tikzcd}
\mathrm{Hom}_{\mathrm{mod}(\mathcal{C})}(P_{0}(X),(\pi_{1})_{\ast}(Y)) \arrow[r] 
&\mathrm{Hom}_{\mathrm{mod}(\mathcal{C})}(L_{0},(\pi_{1})_{\ast}(Y)) \arrow[r]
&  \mathrm{Ext}_{\mathrm{mod}(\mathcal{C})}^{1}(X,(\pi_{1})_{\ast}(Y)).  
\end{tikzcd}
};
\end{tikzpicture}
\vspace{0.3cm}

Since $P_{0}(X)\in \mathrm{add}(P)$, we have that $ \mathrm{Hom}_{\mathrm{mod}(\mathcal{C})}(P_{0}(X),(\pi_{1})_{\ast}(Y))=0$ (see \ref{percero}) and,  by hypothesis, we have that  $\mathrm{Ex}_{\mathrm{mod}(\mathcal{C})}^{1}(X,(\pi_{1})_{\ast}(Y))=0$. Thus, we conclude that $\mathrm{Hom}_{\mathrm{mod}(\mathcal{C})}(L_{0},(\pi_{1})_{\ast}(Y)=0$ for all $Y\in \mathrm{mod}(\mathcal{C}/\mathcal{I})$. Then, in the same way as we did for $X$, we get that $P_{0}(L)\in \mathrm{add}(P)$. Hence, we have an exact sequence $P_{0}(L)\rightarrow  P_{0}(X)\rightarrow  X\rightarrow  0,$ with
$P_{0}(L),P_{0}(X)\in \mathrm{add}(P)$, proving that $X\in \mathbb{P}_{1}$.\\
Suppose that is true for $k-1$. Let  $X\in \mathrm{mod}(\mathcal{C})$ such that $\mathrm{Ext}^{i}_{\mathrm{mod}(\mathcal{C})}(X,(\pi_{1})_{\ast}(Y))=0$ for all $Y\in\mathrm{mod}(\mathcal{C}/\mathcal{I})$ and $i=0,\dots, k$. In particular, $\mathrm{Ext}^{i}_{\mathrm{mod}(\mathcal{C})}(X,(\pi_{1})_{\ast}(Y))=0$ for all $Y\in\mathrm{mod}(\mathcal{C}/\mathcal{I})$ and $i=0,\dots, k-1$. Then, by induction, there exists a resolution $\xymatrix{\cdots\ar[r] & P_{k-1}\ar[r]^{d_{k-1}} &\cdots\ar[r] & P_{1}\ar[r]^{d_{1}} & P_{0}\ar[r]^{d_{0}} & X\ar[r] & 0}$
with $P_{i}\in \mathrm{add}(P)$ for all $i=0,\dots, k-1$.
Consider the exact sequence
$$\xymatrix{0\ar[r] & L_{k-1}=\mathrm{Ker}(d_{k-1})\ar[r] & P_{k-1}\ar[r] & \mathrm{Ker}(d_{k-2})=L_{k-2}\ar[r] & 0}$$
Applying $ \mathrm{Hom}_{\mathrm{mod}(\mathcal{C})}(-,(\pi_{1})_{\ast}(Y))$ we have the exact sequence

\vspace{0.3cm}
\begin{tikzpicture}[baseline= (a).base]
\node[scale=.8] (a) at (0,0){
\begin{tikzcd}
\mathrm{Hom}_{\mathrm{mod}(\mathcal{C})}(P_{k-1},(\pi_{1})_{\ast}(Y))\arrow[r] 
&\mathrm{Hom}_{\mathrm{mod}(\mathcal{C})}(L_{k-1},(\pi_{1})_{\ast}(Y))\arrow[r]
&  \mathrm{Ext}_{\mathrm{mod}(\mathcal{C})}^{1}(L_{k-2},(\pi_{1})_{\ast}(Y))
\end{tikzcd}
};
\end{tikzpicture}
\vspace{0.3cm}

By shifting lemma, we get $ \mathrm{Ext}_{\mathrm{mod}(\mathcal{C})}^{1}(L_{k-2},(\pi_{1})_{\ast}(Y))\simeq  \mathrm{Ext}_{\mathrm{mod}(\mathcal{C})}^{k}(X,(\pi_{1})_{\ast}(Y))$ $=0$. Since $P_{k-1}\in \mathrm{add}(P)$, we conclude that  $\mathrm{Hom}_{\mathrm{mod}(\mathcal{C})}(P_{k-1},(\pi_{1})_{\ast}(Y))=0$ (see \ref{percero}); and hence $\mathrm{Hom}_{\mathrm{mod}(\mathcal{C})}(L_{k-1},(\pi_{1})_{\ast}(Y))=0$ for all $Y\in \mathrm{mod}(\mathcal{C}/\mathcal{I})$. In the same way as we did for the case $k=1$, we get that $P_{0}(L_{k-1})\in \mathrm{add}(P)$. Then we can construct  $\cdots\rightarrow P_{k}=P_{0}(L_{k-1})\rightarrow  P_{k-1}\rightarrow \cdots\rightarrow P_{1}\rightarrow  P_{0}\rightarrow  X\rightarrow  0,$ which is exact
with $P_{i}\in \mathrm{add}(P)$ for all $i=0,\dots, k$,  proving that $X\in \mathbb{P}_{k}$.\\
$(b)\Rightarrow (c)$. Trivial.\\
$(c)\Rightarrow (b)$. Follows by induction on $i$.
\end{proof}

\begin{proposition}\label{otracaidemo}
Let $\mathcal{C}$ be a dualizing $R$-variety, $P=\mathrm{Hom}_{\mathcal{C}}(C,-)\in \mathrm{mod}(\mathcal{C})$ and $\mathcal{I}=\mathrm{Tr}_{P}\mathcal{C}$. Let  $1\leq k\leq \infty$, then $\mathcal{I}$ is $(k+1)$-idempotent  if and only if $\mathcal{I}(C',-)\in \mathbb{P}_{k}$ for all $C'\in \mathcal{C}$.
\end{proposition}
\begin{proof}
First, we note that by \ref{trasatproA}, we have that $\mathcal{I}$ satisfies property $A$.\\
Let $\pi:\mathcal{C}\longrightarrow \mathcal{C}/\mathcal{I}$ be canonical functor.
By \ref{almostrecolle}, we have that $\pi_{\ast}(\mathrm{Hom}_{\mathcal{C}/\mathcal{I}}(C',-))=\frac{\mathrm{Hom}_{\mathcal{C}}(C',-)}{\mathcal{I}(C',-)}\in \mathrm{mod}(\mathcal{C})$.
Consider the following exact sequence in $\mathrm{Mod}(\mathcal{C})$
$$\xymatrix{0\ar[r] & \mathcal{I}(C',-)\ar[r] & \mathrm{Hom}_{\mathcal{C}}(C',-)\ar[r] & \frac{\mathrm{Hom}_{\mathcal{C}}(C',-)}{\mathcal{I}(C',-)}\ar[r] & 0.}$$
Since $\mathcal{C}$ is a dualizing $R$-variety, we have that $\mathrm{mod}(\mathcal{C})$ is abelian subcategory of $\mathrm{Mod}(\mathcal{C})$ and thus $\mathcal{I}(C',-)\in \mathrm{mod}(\mathcal{C})$ (see \cite[Theorem 2.4]{AusVarietyI}).\\
Since $\mathcal{I}=\mathrm{Tr}_{P}\mathcal{C}$, there exists an epimorphism $\gamma:P^{n}\longrightarrow \mathcal{I}(C',-)$ for some $n\in \mathbb{N}$, this is because $\mathcal{I}(C',-)$ is finitely generated (see \cite[2.1(b)]{AusM1}). Let $Y\in \mathrm{mod}(\mathcal{C}/\mathcal{I})$. If there exists a non zero morphism $\alpha:\mathcal{I}(C',-)\longrightarrow \pi_{\ast}(Y)$ we have that $\alpha \gamma \neq 0$, which contradicts \ref{percero}. Therefore, we have that $\mathrm{Hom}_{\mathrm{mod}(\mathcal{C})}\Big(\mathcal{I}(C',-),\pi_{\ast}(Y)\Big)=0$. On the other hand, applying $\mathrm{Hom}_{\mathrm{mod}(\mathcal{C})}(-,\pi_{\ast}(Y))$ to the last exact sequence we get
an isomorphism for $i\geq 1$
$$(\ast):\mathrm{Ext}^{i}_{\mathrm{mod}(\mathcal{C})}\Big(\mathcal{I}(C',-),\pi_{\ast}(Y)\Big)\longrightarrow \mathrm{Ext}^{i+1}_{\mathrm{mod}(\mathcal{C})}\Big(\frac{\mathrm{Hom}_{\mathcal{C}}(C',-)}{\mathcal{I}(C',-)},\pi_{\ast}(Y)\Big).$$
We know that $\mathcal{I}$ is $(k+1)$-idempotent if and only if 
$\mathbb{EXT}^{i}_{\mathcal{C}}\Big(\mathcal{C}/\mathcal{I},\pi_{\ast}(Y)\Big)=0$ for $1\leq i\leq k+1$ and for all $Y\in \mathrm{mod}(\mathcal{C}/\mathcal{I})$ (see \ref{caractidemfin}). Now, the result follows from \ref{caractidemfin}, \ref{Pkcarect}.

\end{proof}

\begin{corollary}\label{caractstronidemp}
Let $\mathcal{C}$ be a dualizing $R$-variety, $P=\mathrm{Hom}_{\mathcal{C}}(C,-)\in \mathrm{mod}(\mathcal{C})$ and $\mathcal{I}=\mathrm{Tr}_{P}\mathcal{C}$.
Then $\mathcal{I}$ is strongly idempotent if and only if $\mathcal{I}(C',-)\in \mathbb{P}_{\infty}$ $\forall C'\in \mathcal{C}$.
\end{corollary}

Let $\mathcal{C}$ be a dualizing $R$-variety. Given $M\in \mathrm{mod}(\mathcal{C})$, we recall that $\mathrm{rad}(M)$ denotes the $\textbf{radical}$ of $M$, that is, $\mathrm{rad}(M)$ is the intersection of the maximal submodules of $M$.

\begin{definition}\label{defiIk} 
Let $\mathcal{C}$ be a dualizing $R$-variety,  $P\in \mathrm{mod}(\mathcal{C})$  a projective module and $J:=I_{0}(\frac{P}{\mathrm{rad}(P)})\in \mathrm{mod}(\mathcal{C})$ the injective envelope of $\frac{P}{\mathrm{rad}(P)}$. For each $0\leq k\leq \infty$ we define $\mathbb{I}_{k}$ to be the full subcategory of $\mathrm{mod}(\mathcal{C})$ consisting of the $\mathcal{C}$-modules $Y$ having an injective coresolution $\xymatrix{0\ar[r] & Y\ar[r] &  J_{0}\ar[r] & J_{1}\ar[r] &\cdots }$ with $J_{i}\in \mathrm{add}(J)$ for $0\leq i\leq k$.
\end{definition}

Let $\mathcal{C}$ be a dualizing $R$-variety. Since the endomorphism ring of each object in $\mathcal{C}$ is an artin algebra, it follows that $\mathcal{C}$ is a 
Krull-Schmidt category \cite[p.337]{AusVarietyI}. By \ref{projfingen}, we conclude that $P\in \mathrm{proj}(\mathcal{C})$ is indecomposable if and only if $P\simeq \mathrm{Hom}_{\mathcal{C}}(C,-)$ where $C\in \mathcal{C}$ is indecomposable (see also  \cite[Lemma 2.2 (b)]{MOSS}).

\begin{lemma}\label{HOmInje}
Let $\mathcal{C}$ be a dualizing $R$-variety, $P=\mathrm{Hom}_{\mathcal{C}}(C,-)\in \mathrm{mod}(\mathcal{C})$ an indecomposable projective module, $J:=I_{0}(\frac{P}{\mathrm{rad}(P)})\in \mathrm{mod}(\mathcal{C})$ and $\mathcal{I}=\mathrm{Tr}_{P}\mathcal{C}$.
Consider the projection $\pi_{1}:\mathcal{C}\longrightarrow \mathcal{C}/\mathcal{I}$. Then 
$\mathrm{Hom}_{\mathrm{mod}(\mathcal{C})}((\pi_{1})_{\ast}(X),J)=0$ for all $X\in \mathrm{mod}(\mathcal{C}/\mathcal{I})$. 
\end{lemma}
\begin{proof}
It can be seen that $\mathrm{Tr}_{\frac{\mathcal{C}}{\mathcal{I}}}(J)=0$, and hence  $\overline{\mathrm{Tr}}_{\frac{\mathcal{C}}{\mathcal{I}}}(J)=\Omega(\mathrm{Tr}_{\frac{\mathcal{C}}{\mathcal{I}}}(J))=0$. Now, by adjunction (see \ref{firstadjoin}),
$\mathrm{Hom}_{\mathrm{mod}(\mathcal{C})}((\pi_{1})_{\ast}(X),J)\simeq \mathrm{Hom}_{\mathrm{mod}(\mathcal{C}/\mathcal{I})}(X,\overline{\mathrm{Tr}}_{\frac{\mathcal{C}}{\mathcal{I}}}(J))=\mathrm{Hom}_{\mathrm{mod}(\mathcal{C}/\mathcal{I})}(X,0)=0.$
\end{proof}

\begin{proposition}\label{desenvoline}
Let $\mathcal{C}$ be a dualizing $R$-variety, $P=\mathrm{Hom}_{\mathcal{C}}(C,-)\in \mathrm{mod}(\mathcal{C})$ an indecomposable projective module, $J:=I_{0}(\frac{P}{\mathrm{rad}(P)})\in \mathrm{mod}(\mathcal{C})$ and $\mathcal{I}=\mathrm{Tr}_{P}\mathcal{C}$.
Consider the projection $\pi_{1}:\mathcal{C}\longrightarrow \mathcal{C}/\mathcal{I}$. Then $J\simeq \mathbb{D}_{\mathcal{C}}^{-1}(\mathrm{Hom}_{\mathcal{C}}(-,C)).$
\end{proposition}
\begin{proof}
Since $J$ is the injective envelope of the simple $\frac{P}{\mathrm{rad}(P)}$, we have that  $J$ is indecomposable. Then we have that $\mathbb{D}_{\mathcal{C}}(J)\in \mathrm{mod}(\mathcal{C}^{op})$ is an indecomposable projective  $\mathcal{C}$-module.  By \ref{HOmInje} and  \ref{Pkcarect}, we get that there exists an exact sequence $\xymatrix{Q_{1}\ar[r] & Q_{0}\ar[r] & \mathbb{D}_{\mathcal{C}}(J)\ar[r] & 0}$
with $Q_{i}\in \mathrm{add}(P^{\ast})=\mathrm{add}(\mathrm{Hom}_{\mathcal{C}}(-,C))$. Since $\mathbb{D}_{\mathcal{C}}(J)$, is projective we conclude that $\mathbb{D}_{\mathcal{C}}(J)$ is a direct summand of $Q_{0}$ and therefore $\mathbb{D}_{\mathcal{C}}(J)\in \mathrm{add}(\mathrm{Hom}_{\mathcal{C}}(-,C))$. Since $\mathrm{Hom}_{\mathcal{C}}(-,C)$ is indecomposable and $\mathrm{mod}(\mathcal{C})$ is Krull-Schmidt (because $\mathrm{mod}(\mathcal{C})$ is a dualizing $R$-variety), we conclude that $\mathbb{D}_{\mathcal{C}}(J)\simeq \mathrm{Hom}_{\mathcal{C}}(-,C)$.
\end{proof}

The following shows that the last proposition holds for every finitely generated projective $\mathcal{C}$-module.
\begin{proposition}\label{desenvolinegen}
Let $\mathcal{C}$ be a dualizing $R$-variety, $P=\mathrm{Hom}_{\mathcal{C}}(C,-)\in \mathrm{mod}(\mathcal{C})$ a projective module, $J:=I_{0}(\frac{P}{\mathrm{rad}(P)})\in \mathrm{mod}(\mathcal{C})$ the injective envelope of $\frac{P}{\mathrm{rad}(P)}$ and $\mathcal{I}=\mathrm{Tr}_{P}\mathcal{C}$. Then $J\simeq \mathbb{D}_{\mathcal{C}}^{-1}(\mathrm{Hom}_{\mathcal{C}}(-,C)).$
\end{proposition}

Similarly, the result given in \ref{HOmInje} holds for every finitely generated projective $\mathcal{C}$-module. That is, we have the following result.

\begin{Remark}\label{HOmInjegen}
Let $\mathcal{C}$ be a dualizing $R$-variety, $P=\mathrm{Hom}_{\mathcal{C}}(C,-)\in \mathrm{mod}(\mathcal{C})$ a projective module, $J:=I_{0}(\frac{P}{\mathrm{rad}(P)})\in \mathrm{mod}(\mathcal{C})$ and $\mathcal{I}=\mathrm{Tr}_{P}\mathcal{C}$.
Consider the projection $\pi_{1}:\mathcal{C}\longrightarrow \mathcal{C}/\mathcal{I}$. Then 
$\mathrm{Hom}_{\mathrm{mod}(\mathcal{C})}((\pi_{1})_{\ast}(X),J)=0$ for all $X\in \mathrm{mod}(\mathcal{C}/\mathcal{I})$. 
\end{Remark}

\begin{corollary}\label{otherpercer}
Let $\mathcal{C}$ be a dualizing $R$-variety, $P=\mathrm{Hom}_{\mathcal{C}}(C,-)\in \mathrm{mod}(\mathcal{C})$ a projective module, $J:=I_{0}(\frac{P}{\mathrm{rad}(P)})\in \mathrm{mod}(\mathcal{C})$ and $\mathcal{I}=\mathrm{Tr}_{P}\mathcal{C}$.
Consider the projection $\pi_{1}:\mathcal{C}\longrightarrow \mathcal{C}/\mathcal{I}$. Then $Y\in \mathbb{I}_{0}$ if and only if $\mathrm{Hom}_{\mathrm{mod}(\mathcal{C})}((\pi_{1})_{\ast}(X),Y)=0$ for all $X\in \mathrm{mod}(\mathcal{C}/\mathcal{I})$.
\end{corollary}
\begin{proof}
If follows by duality using \ref{PPast},  \ref{percero} and  \ref{desenvolinegen}.
\end{proof}

\begin{proposition}\label{muchaseuiii}
Let $P=\mathrm{Hom}_{\mathcal{C}}(C,-)\in \mathrm{mod}(\mathcal{C})$ be a projective module, $\mathcal{I}=\mathrm{Tr}_{P}\mathcal{C}$ and $1\leq k\leq \infty$. Let $\pi_{1}:\mathcal{C}\longrightarrow \mathcal{C}/\mathcal{I}$ the canonical projection. The following conditions are equivalent for $Y\in \mathrm{mod}(\mathcal{C})$.
\begin{enumerate}
\item [(a)] $Y\in \mathbb{I}_{k}$.

\item [(b)] $\mathrm{Ext}^{i}_{\mathrm{mod}(\mathcal{C})}((\pi_{1})_{\ast}(X),Y)=0$ for all $X\in\mathrm{mod}(\mathcal{C}/\mathcal{I})$ and $i=0,\dots, k$.

\item [(c)] $\mathrm{Ext}^{i}_{\mathrm{mod}(\mathcal{C})}((\pi_{1})_{\ast}(Q),Y)=0$ for all $Q\in\mathrm{mod}(\mathcal{C}/\mathcal{I})$ projective and $i=0,\dots, k$.
\end{enumerate}
\end{proposition}
\begin{proof}
It follows from \ref{Pkcarect} using the duality.
\end{proof}

\section{A recollement}\label{sec:7}

Let $\mathcal{C}$ be a preadditive category. Throught this section $P$ will be a finitely generated projective module in $\mathrm{Mod}(\mathcal{C})$ and $R_{P}:=\mathrm{End}_{\mathrm{Mod}(\mathcal{C})}(P)^{op}$. In this section we will study the functor $\mathrm{Hom}_{\mathrm{Mod}(\mathcal{C})}(P,-):\mathrm{Mod}(\mathcal{C})\longrightarrow \mathrm{Mod}(R_{P})$. The following remark is straightforward and we left the details to the reader.
 
\begin{Remark}\label{equicovacon}
Let $\mathcal{C}$ be a preadditive category and $P\in \mathrm{proj}(\mathcal{C})$.
\begin{enumerate}
\item [(a)] Consider the Yoneda embedding $\mathbb{Y}:\mathcal{C}\longrightarrow \mathrm{proj}(\mathcal{C})$ defined as $\mathbb{Y}(C):=\mathrm{Hom}_{\mathcal{C}}(C,-)$. Then $\mathbb{Y}_{\ast}:\mathrm{Mod}((\mathrm{proj}(\mathcal{C}))^{op})\longrightarrow \mathrm{Mod}(\mathcal{C})$
is an equivalence.
\item [(b)] Consider the inclusion $\{P\}^{op}\subseteq \mathrm{proj}(\mathcal{C})^{op}$ and the restriction functor $\mathrm{res}:\mathrm{Mod}((\mathrm{proj}(\mathcal{C}))^{op})\longrightarrow \mathrm{Mod}(\{P\}^{op})$ given as: $M\mapsto M|_{\{P^{op}\}}$. Then the following diagram is commutative
$$(\ast):\quad \xymatrix{\mathrm{Mod}((\mathrm{proj}(\mathcal{C}))^{op})\ar[rr]^{\mathrm{res}}\ar[d]^{\mathbb{Y}_{\ast}} & &  \mathrm{Mod}(\{P\}^{op})\ar[d]^{e_{P}}\\
\mathrm{Mod}(\mathcal{C})\ar[rr]_{\mathrm{Hom}_{\mathrm{Mod}(\mathcal{C})}(P,-)} & & \mathrm{Mod}(R_{P})}$$
where $R_{P}:=\mathrm{End}_{\mathrm{Mod}(\mathcal{C})}(P)^{op}$ and $e_{P}$ is the evaluation functor defined as follows: $e_{P}(M)=M(P)$ for $M\in  \mathrm{Mod}(\{P\}^{op})$.
\item [(c)] We define a functor $P\otimes_{R_{P}}- :\mathrm{Mod}(R_{P})\longrightarrow  \mathrm{Mod}(\mathcal{C})$ as follows: $(P\otimes_{R_{P}}M)(C)=P(C)\otimes_{R_{P}}M$ for all $M\in \mathrm{Mod}(R_{P})$ and $C\in \mathcal{C}$. Then the following diagram commutes
$$\xymatrix{\mathrm{Mod}(\{P\}^{op})\ar[rrr]^{(\mathrm{proj}(\mathcal{C})^{op})\otimes_{\{P\}^{op}}-}\ar[d]_{e_{P}'} & & & \mathrm{Mod}(\mathrm{proj}(\mathcal{C}))^{op})\ar[d]^{\mathbb{Y}_{\ast}}\\
\mathrm{Mod}(R_{P})\ar[rrr]_{P\otimes_{R_{P}}-} & & & \mathrm{Mod}(\mathcal{C})}$$
where $e_{P}'$ is the evaluation and $(\mathrm{proj}(\mathcal{C})^{op})\otimes_{\{P\}^{op}}-$  is the functor defined in \ref{Rproposition1}.
\item [(d)] Recall that  $P^{\ast}\in \mathrm{Mod}(\mathcal{C}^{op})$ is given by  $P^{\ast}(C):=\mathrm{Hom}_{\mathrm{Mod}(\mathcal{C})}(P,\mathrm{Hom}_{\mathcal{C}}(C,-))$. Now we can construct a functor $\mathrm{Hom}_{R_{P}}(P^{\ast},-):\mathrm{Mod}(R_{P})\longrightarrow \mathrm{Mod}(\mathcal{C})$ where for $M\in \mathrm{Mod}(R_{P})$ we define $\mathrm{Hom}_{R_{P}}(P^{\ast},M):\mathcal{C}\longrightarrow \mathbf{Ab}$ as follows: 
$\Big(\mathrm{Hom}_{R_{P}}(P^{\ast},M)\Big)(C):=\mathrm{Hom}_{R_{P}}(P^{\ast}(C),M)$. Then the following diagram commutes 
$$\xymatrix{\mathrm{Mod}(\{P\}^{op})\ar[rrr]^{\{P\}^{op}(\mathrm{proj}(\mathcal{C})^{op},-)}\ar[d]_{e_{P}'} & & & \mathrm{Mod}(\mathrm{proj}(\mathcal{C}))^{op})\ar[d]^{\mathbb{Y}_{\ast}}\\
\mathrm{Mod}(R_{P})\ar[rrr]_{\mathrm{Hom}_{R_{P}}(P^{\ast},-)} & & & \mathrm{Mod}(\mathcal{C})}$$
where $e_{P}'$ is the evaluation and $\{P\}^{op}(\mathrm{proj}(\mathcal{C})^{op},-)$ is the functor defined in \ref{Rproposition2}.
\end{enumerate}
\end{Remark}

Now, we give the following definition which encodes the information of several adjunctions.

\begin{definition}\label{recolldef}
Let $\mathcal{A}$, $\mathcal{B}$ and $\mathcal{C}$ be abelian categories. Then the diagram

$$\xymatrix{\mathcal{B}\ar[rr]|{i_{\ast}=i_{!}}  &  &\mathcal{A}\ar[rr]|{j^{!}=j^{\ast}}\ar@<-2ex>[ll]_{i^{\ast}}\ar@<2ex>[ll]^{i^{!}}  &   &\mathcal{C}\ar@<-2ex>[ll]_{j_{!}}\ar@<2ex>[ll]^{j_{\ast}}}$$
is called a $\textbf{recollement}$, if the additive functors $i^{\ast},i_{\ast}=i_{!}, i^{!},j_{!},j^{!}=j^{*}$ and $j_{*}$ satisfy the following conditions:
\begin{itemize}
\item[(R1)] $(i^{*},i_{*}=i_{!},i^{!})$ and $(j_{!},j^{!}=j^{*},j_{*})$ are adjoint  triples, i.e. $(i^{*},i_{*})$, $(i_{!},i^{!})$  $(j_{!},j^{!})$ and $(j^{*},j_{*})$ are adjoint pairs;
\item[(R2)] $j^{*}i_{*}=0$;
\item[(R3)] $i_{*}, j_{!},j_{*}$ are full embedding functors.
\end{itemize}
\end{definition}

Next, we will see that we can construct a recollement.

\begin{proposition}\label{recolleproye}
Let $\mathcal{C}$ be an $R$-category, $P=\mathrm{Hom}_{\mathcal{C}}(C,-)\in \mathrm{Mod}(\mathcal{C})$ a finitely generated projective module, $\mathcal{B}=\mathrm{add}(C)\subseteq \mathcal{C}$ and $R_{P}=\mathrm{End}_{\mathrm{Mod}(\mathcal{C})}(P)^{op}$.  Then, there exists a recollement of the form
$$\xymatrix{\mathrm{Mod}(\mathcal{C}/\mathcal{I}_{\mathcal{B}})\ar[rrr]|{\pi_{\ast}}  &  &&\mathrm{Mod}(\mathcal{C})\ar[rrr]|{\mathrm{Hom}_{\mathrm{Mod}(\mathcal{C})}(P,-)}
\ar@<-2ex>[lll]_{\mathcal{C}/\mathcal{I}_{\mathcal{B}}\otimes_{\mathcal{C}}-}\ar@<2ex>[lll]^{\mathcal{C}(\mathcal{C}/\mathcal{I}_{\mathcal{B}},-)}  & &   &\mathrm{Mod}(R_{P})\ar@<-2ex>[lll]_{P\otimes_{R_{P}}-}\ar@<2ex>[lll]^{\mathrm{Hom}_{R_{P}}(P^{\ast},-)}}$$
where $\mathcal{I}_{\mathcal{B}}$ is the ideal of morphisms in $\mathcal{C}$ which factor through objects in $\mathcal{B}$.
\end{proposition}
\begin{proof}
This follows by \cite[Theorem 3.10]{LeOS2}, using the identifications \ref{equicovacon}.
\end{proof}

We can restrict the last recollement to the finitely presented modules. So, we have the following result that is an analogous to the one given in artin algebras.

\begin{proposition}\label{recolleproyefin}
Let $\mathcal{C}$ be a dualizing $R$-variety, $P=\mathrm{Hom}_{\mathcal{C}}(C,-)\in \mathrm{mod}(\mathcal{C})$, $\mathcal{B}=\mathrm{add}(C)$ and $R_{P}=\mathrm{End}_{\mathrm{Mod}(\mathcal{C})}(P)^{op}$. Then,  there exist a recollement
$$\xymatrix{\mathrm{mod}(\mathcal{C}/\mathcal{I}_{\mathcal{B}})\ar[rrr]|{\pi_{\ast}}  &  &&\mathrm{mod}(\mathcal{C})\ar[rrr]|{\mathrm{Hom}_{\mathrm{mod}(\mathcal{C})}(P,-)}
\ar@<-2ex>[lll]_{\mathcal{C}/\mathcal{I}_{\mathcal{B}}\otimes_{\mathcal{C}}-}\ar@<2ex>[lll]^{\mathcal{C}(\mathcal{C}/\mathcal{I}_{\mathcal{B}},-)}  & &   &\mathrm{mod}(R_{P})\ar@<-2ex>[lll]_{P\otimes_{R_{P}}-}\ar@<2ex>[lll]^{\mathrm{Hom}_{R_{P}}(P^{\ast},-)}}$$
where $\mathcal{I}_{\mathcal{B}}$ is the ideal of morphisms in $\mathcal{C}$ which factor through objects in $\mathcal{B}$. Moreover, we have that $\mathcal{I}_{\mathcal{B}}=\mathrm{Tr}_{P}\mathcal{C}$.
\end{proposition}
\begin{proof}
Let us take $\mathcal{C}'=\mathrm{proj}(\mathcal{C})^{op}$ and $\mathcal{B}'=\mathrm{add}(P)^{op}$. In the proof of \ref{recolleproye} we have a recollement where $\mathcal{I}_{\mathcal{B}'}=\mathcal{I}_{\mathrm{add}(P)}$ is the ideal of morphisms in $\mathcal{C}'$ which factor through objects in $\mathrm{add}(P)$. The Yoneda's embedding gives us an equivalence $\mathcal{C}\simeq \mathrm{proj}(\mathcal{C})^{op}$ (see \ref{projfingen}), and hence $\mathcal{C}'=\mathrm{proj}(\mathcal{C})^{op}$ is a dualizing $R$-variety. Now, $\mathcal{B'}=\mathrm{add}(P)^{op}$ is a functorially finite subcategory of $\mathcal{C}'$ (see \ref{addfun}). Therefore, by \cite[Theorem 2.5]{Yasuaki} and \ref{recolleproye}, we can restrict the last recollement to the finitely presented modules where $\mathcal{I}_{\mathcal{B}}$ is the ideal of morphisms in $\mathcal{C}$ which factor through objects in $\mathcal{B}$, since $\mathrm{mod}(R_{P})$ coincides with the finitely presented $R$-modules (because $R_{P}$ is an artin $R$-algebra). Finally, by \ref{Traza=Iadd}, we have that  $\mathcal{I}_{\mathcal{B}}=\mathrm{Tr}_{P}\mathcal{C}$.
\end{proof}

We have the following definition due to Auslander \cite{AusM1}.

\begin{definition}\label{defiprojinjpres}
Let $\mathcal{C}$ a dualizing $R$-category and  $P=\mathrm{Hom}_{\mathcal{C}}(C,-)\in \mathrm{mod}(\mathcal{C})$.
Let $M\in \mathrm{mod}(\mathcal{C})$.
\begin{enumerate}
\item [(a)] It is said that $M$ is $\textbf{proyectively presented}$ over $P$ if $\epsilon_{M}'$ is an isomorphism.  Let us denote by $\mathbb{F.P.P}(P)$ the full subcategory of $\mathrm{mod}(\mathcal{C})$ consisting of the projectively presented modules

\item [(b)] It is said that $M$ is $\textbf{inyectively copresented}$ over $P$ if $\eta_{M}$ is an isomorphism. Let us denote by $\mathbb{F.I.C}(P)$ the full subcategory of $\mathrm{mod}(\mathcal{C})$ consisting of the injectively copresented modules. 
\end{enumerate}
\end{definition}

We recall that in the case of a dualizing $R$-variety every finitely generated projective $\mathcal{C}$-module is of the form $\mathrm{Hom}_{\mathcal{C}}(C,-)$ (see \ref{projfingen}).

\begin{proposition}\label{euivalenciaproj1}
Let $\mathcal{C}$ be a dualizing $R$-variety and  $P=\mathrm{Hom}_{\mathcal{C}}(C,-)\in \mathrm{mod}(\mathcal{C})$.  For  $M\in \mathrm{mod}(\mathcal{C})$, the following are equivalent.
\begin{enumerate}
\item [(a)] $M\in \mathbb{F.P.P}(P)$.

\item [(b)] There exists a module $X\in \mathrm{mod}(R_{P})$ such that
$M\simeq P\otimes_{R_{P}}X$.

\item [(c)]  There exists an exact sequence $P_{1}\rightarrow P_{0}\rightarrow M\rightarrow 0$ with $P_{1},P_{0}\in \mathrm{add}(P)$.

\item [(d)] $\mathrm{Hom}_{\mathrm{Mod}(\mathcal{C})}(M,N)\rightarrow \mathrm{Hom}_{R_{P}}\Big(\mathrm{Hom}_{\mathrm{Mod}(\mathcal{C})}(P,M), \mathrm{Hom}_{\mathrm{Mod}(\mathcal{C})}(P,N)\Big)$
is an isomorphism for each module $N\in \mathrm{mod}(\mathcal{C})$.
\end{enumerate}
\end{proposition}
\begin{proof}
See \cite[Proposition 3.2]{AusM1}.
\end{proof}
The next result give us a characterization of the the categories $\mathbb{F.P.P}(P)$ and $\mathbb{F.I.C}(P)$ which will help us in the forthcoming section.

\begin{proposition}\label{euiijedinifin}
Let $\mathcal{C}$ be a dualizing $R$-variety and $P=\mathrm{Hom}_{\mathcal{C}}(C,-) \in \mathrm{mod}(\mathcal{C})$. Then the following conditions hold.
\begin{enumerate}
\item [(a)] $M\in \mathbb{F.P.P}(P)$ if and only if $\mathrm{Hom}_{\mathrm{mod}(\mathcal{C})}(M,N)=0$ and  $\mathrm{Ext}_{\mathrm{mod}(\mathcal{C})}^{1}(M,N)=0$ for all $N\in \mathrm{mod}(\mathcal{C})$ with  $N\in \mathrm{Ker}(\mathrm{Hom}_{\mathrm{mod}(\mathcal{C})}(P,-))$.

\item [(b)] $N\in \mathbb{F.I.C}(P)$ if and only if $\mathrm{Hom}_{\mathrm{Mod}(\mathcal{C})}(M,N)=0$ and  $\mathrm{Ext}_{\mathrm{mod}(\mathcal{C})}^{1}(M,N)=0$ for all  $M\in \mathrm{mod}(\mathcal{C})$ with  $M\in \mathrm{Ker}(\mathrm{Hom}_{\mathrm{mod}(\mathcal{C})}(P,-))$.
\end{enumerate}
\end{proposition}
\begin{proof}
The proof given in  \cite[Proposition 3.7]{AusM1} works for the finitely presented modules, since $\mathrm{mod}(\mathcal{C})$ is an abelian subcategory of $\mathrm{Mod}(\mathcal{C})$ with enough injectives and projectives and we have the adjunctions in \ref{recolleproyefin}.
\end{proof}

\begin{proposition}\label{lemainteres}
Let $\mathcal{C}$ be a dualizing $R$-variety and $P=\mathrm{Hom}_{\mathcal{C}}(C,-)\in \mathrm{mod}(\mathcal{C})$. Consider the functor $\mathrm{Hom}_{\mathrm{mod}(\mathcal{C})}(P,-):\mathrm{mod}(\mathcal{C})\longrightarrow \mathrm{mod}(R_{P})$. The following hold.
\begin{enumerate}
\item [(a)] $\mathrm{Hom}_{\mathrm{mod}(\mathcal{C})}(P,-)|_{\mathbb{P}_{1}}:\mathbb{P}_{1}\longrightarrow \mathrm{mod}(R_{P})$ and 
$\mathrm{Hom}_{\mathrm{mod}(\mathcal{C})}(P,-)|_{\mathbb{I}_{1}}:\mathbb{I}_{1}\longrightarrow \mathrm{mod}(R_{P})$ are equivalences, where $\mathbb{P}_{1}$ and $\mathbb{I}_{1}$ are the categories defined in \ref{defiPk} and  \ref{defiIk}.

\item [(b)] Let  $\xymatrix{\mathrm{Hom}_{\mathrm{mod}(\mathcal{C})}(X,Y)\ar[r]^(.3){\rho_{X,Y}} &  \mathrm{Hom}_{R_{P}}\Big( \mathrm{Hom}_{\mathrm{mod}(\mathcal{C})}(P,X),\mathrm{Hom}_{\mathrm{mod}(\mathcal{C})}(P,Y)\Big)}.$ Then:
\item [(i)] $\rho_{X,Y}$ is a monomorphism if either $X\in \mathbb{P}_{0}$ or $Y\in \mathbb{I}_{0}$,

\item [(ii)] $\rho_{X,Y}$ is an isomorphism if  $X\in \mathbb{P}_{0}$ and $Y\in \mathbb{I}_{0}$,

\item [(iii)] $\rho_{X,Y}$ is an isomorphism if either $X\in \mathbb{P}_{1}$ or $Y\in \mathbb{I}_{1}$.

\item [(c)] The functor $\mathrm{Hom}_{\mathrm{mod}(\mathcal{C})}(P,-)$ induces an equivalence of categories between $\mathrm{add}(P)$ an the category of projective $R_{P}$-modules, and between $\mathrm{add}(J)$ and the category of injective $R_{P}$-modules.

\end{enumerate}
\end{proposition}

\begin{proof}
$(a)$. The fact that $\mathrm{Hom}_{\mathrm{mod}(\mathcal{C})}(P,-)|_{\mathbb{P}_{1}}:\mathbb{P}_{1}\longrightarrow \mathrm{mod}(R_{P})$ is an equivalence follows from  \cite[Proposition 3.3]{AusM1}, since $\mathbb{F.P.P}(P)=\mathbb{P}_{1}$.\\
By  \ref{euiijedinifin} and \ref{muchaseuiii}, we have that $\mathbb{F.I.C}(P)=\mathbb{I}_{1}$. Then, by  \cite[Proposition 3.6]{AusM1}, we have that $\mathrm{Hom}_{\mathrm{mod}(\mathcal{C})}(P,-)|_{\mathbb{I}_{1}}:\mathbb{I}_{1}\longrightarrow \mathrm{mod}(R_{P})$ is an equivalence.\\
$(bi)$. It follows from \ref{euivalenciaproj1}(d) and the proof is left to the reader.\\
$(bii)$.  Since $X\in \mathbb{P}_{0}$, we have by $(bi)$ that 
$$\rho_{X,Y}:\mathrm{Hom}_{\mathrm{mod}(\mathcal{C})}(X,Y)\longrightarrow \mathrm{Hom}_{R_{P}}\Big( \mathrm{Hom}_{\mathrm{mod}(\mathcal{C})}(P,X),\mathrm{Hom}_{\mathrm{mod}(\mathcal{C})}(P,Y)\Big)$$ is a monomorphism. Since $X\in \mathbb{P}_{0}$ and $Y\in \mathbb{I}_{0}$, there exists an epimorphism $\xymatrix{Q\ar[r]^{\pi} & X\ar[r] &  0}$ and a monomorphism
$\xymatrix{0\ar[r] &  Y\ar[r]^{\mu} &  I}$ with
$Q\in \mathrm{add}(P)$ and $I\in \mathrm{add}(J)$. By definition  of $\mathbb{P}_{1}$ and $\mathbb{I}_{1}$, we have that $Q\in \mathbb{P}_{1}$ and $I\in \mathbb{I}_{1}$.\\
Let us see that $\rho_{X,Y}$ is surjective.\\
 Let $\varphi:\mathrm{Hom}_{\mathrm{mod}(\mathcal{C})}(P,X)\longrightarrow  \mathrm{Hom}_{\mathrm{mod}(\mathcal{C})}(P,Y)$ be a morphism of $R_{P}$-modules. Consider the morphism $\mathrm{Hom}_{\mathrm{mod}(\mathcal{C})}(P,\mu):\mathrm{Hom}_{\mathrm{mod}(\mathcal{C})}(P,Y)\longrightarrow  \mathrm{Hom}_{\mathrm{mod}(\mathcal{C})}(P,I)$ and then we get
$\mathrm{Hom}_{\mathrm{mod}(\mathcal{C})}(P,\mu)\circ \varphi:\mathrm{Hom}_{\mathrm{mod}(\mathcal{C})}(P,X)\longrightarrow  \mathrm{Hom}_{\mathrm{mod}(\mathcal{C})}(P,I)$. Since $I\in \mathbb{I}_{1}=\mathbb{F.I.C}(P)$,  we have that
$$\rho_{X,I}:\mathrm{Hom}_{\mathrm{mod}(\mathcal{C})}(X,I)\longrightarrow \mathrm{Hom}_{R_{P}}\Big( \mathrm{Hom}_{\mathrm{mod}(\mathcal{C})}(P,X),\mathrm{Hom}_{\mathrm{mod}(\mathcal{C})}(P,I)\Big)$$
is an isomorphism (see dual of \ref{euivalenciaproj1}). Then there exists a morphism $\lambda:X\longrightarrow I$ such that $\mathrm{Hom}_{\mathrm{mod}(\mathcal{C})}(P,\lambda)=\mathrm{Hom}_{\mathrm{mod}(\mathcal{C})}(P,\mu)\circ \varphi$.  Similarly,  we have that 
$$\rho_{Q,Y}:\mathrm{Hom}_{\mathrm{mod}(\mathcal{C})}(Q,Y)\longrightarrow \mathrm{Hom}_{R_{P}}\Big( \mathrm{Hom}_{\mathrm{mod}(\mathcal{C})}(P,Q),\mathrm{Hom}_{\mathrm{mod}(\mathcal{C})}(P,Y)\Big)$$
is an isomorphism (see \ref{euivalenciaproj1}(d)). Then there exists a morphism $\beta:Q\longrightarrow Y$ such that $\varphi\circ \mathrm{Hom}_{\mathrm{mod}(\mathcal{C})}(P,\pi)=\mathrm{Hom}_{\mathrm{mod}(\mathcal{C})}(P,\beta)$.\\
Then we have two morphisms  $\lambda\pi,\mu\beta:Q\longrightarrow I$ and it is straigthforward to check that $\lambda\pi=\mu\beta$. Let  us consider the factorization $\xymatrix{X\ar[r]^{p} & K\ar[r]^{\delta} & I}$ of $\lambda$ through its image. Then we have that $\mu\beta=\lambda\pi=\delta p\pi=\delta(p\pi)$ with $p\pi$ an epimorphism and $\delta$ a monomorphism. Then we have that $\delta$ is the image of $\lambda\pi$. Since $\mu$ is a monomorphism, by the universal property of the image, 
there exists $\psi:K\longrightarrow Y$ such that $\delta=\mu\psi$. Now, it is easy to see that 
$\varphi=\mathrm{Hom}_{\mathrm{mod}(\mathcal{C})}(P,\psi\circ p)$.
Therefore, we conclude that $\rho_{X,Y}$ is surjective, and then an isomorphism.\\
$(biii)$ Follows from \ref{euivalenciaproj1} and its dual
since $\mathbb{F.P.P}(P)=\mathbb{P}_{1}$ and $\mathbb{F.C.I}(P)=\mathbb{I}_{1}$.\\
$(c)$ This follows from the fact  that $\mathrm{Hom}_{\mathrm{mod}(\mathcal{C})}(P,-):\mathrm{add}(P)\longrightarrow \mathrm{add}(R_{P})=\mathrm{proj}(R_{P})$ is an equivalence.\\
\end{proof}

\section{Extension over the endomorphism ring of a projective module}\label{sec:8}
In this section we will explore the relationship between injective coresolution in $\mathrm{mod}(\mathcal{C})$ and $\mathrm{mod}(R_{P})$. For each $X,Y\in \mathrm{mod}(\mathcal{C})$ consider the mapping
$$\rho_{X,Y}:\mathrm{Hom}_{\mathrm{mod}(\mathcal{C})}(X,Y)\longrightarrow \mathrm{Hom}_{R_{P}}\Big( \mathrm{Hom}_{\mathrm{mod}(\mathcal{C})}(P,X),\mathrm{Hom}_{\mathrm{mod}(\mathcal{C})}(P,Y)\Big)$$
defined as  $\rho_{X,Y}(f)=\mathrm{Hom}_{\mathrm{mod}(\mathcal{C})}(P,f)$ for all $f\in \mathrm{Hom}_{\mathrm{mod}(\mathcal{C})}(X,Y)$. It is easy to see that $\rho_{X,Y}$ is functorial in $X$ and $Y,$ then we have the following construction.

\begin{proposition}\label{mapshomolo}
Let $\mathcal{C}$ be a dualizing $R$-variety and $P=\mathrm{Hom}_{\mathcal{C}}(C,-)\in \mathrm{mod}(\mathcal{C})$. For each $X,Y\in \mathrm{mod}(\mathcal{C})$ and for all $i\geq 0$ we have canonical  morphisms
$$\Phi_{X,Y}^{i}:\mathrm{Ext}_{\mathrm{mod}(\mathcal{C})}^{i}(X,Y)\longrightarrow \mathrm{Ext}_{R_{P}}^{i}\Big( \mathrm{Hom}_{\mathrm{mod}(\mathcal{C})}(P,X),\mathrm{Hom}_{\mathrm{mod}(\mathcal{C})}(P,Y)\Big)$$
where $\Phi_{X,Y}^{0}=\rho_{X,Y}$.
 
\end{proposition}
\begin{proof}
It is straightforward, using injective coresolutions of $Y$ and $\mathrm{Hom}_{\mathrm{mod}(\mathcal{C})}(P,Y)$, and the comparison lemma.
\end{proof}
We give conditions in order to know when the morphisms $\Phi_{X,Y}^{i}$ are isomorphisms.

\begin{proposition}\label{Phiniso}
Let $\mathcal{C}$ be a dualizing $R$-variety and $P=\mathrm{Hom}_{\mathcal{C}}(C,-)\in \mathrm{mod}(\mathcal{C})$. The mapping 
$\Phi_{X,Y}^{n}:\mathrm{Ext}_{\mathrm{mod}(\mathcal{C})}^{n}(X,Y)\longrightarrow \mathrm{Ext}_{R_{P}}^{n}\Big( \mathrm{Hom}_{\mathrm{mod}(\mathcal{C})}(P,X),\mathrm{Hom}_{\mathrm{mod}(\mathcal{C})}(P,Y)\Big)$
above defined is an isomorphism for all $n\geq 0$, provided one of the three following conditions holds:
\begin{enumerate}
\item [(a)] $X\in \mathbb{P}_{i}$, $Y\in \mathbb{I}_{j}$ and $n\leq i+j$,

\item [(b)] $X\in \mathrm{mod}(\mathcal{C})$ and $Y\in \mathbb{I}_{n+1}$,

\item [(c)] $X\in \mathbb{P}_{n+1}$ and $Y\in \mathrm{mod}(\mathcal{C})$.
\end{enumerate}
\end{proposition}
\begin{proof}
The prove given in \cite[Theorem 3.2]{APG} works for this setting.
\end{proof}
We recall that the $\textbf{projective dimension}$ of an object $M$ in an abelian category $\mathcal{A}$ with enough projectives is the length of the shortest  projective resolution of $M$, and it is denoted by $\mathrm{pd}(M)$.
\begin{corollary}\label{dinprojigualdimproj}
Let $\mathcal{C}$ be a dualizing $R$-variety and $P=\mathrm{Hom}_{\mathcal{C}}(C,-)\in \mathrm{mod}(\mathcal{C})$. The following conditions hold.
\begin{enumerate}
\item [(a)] If $X\in \mathbb{P}_{\infty}$ then $\mathrm{pd}(X)=\mathrm{pd}_{R_{P}}\Big((P,X)\Big)$.

\item [(b)] If $X\in \mathbb{I}_{\infty}$ then $\mathrm{id}(X)=\mathrm{id}_{R_{P}}\Big((P,X)\Big)$.
\end{enumerate}

\end{corollary}
\begin{proof}
It follows from \ref{Phiniso}.
\end{proof}
The $\textbf{global dimension}$ of $\mathcal{A}$ is the supremum of the projective dimensions $\mathrm{p.d}(M)$ with $M\in \mathcal{A}$; and it is denoted by $\mathrm{gl.dim}(\mathcal{A})$.
\begin{proposition}\label{cotaglobaldim}
Let $\mathcal{C}$ be a dualizing $R$-variety and $P=\mathrm{Hom}_{\mathcal{C}}(C,-)\in \mathrm{mod}(\mathcal{C})$.
If $\mathbb{P}_{1}=\mathbb{P}_{\infty}$ or $\mathbb{I}_{1}=\mathbb{I}_{\infty}$ then $\mathrm{gl. dim}(\mathrm{mod}(R_{P}))\leq \mathrm{gl. dim}(\mathrm{mod}(\mathcal{C}))$.
\end{proposition}
\begin{proof}
We know that $\mathrm{Hom}_{\mathrm{mod}(\mathcal{C})}(P,-):\mathbb{P}_{1}\longrightarrow \mathrm{mod}(R_{P})$ and  $\mathrm{Hom}_{\mathrm{mod}(\mathcal{C})}(P,-):\mathbb{I}_{1}\longrightarrow \mathrm{mod}(R_{P})$ are equivalences (see \ref{lemainteres}(a)).  By \ref{dinprojigualdimproj}, for each $M\in \mathrm{mod}(R_{P})$ there exists $X\in \mathrm{mod}(\mathcal{C})$ such that $\mathrm{p.d}_{\mathrm{mod}(\mathcal{C})}(X)=\mathrm{p.d}_{\mathrm{mod}(R_{P})}(M)$. This implies that $\mathrm{gl. dim}(\mathrm{mod}(R_{P}))\leq \mathrm{gl. dim}(\mathrm{mod}(\mathcal{C}))$.
\end{proof}

\begin{proposition}\label{quasiRp}
Let $\mathcal{C}$ be a dualizing $R$-variety with cokernels and  consider $P=\mathrm{Hom}_{\mathcal{C}}(C,-)\in \mathrm{mod}(\mathcal{C})$. If $\mathbb{P}_{1}=\mathbb{P}_{\infty}$ or $\mathbb{I}_{1}=\mathbb{I}_{\infty}$ then $\mathrm{gl. dim}(R_{P})\leq 2$. In particular, $R_{P}$ is a quasi-hereditary algebra.
\end{proposition}
\begin{proof}
If $\mathcal{C}$ has cokernels we know that  $\mathrm{gl. dim}(\mathrm{mod}(\mathcal{C}))\leq 2$ (see \cite[Theorem 2.2(b)]{AusCoheren}. By \ref{cotaglobaldim}, we have that  $\mathrm{gl. dim}(R_{P})\leq 2$. By a well known result of Dlab-Ringel (see \cite[Theorem 2]{DR2}), we now that every artin algebra with global dimension less or equal to 2 is quasi-hereditary.
\end{proof}

If we work in the category $\mathrm{mod}(\mathcal{C}^{op})$, with the projective $P^{\ast}=\mathrm{Hom}_{\mathcal{C}}(-,C)$ and $J^{\ast}:=I_{0}\Big(\frac{P^{\ast}}{\mathrm{rad}(P^{\ast})}\Big)$, the injective envelope of  $\frac{P^{\ast}}{\mathrm{rad}(P^{\ast})}$ in  $\mathrm{mod}(\mathcal{C}^{op})$; we can define $\mathbb{P}_{k}^{\ast}$ and $\mathbb{I}_{k}^{\ast}$ in a similar way to definitions \ref{defiPk} and  \ref{defiIk}.

\begin{proposition}\label{DualPkIk}
Let $\mathcal{C}$ be a dualizing $R$-variety and $P=\mathrm{Hom}_{\mathcal{C}}(C,-)\in \mathrm{mod}(\mathcal{C})$.
\begin{enumerate}
\item [(a)] Then we have that $X\in \mathbb{P}_{k}$ if and only if $\mathbb{D}_{\mathcal{C}}(X)\in \mathbb{I}_{k}^{\ast}$.

\item [(b)] Then we have that $X\in \mathbb{I}_{k}$ if and only if $\mathbb{D}_{\mathcal{C}}(X)\in \mathbb{P}_{k}^{\ast}$.

\end{enumerate}
\end{proposition}
\begin{proof}
It follows from duality using \ref{PPast}, \ref{muchaseuiii} and  \ref{Pkcarect}.
\end{proof}

\begin{proposition}\label{IkisoP}
Let $\mathcal{C}$ be a dualizing $R$-variety and $P=\mathrm{Hom}_{\mathcal{C}}(C,-)\in \mathrm{mod}(\mathcal{C})$. Let $1\leq k\leq \infty$. Then 
\begin{enumerate}
\item [(a)] $Y\in \mathbb{I}_{k}$ if and only if
$\Phi_{X,Y}^{i}:\mathrm{Ext}^{i}_{\mathrm{mod}(\mathcal{C})}(X,Y)\longrightarrow \mathrm{Ext}^{i}_{R_{P}}\Big((P,X),(P,Y)\Big)$
is an isomorphism for all $0\leq i\leq k-1$ and for all $X\in \mathrm{mod}(\mathcal{C})$.

\item [(b)] $X\in \mathbb{P}_{k}$ if and only if $\Phi_{X,Y}^{i}:\mathrm{Ext}^{i}_{\mathrm{mod}(\mathcal{C})}(X,Y)\longrightarrow \mathrm{Ext}^{i}_{R_{P}}\Big((P,X),(P,Y)\Big)$
is an isomorphism for all $0\leq i\leq k-1$ and for all $Y\in \mathrm{mod}(\mathcal{C})$.
\end{enumerate}
\end{proposition}
\begin{proof}
$(a)$ $(\Rightarrow)$. It follows by \ref{Phiniso}(b).\\
$(\Leftarrow)$. Consider the ideal  $\mathcal{I}:=\mathrm{Tr}_{P}\mathcal{C}$ and the exact sequence in $\mathrm{mod}(\mathcal{C})$
$$(\ast):\xymatrix{0\ar[r]  & \mathcal{I}(C',-)\ar[r]^{u} & \mathrm{Hom}_{\mathcal{C}}(C',-)\ar[r] & \frac{\mathrm{Hom}_{\mathcal{C}}(C',-)}{\mathcal{I}(C',-)}\ar[r] & 0}$$
We have that $(P,u):(P,\mathcal{I}(C',-))\longrightarrow (P,\mathrm{Hom}_{\mathcal{C}}(C',-))$
is an isomorphism. Then 
$$\xymatrix{\mathrm{Ext}_{R_{P}}^{i}\Big((P,\mathrm{Hom}_{\mathcal{C}}(C',-)),(P,Y)\Big)\ar[r] & \mathrm{Ext}_{R_{P}}^{i}\Big((P,\mathcal{I}(C',-)),(P,Y)\Big)}$$ is an isomorphism for all $i\geq 0$. On the other hand, we have the commutative diagram
$$\xymatrix{\mathrm{Ext}_{\mathrm{mod}(\mathcal{C})}^{i}(\mathrm{Hom}_{\mathcal{C}}(C',-),Y)\ar[r]\ar[d]^{\Phi_{(C',-),Y}^{i}}& \mathrm{Ext}_{\mathrm{mod}(\mathcal{C})}^{i}(\mathcal{I}(C',-),Y)\ar[d]^{\Phi_{\mathcal{I}(C',-),Y}^{i}}\\
\mathrm{Ext}_{R_{P}}^{i}\Big((P,\mathrm{Hom}_{\mathcal{C}}(C',-)),(P,Y)\Big)\ar[r] & \mathrm{Ext}_{R_{P}}^{i}\Big((P,\mathcal{I}(C',-)),(P,Y)\Big)}$$
By hypothesis, the vertical mappings are isomorphisms for all $0\leq i\leq k-1$. Then we have the following isomorphism for all $0\leq i\leq k-1$
$$(\ast\ast):\xymatrix{\mathrm{Ext}_{\mathrm{mod}(\mathcal{C})}^{i}(\mathrm{Hom}_{\mathcal{C}}(C',-),Y)\ar[r] & \mathrm{Ext}_{\mathrm{mod}(\mathcal{C})}^{i}(\mathcal{I}(C',-),Y).}$$
For $i=0$, we get the isomorphism  $\Big(\mathrm{Hom}_{\mathcal{C}}(C',-),Y\Big)\rightarrow \Big(\mathcal{I}(C',-),Y\Big)$. Applying $\mathrm{Hom}_{\mathrm{mod}(\mathcal{C})}(-,Y)$ to the sequence $(\ast)$ we get that
$\Big(\frac{\mathrm{Hom}_{\mathcal{C}}(C',-)}{\mathcal{I}(C',-)},Y\Big)=0=\mathrm{Ext}^{1}_{\mathrm{mod}(\mathcal{C})}\Big(\frac{\mathrm{Hom}_{\mathcal{C}}(C',-)}{\mathcal{I}(C',-)},Y\Big).$
For $1\leq i\leq k-1$, using the isomorphism $(\ast\ast)$, we have that $\mathrm{Ext}_{\mathrm{mod}(\mathcal{C})}^{i}(\mathcal{I}(C',-),Y)=0$. Hence, we  get  that
$\mathrm{Ext}_{\mathrm{mod}(\mathcal{C})}^{i}\Big(\mathcal{I}(C',-),Y\Big)\rightarrow  \mathrm{Ext}_{\mathrm{mod}(\mathcal{C})}^{i+1}\Big(\frac{\mathrm{Hom}_{\mathcal{C}}(C',-)}{\mathcal{I}(C',-)},Y\Big)$ is an isomorphism for $i\geq 1$. Therefore, we obtain that    $\mathrm{Ext}_{\mathrm{mod}(\mathcal{C})}^{i}\Big(\frac{\mathrm{Hom}_{\mathcal{C}}(C',-)}{\mathcal{I}(C',-)},Y\Big)=0$ for $2\leq i\leq k$. Then, $\mathrm{Ext}_{\mathrm{mod}(\mathcal{C})}^{i}\Big(\frac{\mathrm{Hom}_{\mathcal{C}}(C',-)}{\mathcal{I}(C',-)},Y\Big)\!=\!
\mathrm{Ext}_{\mathrm{mod}(\mathcal{C})}^{i}\Big((\pi_{1})_{\ast}(\mathrm{Hom}_{\mathcal{C}/\mathcal{I}}(C',-)),Y\Big)$ $=0$ for $0\leq i\leq k$. Therefore, we have that $\mathrm{Ext}^{i}_{\mathrm{mod}(\mathcal{C})}((\pi_{1})_{\ast}(Q),Y)=0$ for all $Q\in\mathrm{mod}(\mathcal{C}/\mathcal{I})$ projective and $i=0,\dots, k$. By \ref{muchaseuiii}, we have that $Y\in \mathbb{I}_{k}$. The item $(b)$  follows from (a), by using duality and   \ref{DualPkIk}.
\end{proof}

\begin{proposition}\label{unacarusinmo}
Let $\mathcal{C}$ be a dualizing $R$-variety and $P=\mathrm{Hom}_{\mathcal{C}}(C,-)\in \mathrm{mod}(\mathcal{C})$. Let $X\in \mathbb{I}_{1}$ and $k\geq 1$. Then $X\in \mathbb{I}_{k}$ if and only if $\mathrm{Ext}^{i}_{R_{P}}(P^{\ast}(C'), (P,X))=0$ for all $1\leq i\leq k-1$ and for all $C'\in \mathcal{C}$.
\end{proposition}
\begin{proof}
$(\Rightarrow)$. Suppose that $X\in \mathbb{I}_{k}$. Then, wet that
$\mathrm{Ext}^{i}_{R_{P}}(P^{\ast}(C'), (P,X))$ $\simeq \mathrm{Ext}^{i}_{\mathrm{mod}(\mathcal{C})}\Big((C',-),X\Big)=0$
for all $1\leq i\leq k-1$ and for all $C'\in \mathcal{C}$.\\
$(\Leftarrow)$. Suppose that $\mathrm{Ext}^{i}_{R_{P}}(P^{\ast}(C'), (P,X))=0$ for all $1\leq i\leq k-1$ and for all $C'\in \mathcal{C}$. Let us see by induction on $k$ that $X\in \mathbb{I}_{k}$. For $k=1$, by hypothesis, we have that $X\in \mathbb{I}_{1}$. So let us check the first non trivial case, then suppose that $k=2$. Consider the ideal $\mathcal{I}=\mathrm{Tr}_{P}\mathcal{C}$ and $\pi_{1}:\mathcal{C}\longrightarrow \mathcal{C}/\mathcal{I}$ the projection. Since
$\mathcal{I}(C',-)=\mathrm{Tr}_{P}(\mathrm{Hom}_{\mathcal{C}}(C',-))$, we have that $\mathcal{I}(C',-)\in \mathbb{P}_{0}$. By \ref{Phiniso}(a), we have an isomorphism $\mathrm{Ext}^{1}_{\mathrm{mod}(\mathcal{C})}(\mathcal{I}(C',-),X)\longrightarrow \mathrm{Ext}_{R_{P}}^{1}\Big((P,\mathcal{I}(C',-)),(P,X)\Big)$. 
Since $\mathcal{I}(C',-)=\mathrm{Tr}_{P}(\mathrm{Hom}_{\mathcal{C}}(C',-))$, there exists an isomorphism $(P,\mathcal{I}(C',-))\simeq (P,\mathrm{Hom}_{\mathcal{C}}(C',-)).$
Then we have that
\begin{align*}
\mathrm{Ext}_{R_{P}}^{1}\Big(P^{\ast}(C'),(P,X)\Big) \simeq \mathrm{Ext}^{2}_{\mathrm{mod}(\mathcal{C})}\Big(\frac{\mathrm{Hom}_{\mathcal{C}}(C',-)}{\mathcal{I}(C',-)},X\Big)
\end{align*}
By hypothesis we have that $\mathrm{Ext}_{R_{P}}^{1}\Big(P^{\ast}(C'),(P,X)\Big) =0$, then we conclude that  $\mathrm{Ext}^{2}_{\mathrm{mod}(\mathcal{C})}\Big(\frac{\mathrm{Hom}_{\mathcal{C}}(C',-)}{\mathcal{I}(C',-)},X\Big)=0$.
This implies that  $\mathrm{Ext}^{2}_{\mathrm{mod}(\mathcal{C})}\Big((\pi_{1})_{\ast}(Q),X\Big)=0$ for all projective module $Q\in \mathrm{mod}(\mathcal{C}/\mathcal{I})$. Since $X\in \mathbb{I}_{1}$ (hypothesis), we have that $\mathrm{Ext}^{i}_{\mathrm{mod}(\mathcal{C})}\Big((\pi_{1})_{\ast}(Q),X\Big)=0$ for all projective module $Q\in \mathrm{mod}(\mathcal{C}/\mathcal{I})$ and $i=0,1$ (see \ref{muchaseuiii}.) By \ref{muchaseuiii}, we have that $X\in \mathbb{I}_{2}$.\\ 
Suppose that the theorem is true for $k-1$ with $k-1\geq 2$. Let $X\in \mathrm{mod}(\mathcal{C})$ such that $\mathrm{Ext}^{i}_{R_{P}}(P^{\ast}(C'), (P,X))=0$ for all $1\leq i\leq k-1$ and for all $C'\in \mathcal{C}$. In particular, we have that $\mathrm{Ext}^{1}_{R_{P}}(P^{\ast}(C'), (P,X))=0$  for all $C'\in \mathcal{C}$. Then, by the case $k=2$ just proved above, we have that $X\in \mathbb{I}_{2}$. Then, we have an exact sequence:
$(\star):0\rightarrow X\rightarrow I_{0}\rightarrow  L\rightarrow 0$ with
$I_{0}\in \mathrm{add}(J)$ and $L\in \mathbb{I}_{1}$. Applying $\mathrm{Hom}_{\mathrm{mod}(\mathcal{C})}(P,-)$ we get an exact sequence $0\rightarrow (P,X)\rightarrow (P, I_{0})\rightarrow (P,L)\rightarrow  0.$
Since $I_{0}\in \mathrm{add}(J)$, we have that $(P,I_{0})$ is an injective $R_{P}$-module (see \ref{lemainteres}(c)). Then, applying $\mathrm{Hom}_{R_{P}}(P^{\ast}(C'),-)$ to the last exact sequence, we have an isomorphism $\mathrm{Ext}^{i}_{R_{P}}(P^{\ast}(C'),(P,L))\simeq 	\mathrm{Ext}_{R_{P}}^{i+1}(P^{\ast}(C'),(P,X))$ for all $i\geq 1$. By hypothesis, we can conclude that  $\mathrm{Ext}^{i}_{R_{P}}(P^{\ast}(C'),(P,L))=0$ for all $i=1,\dots, k-2$. Since $L\in \mathbb{I}_{1}$, we can apply the induction to $L$. Then we conclude that $L\in \mathbb{I}_{k-1}$. From the exact sequence $(\star)$ we conclude that $X\in \mathbb{I}_{k}$. 
\end{proof}

\begin{proposition}\label{I1=Iinfi}
Let $\mathcal{C}$ be a dualizing $R$-variety and $P=\mathrm{Hom}_{\mathcal{C}}(C,-)\in \mathrm{mod}(\mathcal{C})$.
The following holds.
\begin{enumerate}
\item [(a)] $\mathbb{I}_{1}=\mathbb{I}_{\infty}$ if and only if $P^{\ast}(C')=\mathrm{Hom}_{\mathrm{mod}(\mathcal{C})}(P,\mathrm{Hom}_{\mathcal{C}}(C',-))$ is a projective $R_{P}$-module for all $C'\in \mathcal{C}$.

\item [(b)]  $\mathbb{P}_{1}=\mathbb{P}_{\infty}$ if and only if $P(C')\simeq \mathrm{Hom}_{\mathrm{mod}(\mathcal{C})}(\mathrm{Hom}_{\mathcal{C}}(C',-),P)$ is a projective $R_{P}^{op}$-module for all $C'\in \mathcal{C}$.
\end{enumerate}
\end{proposition}
\begin{proof}
$(\Rightarrow)$. Suppose that $\mathbb{I}_{1}=\mathbb{I}_{\infty}$. Consider $Z\in \mathrm{mod}(R_{P})$. Since $\mathrm{Hom}_{\mathrm{mod}(\mathcal{C})}(P,-)|_{\mathbb{I}_{1}}:\mathbb{I}_{1}\longrightarrow \mathrm{mod}(R_{P})$ is an equivalence (see  \ref{lemainteres}), there exists a $X\in \mathbb{I}_{1}$ such that 
$Z\simeq (P,X)$. Then, $\mathrm{Ext}_{R_{P}}^{i}(P^{\ast}(C'),Z)\simeq \mathrm{Ext}_{R_{P}}^{i}(P^{\ast}(C'),(P,X))=0$ for all $i\geq 1$ (since $X \in \mathbb{I}_{1}=\mathbb{I}_{\infty}$ and \ref{unacarusinmo}). This proves that $P^{\ast}(C')$ is a projective $R_{P}$-module. The converse is analogous and the proof of $(b)$ is similar.

\end{proof}

\begin{proposition}\label{especialsuc1}
Let $\mathcal{C}$ be a dualizing $R$-variety, $P=\mathrm{Hom}_{\mathcal{C}}(C,-)\in \mathrm{mod}(\mathcal{C})$ and $M\in \mathrm{mod}(\mathcal{C})$. 
\begin{enumerate}
\item [(a)]
There exists an exact sequence
$$\xymatrix{0\ar[r] & K_{1}\ar[r]^(.2){u} & P \otimes_{R_{P}} \mathrm{Hom}_{\mathrm{Mod}(\mathcal{C})}(P,M)\ar[r]^(.8){\epsilon_{M}'} &  M\ar[r]^{p} & K_{2}\ar[r] & 0}$$
with $M_{1}:=P \otimes_{R_{P}} \mathrm{Hom}_{\mathrm{Mod}(\mathcal{C})}(P,M)\in \mathbb{P}_{1}$ and $K_{1},K_{2}\in \mathrm{Ker}(\mathrm{Hom}(P,-))$.	

\item [(b)] $\mathrm{Hom}_{\mathrm{mod}(\mathcal{C})}(P,\epsilon_{M}'):\mathrm{Hom}_{\mathrm{mod}(\mathcal{C})}(P,M_{1})\longrightarrow \mathrm{Hom}_{\mathrm{mod}(\mathcal{C})}(P,M)$ is an isomorphism.

\end{enumerate}
\end{proposition}
\begin{proof}
This follows by considering, the counit of the following adjoint pair of functors $(P\otimes_{R_{P}}-,\mathrm{Hom}_{\mathrm{mod}(\mathcal{C})}(P,-))$.
\end{proof}
Now, we give other necessary and sufficient conditions for $\mathbb{I}_{1}$ to be
equal to $\mathbb{I}_{\infty}$.

\begin{proposition}\label{I1proje}
Let $\mathcal{C}$ be a dualizing $R$-variety and $P=\mathrm{Hom}_{\mathcal{C}}(C,-)\in \mathrm{mod}(\mathcal{C})$. The following are equivalent
\begin{enumerate}
\item [(a)] $\mathbb{I}_{1}=\mathbb{I}_{\infty}$

\item [(b)] $P \otimes_{R_{P}} \mathrm{Hom}_{\mathrm{Mod}(\mathcal{C})}(P,\mathrm{Hom}_{\mathcal{C}}(C',-))$ is a projective $\mathcal{C}$-module for all $C'\in \mathcal{C}$.
\end{enumerate}
\end{proposition}
\begin{proof}
$(b)\Rightarrow (a)$. Let $M=\mathrm{Hom}_{\mathcal{C}}(C',-)$ and we consider the module $M_{1}:=P \otimes_{R_{P}} \mathrm{Hom}_{\mathrm{Mod}(\mathcal{C})}(P,\mathrm{Hom}_{\mathcal{C}}(C',-))\in \mathrm{mod}(\mathcal{C})$. By \ref{especialsuc1},  there exists a morphism $\epsilon_{M}':M_{1}\longrightarrow M$ such that
$\mathrm{Hom}_{\mathrm{mod}(\mathcal{C})}(P,\epsilon_{M}'):\mathrm{Hom}_{\mathrm{mod}(\mathcal{C})}(P,M_{1})\longrightarrow \mathrm{Hom}_{\mathrm{mod}(\mathcal{C})}(P,M)$ is an isomorphism. Suppose that $M_{1}$ is a projective $\mathcal{C}$-module. Since $M_{1}\in \mathbb{P}_{1}$, we have that there exists an epimorphism $P^{n}\longrightarrow M_{1}$. 
Then we have that $M_{1}\in \mathrm{add}(P)$. By \ref{lemainteres}, we have that $\mathrm{Hom}_{\mathrm{mod}(\mathcal{C})}(P,M_{1})$ is a projective $R_{P}$-module. But $\mathrm{Hom}_{\mathrm{mod}(\mathcal{C})}(P,M)=P^{\ast}(C')$. Then we have that $P^{\ast}(C')$ is a  projective $R_{P}$-module for all $C'\in \mathcal{C}$. By \ref{I1=Iinfi}, we have that  $\mathbb{I}_{1}=\mathbb{I}_{\infty}$.  The other implication is similar.
\end{proof}

\begin{proposition}\label{idealproye}
Let $\mathcal{C}$ be a dualizing $R$-variety, $P=\mathrm{Hom}_{\mathcal{C}}(C,-)\in \mathrm{mod}(\mathcal{C})$ and  $\mathcal{I}=\mathrm{Tr}_{P}\mathcal{C}$.  The following are equivalent.
\begin{enumerate}
\item [(a)]  $\mathcal{I}$ is $2$-idempotent and $\mathbb{I}_{1}=\mathbb{I}_{\infty}$;

\item [(b)] $\mathcal{I}(C',-)$ is a projective $\mathcal{C}$-module for all $C'\in \mathcal{C}$.
\end{enumerate}
\end{proposition}
\begin{proof}
$(a)\Rightarrow (b)$. Let $C'\in \mathcal{C}$, by \ref{otracaidemo}, we have that $\mathcal{I}(C',-)\in \mathbb{P}_{1}$ for all $C'\in \mathcal{C}$. Then, by \ref{lemainteres} we have that $\mathcal{I}(C',-)= P \otimes_{R_{P}} \mathrm{Hom}_{\mathrm{mod}(\mathcal{C})}(P,\mathcal{I}(C',-))$.\\ But, since $\mathcal{I}(C',-)=\mathrm{Tr}_{P}(\mathrm{Hom}_{\mathcal{C}}(C',-))$, we have the following isomorphism of $R_{P}$-modules $\mathrm{Hom}_{\mathrm{mod}(\mathcal{C})}(P,\mathcal{I}(C',-))=\mathrm{Hom}_{\mathrm{mod}(\mathcal{C})}(P,\mathrm{Hom}_{\mathcal{C}}(C',-))$.
Then $P \otimes_{R_{P}} \mathrm{Hom}_{\mathrm{mod}(\mathcal{C})}(P,\mathrm{Hom}_{\mathcal{C}}(C',-))=\mathcal{I}(C',-).$ Since $\mathbb{I}_{1}=\mathbb{I}_{\infty}$ by \ref{I1proje}, we conclude that $\mathcal{I}(C',-)$ is projective. The other implication is similar.
\end{proof}

\begin{proposition}\label{quasiher2}
Let $\mathcal{C}$ be a dualizing $R$-variety with cokernels and  consider $P=\mathrm{Hom}_{\mathcal{C}}(C,-)\in \mathrm{mod}(\mathcal{C})$. If  $\mathcal{I}(C',-)$ is projective for all $C'\in \mathcal{C}$. Then we have that $R_{P}$ is quasi-hereditary.
\end{proposition}
\begin{proof}
It follows from \ref{quasiRp} and \ref{idealproye}.
\end{proof}

\begin{proposition}\label{fullcatdev}
Let $\mathcal{C}$ be a dualizing $R$-variety,  $P=\mathrm{Hom}_{\mathcal{C}}(C,-)\in \mathrm{mod}(\mathcal{C})$ and $\mathcal{I}=\mathrm{Tr}_{P}\mathcal{C}$. Consider the functor $\pi_{\ast}:\mathrm{mod}(\mathcal{C}/\mathcal{I})\longrightarrow \mathrm{mod}(\mathcal{C}).$ If  $\mathcal{I}(C',-)$ is projective for all $C'\in \mathcal{C}$, we have a full  embedding
$$\mathrm{D}^{b}(\pi_{\ast}):\mathrm{D}^{b}(\mathrm{mod}(\mathcal{C}/\mathcal{I}))\longrightarrow \mathrm{D}^{b}(\mathrm{mod}(\mathcal{C}))$$
between its  bounded derived categories.
\end{proposition}
\begin{proof}
Since  $\mathcal{I}=\mathrm{Tr}_{P}\mathcal{C}$,  we get an epimorphism  $P^{n}\longrightarrow \mathcal{I}(C',-)$ for each $C'\in \mathcal{C}$.
Since $\mathcal{I}(C',-)$ is projective for all $C'\in \mathcal{C}$, we have that $\mathcal{I}(C',-)\in \mathrm{add}(P)\subseteq \mathbb{P}_{\infty}$. Then, by \ref{caractstronidemp}, we have that $\mathcal{I}$ is strongly idempotent. That is, $$\varphi^{i}_{F,\pi_{\ast}(F')}:\mathrm{Ext}^{i}_{\mathrm{mod}(\mathcal{C}/I)}(F,F')\longrightarrow \mathrm{Ext}^{i}_{\mathrm{mod}(\mathcal{C})}(\pi_{\ast}(F),\pi_{\ast}(F'))$$ is an isomorphism for all $F,F'\in \mathrm{mod}(\mathcal{C}/\mathcal{I})$ and for all $0\leq i < \infty$
(see definition \ref{kidemcat}).  By \cite[Theorem 4.3]{GeigleLen}, we have the required full embedding.
\end{proof}

\section{Some examples}\label{sec:9}

Consider an algebraically closed field $K$  and the infinite quiver 
$$Q:\xymatrix{1\ar[r]^{\alpha_{1}}& 2\ar[r]^{\alpha_{2}}  &\cdots \ar[r]  &  k\ar[r]^{\alpha_{k}} & k+1\ar[r]  & \cdots\ar[r] & \cdots }$$
Consider $\mathcal{C}:=KQ/\langle \rho\rangle$, the path category associated to $Q$ where $\rho$ is given by the relations $\alpha_{i+1}\alpha_{i}=0$ for all $i\geq 1$. By construction, we have that $\mathcal{C}$ is a Hom-finite $K$-category (for more details see for example \cite[Proposition 6.6]{MOSS}).\\ 
It is well known that the category of representations $\mathrm{Rep}(Q,\rho)$ is equivalent to $\mathrm{Mod}(\mathcal{C})$. In this case, the projective and simple representations  associated to the vertex $k$ are of the form
$$P_{k}:\xymatrix{k\ar[d] \\
k+1,}\quad \quad \quad \quad \quad S_{k}:\, k$$
\begin{enumerate}
\item [(1a)] Consider $P=\oplus_{j=1}^{k}P_{j}$ and $\mathcal{I}=\mathrm{Tr}_{P}\mathcal{C}$. In this case, we have that $\frac{\mathrm{Hom}_{\mathcal{C}}(i,-)}{\mathcal{I}(i,-)}=0$  for $1\leq i\leq k$ and $\frac{\mathrm{Hom}_{\mathcal{C}}(i,-)}{\mathcal{I}(i,-)}\simeq \frac{P_{i}}{\mathrm{Tr}_{P}(P_{i})}=P_{i}$ for $i\geq k+1$. Then for all $j\geq 1$ we have that 
$\mathrm{Ext}^{j}_{\mathrm{Mod}(\mathcal{C})}(\frac{\mathrm{Hom}_{\mathcal{C}}(i,-)}{\mathcal{I}(i,-)},F'\circ \pi)=0$ for all $i\in \mathcal{C}=FQ/\langle \rho\rangle$ and for all $F'\in \mathrm{Mod}(\mathcal{C}/\mathcal{I})$. By \ref{caractidem}, we have that $\mathcal{I}$ is strongly idempotent.\\

\item [(1b)] Consider the projective $P:=\oplus_{j=2}^{k}P_{j}$ and let $\mathcal{I}:=\mathrm{Tr}_{P}\mathcal{C}$. We assert that $\mathrm{Tr}_{P}\mathcal{C}$ is $(k-1)$-idempotent.\\
Indeed, firstly we have that $\frac{\mathrm{Hom}_{\mathcal{C}}(1,-)}{\mathcal{I}(1,-)} \simeq \frac{P_{1}}{\mathrm{Tr}_{P}(P_{1})}\simeq S_{1}$,
where $S_{1}$ is the simple representation associated to the vertex $1$. Moreover, we have that
$\mathrm{Tr}_{P}(P_{i})=P_{i}$ for $2\leq i\leq k$. Then we have that
$\frac{\mathrm{Hom}_{\mathcal{C}}(i,-)}{\mathcal{I}(i,-)}\simeq \frac{P_{i}}{\mathrm{Tr}_{P}(P_{i})}
=0$ for $2\leq i\leq k$. We also have that $\mathrm{Tr}_{P}(P_{i})=0$ for all $i\geq k+1$; and hence $\frac{\mathrm{Hom}_{\mathcal{C}}(i,-)}{\mathcal{I}(i,-)}\simeq \frac{P_{i}}{\mathrm{Tr}_{P}(P_{i})}=P_{i}$ for $i\geq k+1$.\\
It is enough to show that
$\mathrm{Ext}^{j}_{\mathrm{Mod}(\mathcal{C})}(S_{1},F'\circ \pi)=0$ for all $0\leq j \leq k-1$ and $\forall F'\in \mathrm{Mod}(\mathcal{C}/\mathcal{I})$ (see \ref{caractidem}).
We have a projective resolution of $S_{1}$ 
$$\xymatrix{\cdots\ar[r] & P_{k+1}\ar[r] & P_{k}\ar[r] & \cdots\ar[r] & P_{1}\ar[r]  & S_{1}\ar[r]  & 0}$$
where each $P_{i}$ is the projective associated to the vertex $i$. Using this resolution we can see that $\mathrm{Ext}_{\mathrm{Mod}(\mathcal{C})}^{j}(S_{1},F'\circ \pi)=0$ for all $1\leq j\leq k-1$. Therefore, we have that $\mathcal{I}$ is $(k-1)$-idempotent.\\

\item [(1c)] The ideal given in item $(b)$ is not $k$-idempotent. Indeed, we have that $S_{k}=\Omega^{k-1}(S_{1})$ (the $k-1$ syzygy of $S_{1}$). By the shifting lemma, we have that $\mathrm{Ext}^{k}_{\mathrm{Mod}(\mathcal{C})}(S_{1},S_{k+1})\simeq \mathrm{Ext}^{1}_{\mathrm{Mod}(\mathcal{C})}(S_{k},S_{k+1})$. We have the exact sequence  $0\rightarrow S_{k+1}\rightarrow P_{k}\rightarrow S_{k}\rightarrow 0$  which does not split. We conclude that $\mathrm{Ext}^{k}_{\mathrm{Mod}(\mathcal{C})}(S_{1},S_{k+1})\neq 0$. Then, $\mathcal{I}$ is not $k$-idempotent.\\

\item [(2)] Let $\mathcal{I}$ be a heredity ideal in $\mathcal{C}$, according to definition 3.2 in \cite{Martin1}. Then we have that $\mathcal{I}(C,-)$ is a projective $\mathcal{C}$-module for all $C\in \mathcal{C}$ and $\mathcal{I}$ is idempotent. Then by \ref{caractidem}, we have that
$\mathrm{Ext}^{1}_{\mathrm{Mod}(\mathcal{C})}(\frac{\mathrm{Hom}_{\mathcal{C}}(C,-)}{\mathcal{I}(C,-)}, F'\circ \pi)=0$ for all $F'\in \mathrm{Mod}(\mathcal{C}/\mathcal{I})$ and for all $C\in \mathcal{C}$.\\
Now, since the projective dimension of each $\frac{\mathrm{Hom}_{\mathcal{C}}(C,-)}{\mathcal{I}(C,-)}$ is less or equal to $1$, we have that $\mathrm{Ext}^{j}_{\mathrm{Mod}(\mathcal{C})}(\frac{\mathrm{Hom}_{\mathcal{C}}(C,-)}{\mathcal{I}(C,-)}, F'\circ \pi)=0$ for all $F'\in \mathrm{Mod}(\mathcal{C}/\mathcal{I})$ and for all $C\in \mathcal{C}$. Then, by \ref{caractidem}, we have that $\mathcal{I}$ is strongly idempotent.
\end{enumerate}

\section*{Acknowledgements}
{\small  This work presents results obtained during the first author's doctoral studies, carried out with a CONACYT grant (see \cite{phdtesisLG}).
The authors are grateful to the project PAPIIT-Universidad Nacional Aut\'onoma de M\'exico IN100520. }

\bibliographystyle{plain}
\bibliography{HTDV}

\footnotesize

\vskip3mm \noindent  $^{1}$Luis Gabriel Rodr\'iguez Vald\'es:\\ Departamento de Matem\'aticas, Facultad de Ciencias, Universidad Nacional Aut\'onoma de M\'exico\\
Circuito Exterior, Ciudad Universitaria,
C.P. 04510, Ciudad de M\'exico, MEXICO.\\ {\tt luisgabriel@ciencias.unam.mx }

\vskip3mm \noindent $^{2}$Martha Lizbeth Shaid Sandoval Miranda:\\ Departamento de Matem\'aticas, Universidad Aut\'onoma Metropolitana Unidad Iztapalapa\\
Av. San Rafael Atlixco 186, Col. Vicentina Iztapalapa 09340, M\'exico, Ciudad de M\'exico.\\ {\tt marlisha@xanum.uam.mx, marlisha@ciencias.unam.mx}

\vskip3mm \noindent $^{3}$Valente Santiago Vargas:\\ Departamento de Matem\'aticas, Facultad de Ciencias, Universidad Nacional Aut\'onoma de M\'exico\\
Circuito Exterior, Ciudad Universitaria,
C.P. 04510, Ciudad de M\'exico, MEXICO.\\ {\tt valente.santiago.v@gmail.com}

\end{document}